\documentclass[11pt,reqno]{amsart}
\usepackage{amssymb,latexsym,graphicx,amscd,upgreek}
\usepackage{lscape}
\usepackage[all]{xy}
\usepackage{color}
\usepackage{epic,eepic}
\usepackage{xspace}
\usepackage{mathrsfs}
\usepackage{setspace}
\usepackage{cases}
\usepackage{bbm}  
\usepackage{url}
\newtheorem{Thm}{Theorem}[section]
\newtheorem{Lem}[Thm]{Lemma}
\newtheorem{Cor}[Thm]{Corollary}
\newtheorem{Prop}[Thm]{Proposition}

\theoremstyle{definition}
\newtheorem{Rem}[Thm]{Remark}
\newtheorem{Def}[Thm]{Definition}

\newcommand{\Z}{\mathbb{Z}}
\newcommand{\Q}{\mathbb{Q}}
\newcommand{\N}{\mathbb{N}}
\newcommand{\C}{\mathbb{C}}

\newcommand{\df}{\colon}
\newcommand{\id}{\operatorname{id}}

\newcommand{\cC}{{\mathcal C}}

\newcommand{\cF}{{\mathcal F}}
\newcommand{\cG}{{\mathcal G}}
\newcommand{\cH}{{\mathcal H}}

\newcommand{\cM}{{\mathcal M}}

\newcommand{\cP}{{\mathcal P}}

\newcommand{\cS}{{\mathcal S}}
\newcommand{\cT}{{\mathcal T}}
\newcommand{\cU}{{\mathcal U}}

\newcommand{\cY}{{\mathcal Y}}

\newcommand{\g}{\mathfrak{g}}
\newcommand{\m}{\mathfrak{m}}

\newcommand{\bi}{{\mathbf i}}
\newcommand{\bj}{{\mathbf j}}

\newcommand{\GG}{\Gamma}
\newcommand{\alp}{\alpha}   
\newcommand{\eps}{\epsilon}

\newcommand{\vep}{\varepsilon}
\newcommand{\la}{\lambda}

\newcommand{\ad}{\operatorname{ad}}
\newcommand{\rk}{\operatorname{rank}}

\newcommand{\md}{\operatorname{mod}}

\newcommand{\rep}{\operatorname{rep}}
\newcommand{\Rep}{\operatorname{Rep}}

\newcommand{\add}{\operatorname{add}}

\newcommand{\pdim}{\operatorname{proj.dim}}
\newcommand{\idim}{\operatorname{inj.dim}}

\newcommand{\dimv}{\underline{\dim}}
\newcommand{\rkv}{\underline{\rk}}

\newcommand{\sgn}{\operatorname{sgn}}

\newcommand{\soc}{\operatorname{soc}}
\newcommand{\rad}{\operatorname{rad}}
\newcommand{\tp}{\operatorname{top}}
\newcommand{\sub}{\operatorname{sub}}
\newcommand{\fac}{\operatorname{fac}}
\newcommand{\proj}{\operatorname{proj}}
\newcommand{\inj}{\operatorname{inj}}

\newcommand{\Hom}{\operatorname{Hom}}
\newcommand{\Ext}{\operatorname{Ext}}
\newcommand{\End}{\operatorname{End}}
\newcommand{\Du}{\operatorname{D}}

\newcommand{\Ima}{\operatorname{Im}}
\newcommand{\Ker}{\operatorname{Ker}}
\newcommand{\Coker}{\operatorname{Cok}}

\newcommand{\ov}{\overline}

\newcommand{\bil}[1]{\langle #1\rangle}

\newcommand{\op}{{\rm op}}

\newcommand{\out}{\rm out}
\newcommand{\inn}{\rm in}
\newcommand{\Res}{\operatorname{Res}}

\newcommand{\bsm}{\begin{smallmatrix}}
\newcommand{\esm}{\end{smallmatrix}}

\newcommand{\bbsm}{\left[\begin{smallmatrix}}
\newcommand{\besm}{\end{smallmatrix}\right]}

\newcommand{\bbm}{\begin{matrix}}
\newcommand{\ebm}{\end{matrix}}

\newcommand{\diag}{\operatorname{diag}}

\newcommand{\wti}{\widetilde}

\newcommand{\Tr}{\operatorname{tr}}

\newcommand{\vp}{{\rm l.f.}}
\newcommand{\can}{{\rm can}}

\newcommand{\ra}{\rightarrow}

\newcommand{\real}{{\rm re}}
\newcommand{\imag}{{\rm im}}

\newcommand{\bbo}{\mathbbm{1}}

\begin{document}

\date{18.10.2016}

\title{Quivers with relations for symmetrizable Cartan matrices I: Foundations}

\author{Christof Gei{\ss}}
\address{Christof Gei{\ss}\newline
Instituto de Matem\'aticas\newline
Universidad Nacional Aut{\'o}noma de M{\'e}xico\newline
Ciudad Universitaria\newline
04510 M{\'e}xico D.F.\newline
M{\'e}xico}
\email{christof@math.unam.mx}

\author{Bernard Leclerc}
\address{Bernard Leclerc\newline
Normandie Univ, France\newline
Unicaen, Lmno, F-14032 Caen France\newline
Cnrs, Umr 6139, F-14032 Caen France\newline
Institut Universitaire de France}
\email{bernard.leclerc@unicaen.fr}

\author{Jan Schr\"oer}
\address{Jan Schr\"oer\newline
Mathematisches Institut\newline
Universit\"at Bonn\newline
Endenicher Allee 60\newline
53115 Bonn\newline
Germany}
\email{schroer@math.uni-bonn.de}



\begin{abstract}
We introduce and study a class of 
Iwanaga-Gorenstein algebras defined via quivers with relations
associated with symmetrizable Cartan matrices.
These algebras generalize the path algebras of quivers  
associated with symmetric Cartan matrices.
We also define a corresponding class of generalized preprojective 
algebras.
For these two classes of algebras we obtain generalizations of classical results of Gabriel, Dlab-Ringel, and Gelfand-Ponomarev.
In particular, we obtain new representation theoretic realizations of all finite root systems without any assumption on the ground field.
\end{abstract}

\maketitle

\setcounter{tocdepth}{1}
\numberwithin{equation}{section}
\tableofcontents

\parskip2mm


\section{Introduction and main results}\label{sec1}


\subsection{Quivers, Coxeter functors and preprojective 
algebras}
\label{intro:1}
Let $Q$ be a finite connected acyclic quiver, and let
$H = KQ$ be the path algebra of $Q$ with coefficients in a field $K$.
The following five results, all proved in the 1970's, form an essential
part of the foundations of modern
representation theory of 
finite-dimensional algebras.
\begin{itemize}

\item[(1)]
Gabriel's Theorem \cite{Ga1}: 
The quiver $Q$ is representation-finite if and only if $Q$ is a Dynkin
quiver of type $A_n$, $D_n$, $E_6$, $E_7$, $E_8$.
In this case, there is a bijection between the isomorphism classes
of indecomposable representations of $Q$ and the set 
of positive roots of the corresponding simple complex 
Lie algebra.\vspace{0.2cm}

\item[(2)]
Bernstein, Gelfand and Ponomarev's \cite{BGP}
discovery of Coxeter functors 
\[
C^\pm(-) = F_{i_n}^\pm \circ \cdots \circ F_{i_1}^\pm:
\rep(H) \to \rep(H),
\] 
which are defined as compositions of reflection functors.
They lead to a more
conceptual proof of Gabriel's Theorem. 
Applied to the indecomposable projective (resp. injective) representations
they yield a family of indecomposable representations, called 
\emph{preprojective} (resp. \emph{preinjective}) representations.\vspace{0.2cm}

\item[(3)]
Gabriel's Theorem \cite{Ga3} saying that there are functorial isomorphisms
$TC^\pm(-) \cong \tau^\pm(-)$,
where $T$ is a twist functor and $\tau(-)$ is the Auslander-Reiten
translation (see also the comment below on an earlier result
by Brenner and Butler).\vspace{0.2cm}

\item[(4)]
Auslander, Platzeck and Reiten's Theorem \cite{APR} saying that
the functors $F_k^+$ and $\Hom_H(T,-)$,
where $F_k^+$ is a BGP-reflection functor and $T$ is the associated
APR-tilting module, are equivalent. 
This result can be  
considered as the starting point of tilting theory.\vspace{0.2cm}

\item[(5)]
Gelfand and Ponomarev's \cite{GP}
discovery of the preprojective algebra
$\Pi(Q)$ of a quiver, and their result that 
$\Pi(Q)$, seen as a module over $H$, is isomorphic
to the direct sum of all preprojective $H$-modules, hence the name
\emph{preprojective algebra}.
The algebra $\Pi(Q)$ 
is isomorphic
to the tensor algebra $T_H(\Ext_H^1(\Du(H),H))$, where $\Du$ denotes the duality
with respect to the base field $K$, see \cite{BGL,CB2,Ri3}.

\end{itemize}
The above results hold for arbitrary ground fields $K$.
At the price of
quite strong assumptions on $K$ they 
were generalized 
from quivers to the more general setup of modulated graphs.
(One needs to assume the existence of finite field extensions
of $K$ with prescribed degrees.)
For the finite type situation,
this extended the theory from the simply laced root systems of
types $A_n$, $D_n$, $E_6$, $E_7$ and $E_8$ to the non-simply
laced root systems $B_n$, $C_n$, $F_4$ and $G_2$.
The definition of a modulated graph (also called \emph{species}) and of its representations
is due to Gabriel \cite{Ga2}.
The theory itself has been developed by Dlab and Ringel, who
generalized (1), (2) and (5) to modulated graphs \cite{D,DR1,DR2,DR3,Ri1}.
Brenner and Butler \cite{BB} proved an earlier result closely related to (3),
which is also valid for modulated graphs.
(They don't treat~$C^\pm$ as endofunctors, and the twist automorphism
$T$ does not appear.)

\subsection{Hereditary, selfinjective and Iwanaga-Gorenstein algebras}
\label{intro:2}
In this section, by an \emph{algebra} we mean a 
finite-dimensional $K$-algebra.

An algebra $A$ is \emph{hereditary} if all $A$-modules have
projective and injective dimension at most $1$.
The representation theory of quivers and species corresponds to the
representation theory of finite-dimensional hereditary algebras.

An algebra $A$ is \emph{selfinjective}
if the classes of projective and injective $A$-modules coincide.
This implies that all modules (except the projective-injectives) have 
infinite projective and injective dimension.
Despite being opposite homological extremes, hereditary and selfinjective 
algebras are often intimately linked.
For example the path algebra $KQ$ is always hereditary,
and in contrast, if $Q$ is a Dynkin quiver, then the closely related 
preprojective algebra $\Pi(Q)$ is selfinjective. 
Also, the classification of representation-finite selfinjective algebras
shows striking similarities to the classification of representation-finite
hereditary algebras.

An algebra $A$ is $m$-\emph{Iwanaga-Gorenstein} if
\[
\idim(A) \le m\quad \mbox{and} \quad \pdim(\Du(A)) \le m.
\]
These algebras were first studied by Iwanaga \cite{I1,I2}.
In this case, \cite[Lemma~6.9]{AR} implies that
$\idim(A) = \pdim(\Du(A))$, 
and by \cite[Theorem~5]{I2} for any $A$-module $M$ the following are equivalent:
\begin{itemize}

\item[(i)]
$\pdim(M) \le m$;

\item[(ii)]
$\idim(M) \le m$;

\item[(iii)]
$\pdim(M) < \infty$;

\item[(iv)]
$\idim(M) < \infty$.

\end{itemize}
Note that with this definition a given
algebra can be $m$-Iwanaga-Gorenstein for different values of $m$.
An algebra is selfinjective if and only if it is $0$-Iwanaga-Gorenstein.
All hereditary algebras and also all selfinjective algebras are 
$1$-Iwanaga-Gorenstein. 
Now let $A$ be a $1$-Iwanaga-Gorenstein algebra.
Then 
there are two subcategories of the category $\rep(A)$ of finite-dimensional $A$-modules which are of interest:
\begin{itemize}

\item[(a)]
The subcategory 
\[
\cH(A) := \{ M \in \rep(A) \mid
\pdim(M) \le 1 \text{ and } \idim(M) \le 1 \}.
\]

\item[(b)]
The subcategory 
\[
{\cG\cP}(A) := \{ M \in \rep(A) \mid \Ext_A^1(M,A) = 0 \}
\]
of \emph{Gorenstein-projective} modules.

\end{itemize}
Let $\cP(A)$ be the subcategory of projective $A$-modules.
We have
\[
\cP(A) = \cH(A) \cap \cG\cP(A).
\]
For each $M \in \rep(A)$ there are short exact sequences
\[
0 \to H_M \to G_M \to M \to 0
\]
and 
\[
0 \to M \to H^M \to G^M \to 0
\]
with $H_M,H^M \in \cH(A)$ and $G_M,G^M \in \cG\cP(A)$, see
\cite[8.1]{AB}.

The category $\cH(A)$ carries the homological features of module
categories of hereditary algebras, whereas $\cG\cP(A)$ is a Frobenius category,
thus displaying the homological features of module categories of
selfinjective algebras.
We have
$\cG\cP(A) = \cP(A)$ and $\cH(A) = \rep(A)$ if and only if
$A$ is hereditary,
and in the other extreme we have
$\cG\cP(A) = \rep(A)$ and $\cH(A) = \cP(A)$ if and only if $A$
is selfinjective.

The stable category of $\cG\cP(A)$ is a triangulated category, which is
triangle equivalent to the \emph{singularity category}
\[
D_{\rm sg}(A) := D^b(A)/K^b(\proj(A))
\]
defined and studied by Buchweitz \cite{Bu}, see also \cite{O}.
(Here $D^b(A)$ denotes the derived category of bounded complexes
of finite-dimensional $A$-modules,
and $K^b(\proj(A))$ is the homotopy category of bounded complexes of
finite-dimensional projective $A$-modules.)
It follows that $D_{\rm sg}(A) = 0$ if and only if $A$ is hereditary.

Thus
the class of $1$-Iwanaga-Gorenstein algebras can be seen as an intermediary 
class sitting between the hereditary and the selfinjective
algebras, and the singularity category $D_{\rm sg}(A)$ can be considered
as a measure of how far $A$ is from being hereditary.

\subsection{$1$-Iwanaga-Gorenstein algebras attached to Cartan 
matrices}\label{intro:3}
To each symmetrizable generalized Cartan matrix $C$ and an
orientation $\Omega$ of $C$ we attach an
infinite series of $1$-Iwanaga-Gorenstein algebras $H = H(C,D,\Omega)$
indexed by the different symmetrizers $D$ of $C$.
These algebras are defined by quivers with relations over an 
arbitrary field $K$.

If $C$ is symmetric and connected, then $(C,\Omega)$ corresponds
to a connected acyclic quiver $Q$, and the series
of algebras $H$ consists of the
algebras of the form
\[
A_m \otimes_K KQ,\qquad (m\ge 1),
\] 
where $A_m := K[X]/(X^m)$ is a truncated polynomial ring.
Representations of such algebras are nothing else than representations of 
$Q$ over the ground rings $A_m$.

In the general case of a symmetrizable matrix $C$, the algebras $H$ can be 
identified with tensor algebras of modulations of the oriented valued graph 
$\Gamma$ corresponding to $(C,\Omega)$.
However, in contrast to 
the classical notion of a modulation, 
the rings attached to vertices 
of $\Gamma$ are truncated polynomial rings instead of division rings.

We also introduce a series of algebras $\Pi = \Pi(C,D)$, again
defined by quivers with relations, which can be regarded as preprojective
algebras of quivers (or more generally of modulated graphs) over truncated 
polynomial rings.

We show that analogues of all five results mentioned in Section~\ref{intro:1}
also hold for our algebras $H$ and $\Pi$.
However certain definitions must be adapted. 
For example, we say that $H$ has \emph{finite $\tau$-representation type} if 
its Auslander-Reiten quiver has only finitely many $\tau$-orbits consisting
entirely of modules of finite homological dimension.
The analogue of (1) states that $H$ has finite $\tau$-representation type if 
and only if $C$ is of Dynkin type.
In this case, there is
a bijection between the isomorphism classes of 
indecomposable modules sitting on these $\tau$-orbits and the  positive roots
of the simple Lie algebra associated with $C$.
So for each Cartan matrix $C$ of Dynkin type we get an infinite family
of new representation theoretic incarnations of the 
root system of $C$.
Let us stress that even in the non-simply laced case these incarnations
are defined without any assumption on the ground field $K$.  
To prove this theorem, we define analogues of the reflection functors and 
Coxeter functors of (2), and we give an analogue of Gabriel's Theorem (3) for 
the subcategory of $H$-modules of finite homological dimension. This yields 
alternative descriptions of the preprojective algebra $\Pi$ similar to (5).
We also obtain an analogue of (4) describing the reflection functors in terms 
of certain tilting $H$-modules.

In the rest of this section we give precise definitions of the algebras 
$H$ and $\Pi$, and we
state our main results in more detail. 
We then point out 
previous appearances of some of the algebras $H$ and $\Pi$ in 
the literature. 

\subsection{Definition of $H$ and $\Pi$} \label{def:HPi}
A matrix $C = (c_{ij}) \in M_n(\Z)$ is a  
\emph{symmetrizable generalized Cartan matrix}
provided the following hold:
\begin{itemize}

\item[(C1)]
$c_{ii} = 2$ for all $i$;

\item[(C2)] 
$c_{ij} \le 0$ for all $i \not= j$;

\item[(C3)]
$c_{ij} \not= 0$ if and only if $c_{ji} \not= 0$.

\item[(C4)]
There is a diagonal integer matrix $D = \diag(c_1,\ldots,c_n)$ with
$c_i \ge 1$ for all $i$ such that
$DC$ is symmetric. 

\end{itemize}
The matrix $D$ appearing in (C4) is called a \emph{symmetrizer} of $C$.
The symmetrizer $D$ is \emph{minimal} if $c_1 + \cdots + c_n$ is
minimal.
From now on, by a \emph{Cartan matrix} we always mean a
symmetrizable generalized Cartan matrix.
In this case, define for all $c_{ij} < 0$
\[
g_{ij} := |\gcd(c_{ij},c_{ji})|,\qquad
f_{ij} := |c_{ij}|/g_{ij},\qquad
k_{ij} := \gcd(c_i,c_j).
\]
Note that we have
\[
g_{ij} = g_{ji}, \qquad
k_{ij} = k_{ji}, \qquad
c_i = k_{ij}f_{ji}.
\]

Let $C = (c_{ij}) \in M_n(\Z)$ be a Cartan matrix. 
An \emph{orientation of} $C$ is a subset 
$\Omega \subset  \{ 1,2,\ldots,n \} \times \{ 1,2,\ldots,n \}$
such that the following hold:
\begin{itemize}

\item[(i)]
$\{ (i,j),(j,i) \} \cap \Omega \not= \varnothing$
if and only if $c_{ij}<0$;

\item[(ii)]
For each sequence $((i_1,i_2),(i_2,i_3),\ldots,(i_t,i_{t+1}))$ with
$t \ge 1$ and $(i_s,i_{s+1}) \in \Omega$ for all $1 \le s \le t$ we have
$i_1 \not= i_{t+1}$.

\end{itemize}
For an orientation $\Omega$ of $C$ let
$Q := Q(C,\Omega) := (Q_0,Q_1,s,t)$ be the quiver with the
set of vertices $Q_0 := \{ 1,\ldots, n\}$ and 
with the set of arrows 
\[
Q_1 := \{ \alpha_{ij}^{(g)}\df j \to i \mid (i,j) \in \Omega, 1 \le g \le g_{ij} \}
\cup \{ \vep_i\df i \to i \mid 1 \le i \le n \}.
\]
(Thus we have $s(\alpha_{ij}^{(g)}) = j$ and $t(\alpha_{ij}^{(g)}) = i$ and
$s(\vep_i) = t(\vep_i) = i$, where $s(a)$ and $t(a)$ denote the
starting and terminal vertex of an arrow $a$, respectively.)
If $g_{ij}=1$, we also write $\alpha_{ij}$ instead of $\alpha_{ij}^{(1)}$.
We call
$Q$ a \emph{quiver of type} $C$.
Let $Q^\circ := Q^\circ(C,\Omega)$ be the quiver obtained from $Q$ by deleting
all loops $\vep_i$.
Clearly,
$Q^\circ$ is an acyclic quiver.
Having said that, one might want to call $\Omega$ an \emph{acyclic orientation}. 
Future research might require to modify the definition of an orientation
and drop the acyclicity assumption.

Throughout let $K$ be a field.
For a quiver $Q = Q(C,\Omega)$ and a symmetrizer $D = \diag(c_1,\ldots,c_n)$ of 
$C$, let
\[
H := H(C,D,\Omega) := KQ/I
\] 
where $KQ$ is the path algebra of $Q$, and $I$ is the ideal of $KQ$
defined by the following
relations:
\begin{itemize}

\item[(H1)]
For each $i$ we have the \emph{nilpotency relation}
\[
\vep_i^{c_i} = 0.
\]

\item[(H2)]
For each $(i,j) \in \Omega$ and each $1 \le g \le g_{ij}$ we have
the \emph{commutativity relation}
\[
\vep_i^{f_{ji}}\alpha_{ij}^{(g)} = \alpha_{ij}^{(g)}\vep_j^{f_{ij}}.
\]

\end{itemize}
The following remarks are straightforward.
\begin{itemize}

\item[(i)]
$H$ is a finite-dimensional $K$-algebra.

\item[(ii)]
$H$ depends on the chosen symmetrizer $D$.
But note that the relation (H2) does not depend on $D$.

\item[(iii)]
The relation (H2)  
becomes redundant for all
$(i,j) \in \Omega$ with $k_{ij} = 1$.

\item[(iv)]
If $C$ is symmetric and if $D$ is minimal, 
then $H$ is isomorphic to the path algebra $KQ^\circ$.

\end{itemize}
The behaviour of $H$ under change of the symmetrizer
$D$ is studied in \cite{GLS1}.

The \emph{opposite orientation} of an orientation $\Omega$ is defined as
$\Omega^* := \{ (j,i) \mid (i,j) \in \Omega \}$.
Let
$\overline{\Omega} := \Omega \cup \Omega^*$.
For later use, let us define
\begin{align*}
\Omega(i,-) &:= \{ j \in Q_0 \mid (i,j) \in \Omega \}, &
\Omega(-,j) &:= \{ i \in Q_0 \mid  (i,j) \in \Omega \},\\
\overline{\Omega}(i,-) &:= \{ j \in Q_0 \mid (i,j) \in \overline{\Omega} \}, &
\overline{\Omega}(-,j) &:= \{ i \in Q_0 \mid  (i,j) \in \overline{\Omega} \}.
\end{align*}
Observe that $\overline{\Omega}(i,-) = \overline{\Omega}(-,i)$.

For $(i,j) \in \overline{\Omega}$ define
\[
\sgn(i,j) :=
\begin{cases}
1 & \text{if $(i,j) \in \Omega$},\\
-1 & \text{if $(i,j) \in \Omega^*$}.
\end{cases}
\]

For $Q = Q(C,\Omega)$ and a symmetrizer $D = \diag(c_1,\ldots,c_n)$ of $C$, 
we define an algebra
\[
\Pi := \Pi(C,D,\Omega) := K\overline{Q}/\overline{I}
\] 
as follows.
The \emph{double quiver} $\overline{Q} = \ov{Q}(C)$ is obtained from 
$Q$ by adding a new arrow 
$\alpha_{ji}^{(g)}\df i \to j$
for each arrow $\alpha_{ij}^{(g)}\df j \to i$ of $Q^\circ$. 
(Note that we did not add any new loops to the quiver $Q$.)
The ideal $\overline{I}$ of the path algebra
$K\overline{Q}$ is defined by the following
relations:
\begin{itemize}

\item[(P1)]
For each $i$ we have the \emph{nilpotency relation}
\[
\vep_i^{c_i} = 0.
\]

\item[(P2)]
For each $(i,j) \in \overline{\Omega}$ and each $1 \le g \le g_{ij}$ we have
the \emph{commutativity relation}
\[
\vep_i^{f_{ji}}\alpha_{ij}^{(g)} = \alpha_{ij}^{(g)}\vep_j^{f_{ij}}.
\]

\item[(P3)]
For each $i$ we have the \emph{mesh relation}
\[
\sum_{j\in \ov{\Omega}(-,i)} \sum_{g=1}^{g_{ji}} \sum_{f=0}^{f_{ji}-1} \sgn(i,j)
\vep_i^f\alpha_{ij}^{(g)} \alpha_{ji}^{(g)} \vep_i^{f_{ji}-1-f} = 0.
\]

\end{itemize}
We call $\Pi$ a \emph{preprojective algebra} of type $C$.
Here are again some straightforward remarks:
\begin{itemize}

\item[(i)]
Up to isomorphism,
the algebra $\Pi := \Pi(C,D) := \Pi(C,D,\Omega)$ does not depend on
the orientation $\Omega$ of $C$.

\item[(ii)]
In general, $\Pi$ can be infinite-dimensional.

\item[(iii)]
$\Pi$ depends on the chosen symmetrizer $D$.
But note that the relations (P2) and (P3) do not depend on $D$.

\item[(iv)]
If $C$ is symmetric and if $D$ is minimal, 
then $\Pi$ is isomorphic to the
classical preprojective algebra $\Pi(Q^\circ)$ associated with $Q^\circ$.

\end{itemize}
For an example illustrating the above definitions, see below 
Section~\ref{example-def-H-Pi}.

\subsection{Main results}\label{sec:main}

Let $e_1,\ldots,e_n$ be the idempotents in $H$ (resp. $\Pi$) corresponding to the vertices of $Q$ (resp. $\ov{Q}$).
Define $H_i := e_iHe_i$.
Clearly, $H_i$ is isomorphic to the truncated polynomial ring
$K[\vep_i]/(\vep_i^{c_i})$.
For each representation $M$ of $H$ or $\Pi$ we get an
$H_i$-module structure on $M_i := e_iM$.
The following definition is of central importance.

\begin{Def}\label{def:locallyfree}
A module $M \in \rep(H)$ or $M \in \rep(\Pi)$ is called
\emph{locally free} if $M_i$ is a free $H_i$-module
for every $i$.
\end{Def}

Let $\rep_\vp(H)$ (resp. $\rep_\vp(\Pi)$) be the subcategory of all locally free $M \in \rep(H)$ (resp. $M \in \rep(\Pi)$).

\begin{Thm}\label{thm:main0}
The algebra $H$ is a $1$-Iwanaga-Gorenstein algebra.
For $M \in \rep(H)$ the following are equivalent:
\begin{itemize}

\item[(i)]
$\pdim(M) \le 1$;

\item[(ii)]
$\idim(M) \le 1$;

\item[(iii)]
$\pdim(M) < \infty$;

\item[(iv)]
$\idim(M)  < \infty$;

\item[(v)]
$M$ is locally free.

\end{itemize}
\end{Thm}

Let $M$ be a locally free module. For each $i\in Q_0$ let $r_i$ be the rank of the free $H_i$-module~$M_i$. 
Thus $\dim_K(M_i) = r_ic_i$. We put 
\[
\rkv(M) := (r_1,\ldots,r_n).
\]

Let $\tau$ be the Auslander-Reiten
translation for the algebra $H$, and let $\tau^-$ be the  
inverse Auslander-Reiten translation.
An indecomposable $H$-module $M$ is \emph{preprojective}
(resp. \emph{preinjective}) if there exists some
$k \ge 0$ such that $M \cong \tau^{-k}(P)$ (resp.
$M \cong \tau^k(I)$) for some indecomposable projective
$H$-module $P$ (resp. indecomposable injective
$H$-module $I$). 
Let us warn the reader that
the usual definition of a preprojective or preinjective module $M$
requires some additional conditions on the Auslander-Reiten
component containing $M$.

In general, the Auslander-Reiten translates $\tau^k(M)$ of an indecomposable  
locally free $H$-module $M$ are not locally free, see the
example in Section~\ref{example:typeA2}.
An indecomposable $H$-module $M$ is called $\tau$-\emph{locally free}, if
$\tau^k(M)$ is locally free for all $k \in \Z$.

A module $M$ over an algebra $A$ is called \emph{rigid} if
$\Ext_A^1(M,M) = 0$.

The following result is an analogue for the algebras $H=H(C,D,\Omega)$
of Gabriel's Theorem (1) for quivers and of its generalization
by Dlab and Ringel to modulated graphs.

\begin{Thm}\label{thm:main1}
There are only finitely many isomorphism classes of $\tau$-locally free 
$H$-modules if and only if $C$ is of Dynkin type.
In this case, the following hold:
\begin{itemize}

\item[(i)]
The map 
$M \mapsto \rkv(M)$ 
yields a bijection between the set of isomorphism classes of 
$\tau$-locally free $H$-modules
and the set $\Delta^+(C)$ of positive roots of the semisimple complex Lie
algebra associated with $C$.

\item[(ii)]
For an indecomposable $H$-module $M$ the following are equivalent:
\begin{itemize}

\item[(a)]
$M$ is preprojective;

\item[(b)]
$M$ is preinjective;

\item[(c)]
$M$ is $\tau$-locally free;

\item[(d)]
$M$ is locally free and rigid.

\end{itemize}
\end{itemize}
\end{Thm}

Crawley-Boevey \cite{CBRigid} studied representations of quivers
over principal ideal domains.
There are some striking analogies between Theorem~\ref{thm:main1} 
and his results.

Note that the algebras $H$ are usually representation infinite, even
if $C$ is a Cartan matrix of Dynkin type with a minimal
symmetrizer $D$.
Already for $C$ of type $B_3$ with
minimal symmetrizer,
there exist indecomposable locally free 
$H$-modules $M$ with $\rkv(M) \notin \Delta^+(C)$,
see Section~\ref{secB3ex}.
Furthermore, for $C$ of type $B_5$ with minimal symmetrizer there exists a $K^*$-family of pairwise non-isomorphic indecomposable locally free $H$-modules, all having the same dimension vector.

Inspired by
the classical theory for path algebras and modulated graphs 
we define \emph{Coxeter functors}
\[
C^+,C^-\df \rep(H) \to \rep(H)
\]
as products of reflection functors, see Section~\ref{sec8}.
Let 
\[
T\df \rep(H) \to \rep(H)
\] 
be the \emph{twist automorphism} 
induced from the algebra automorphism $H \to H$ defined by
$\vep_i \mapsto \vep_i$ and $\alpha_{ij}^{(g)} \mapsto -\alpha_{ij}^{(g)}$.
In other words, $T$
sends
a representation $(M_i,M(\alpha_{ij}^{(g)}),M(\vep_i))$ of $H$ to 
$(M_i,-M(\alpha_{ij}^{(g)}),M(\vep_i))$.
The following theorem, analogous to Gabriel's Theorem (3),
relates Coxeter functors to the Auslander-Reiten translation, and
provides the main step in proving Theorem~\ref{thm:main1}.

\begin{Thm}\label{thm:main2}
For $M \in \rep(H)$ there are functorial isomorphisms
\[
\Du\Ext_H^1(M,H) \cong TC^+(M)
\text{\;\;\; and \;\;\;}
\Ext_H^1(\Du(H),M) \cong TC^-(M).
\]
Furthermore, if $M \in \rep_\vp(H)$,
then there are functorial isomorphisms
\[
\tau(M) \cong TC^+(M)
\text{\;\;\; and \;\;\;}
\tau^{-}(M) \cong TC^-(M).
\]
Vice versa, if $\tau(M) \cong TC^+(M)$ or
$\tau^{-}(M) \cong TC^-(M)$ for some $M \in \rep(H)$, then 
$M \in \rep_\vp(H)$.
\end{Thm}

Recall that $H$ is a $1$-Iwanaga-Gorenstein algebra, and that $\cG\cP(H)$
denotes the subcategory of Gorenstein-projective $H$-modules.

\begin{Cor}
For $M \in \rep(H)$ the following are equivalent:
\begin{itemize}

\item[(i)]
$C^+(M) = 0$;

\item[(ii)]
$M \in \cG\cP(H)$.

\end{itemize}
\end{Cor}

For an algebra $A$ and an $A$-$A$-bimodule $M$, let $T_A(M)$
denote the corresponding tensor algebra.
Theorem~\ref{thm:main2} implies the following description of the preprojective algebra
$\Pi=\Pi(C,D)$ associated with $H$.

\begin{Thm}\label{thm:main3}
$\Pi \cong T_H(\Ext_H^1(\Du(H),H))$.
\end{Thm}

The algebra $\Pi$ contains $H$ as a subalgebra in an
obvious way.
Let ${_H}\Pi$ be the algebra $\Pi$ considered as a left module
over $H$.
The following result says
that ${_H}\Pi$ is isomorphic to the direct sum of all
preprojective $H$-modules.
This justifies that $\Pi$ is called a \emph{preprojective algebra}.

\begin{Thm}\label{thm:main3a}
We have ${_H}\Pi \cong \bigoplus_{m \ge 0} \tau^{-m}({_H}H)$.
In particular, $\Pi$ is finite-dimen\-sio\-nal if and only if
$C$ is of Dynkin type.
\end{Thm}

Finally, we obtain the following analogue for locally free $\Pi$-modules
of the classical important Ext-symmetry of preprojective algebras.

\begin{Thm}\label{thm:main5}
For $M,N \in \rep_\vp(\Pi)$ we have a functorial isomorphism
\[
\Ext_\Pi^1(M,N) \cong \Du\Ext_\Pi^1(N,M).
\]
\end{Thm}

\subsection{Previous appearances of $H(C,D,\Omega)$}

\subsubsection{}
Let $Q$ be a quiver without oriented cycles.
Ringel and Zhang \cite{RiZ} study representations of $Q$
over the algebra $A:=K[X]/(X^2)$ of dual numbers.
This can be interpreted as the category of $\Lambda$-modules with 
$\Lambda := A\otimes_K KQ$. 
It is shown in \cite{RiZ} that $\Lambda$ is a $1$-Iwanaga-Gorenstein algebra, 
and that the stable category of $\cG\cP(\Lambda)$ is triangle equivalent 
to the orbit category $D^b(KQ)/[1]$ of the bounded derived category
$D^b(KQ)$ of the path algebra $KQ$ modulo the shift functor $[1]$.
In our setup, if we take symmetric Cartan matrices $C$ with
symmetrizer $D = \diag(2,\ldots,2)$, then the class of algebras 
$H(C,D,\Omega)$ coincides with the class of algebras studied by Ringel and 
Zhang.
Fan \cite{F} studies the Hall algebra of representations of $Q$ over
$K[X]/(X^m)$ with $m \ge 1$.
Again this is a special case of our setup with $C$ symmetric and
$D = \diag(m,\ldots,m)$.
For $Q$ a quiver of type $A_2$, $A = K[X]/(X^m)$ and $\Lambda := A\otimes_K KQ$
the category
$\cG\cP(\Lambda)$ is studied in work of Ringel and Schmidmeier
\cite{RiS}.
Note also that in this case we have $\Lambda \cong T_2(A)$, where
$T_2(A)$ is the algebra of upper triangular $2 \times 2$-matrices
with entries in $A$.
More generally, the algebras $T_2(A)$ with $A$ a Nakayama algebra have been studied by Skowro\'nski \cite{S},
and the algebras $T_n(A)$ have been studied by Leszczy\'nski and Skowro\'nski \cite{LS}.

\subsubsection{}
A general framework for studying cluster structures arising from 2-Calabi-Yau categories with loops has been provided by \cite{BMV}.
As an example they study the cluster category 
$\cC := D^b(\cT_n)/\tau^-[1]$ of 
the mesh category of a tube $\cT_n$ of rank $n \ge 2$.
The endomorphism algebras of the maximal rigid objects in
$\cC$ have been studied by Vatne and Yang \cite{V,Yan}.
It turns out that
there exists a maximal rigid object $T$ in $\cC$
such that $\End_\cC(T)$ is isomorphic to one of our algebras
$H(C,D,\Omega)$, where $C$ is of Dynkin type $C_{n-1}$ and $D$
is minimal.
(We identify the types $C_1 = A_1$ and $C_2 = B_2$.)

\subsubsection{}
Let $Q$ be a Dynkin quiver of type $E_8$, and let
$F := S^4\Sigma^{-4}$, where
$S$ is the Serre functor and $\Sigma$ is the translation functor for
the bounded derived category $D^b(KQ)$ of the path algebra $KQ$.
Ladkani \cite[Section~2.6]{La} studies the orbit category $\cC := D^b(KQ)/F$.
He shows that $\cC$ is a triangulated $2$-Calabi-Yau category containing
exactly $6$ cluster-tilting objects.
Ladkani shows that $\cC$ categorifies a cluster algebra
of Dynkin type $G_2$.
He also shows that the cluster tilting-objects in $\cC$
have an endomorphism algebra isomorphic to 
$A = KQ/I$, where $Q$ is a quiver of the form
\[
\xymatrix{
1 & 2 \ar[l] \ar@(ul,ur)^{\vep}
}
\text{\;\;\ or \;\;\;} 
\xymatrix{
1 \ar[r] & 2 \ar@(ul,ur)^{\vep}
}
\]
and $I$ is generated by $\vep^3$.
Note that the algebras $A$ are isomorphic to the algebras $H(C,D,\Omega)$ with
$C$ of type $G_2$ and $D$ minimal.

\subsection{Previous appearances of $\Pi(C,D)$}

\subsubsection{}\label{sec1.5.2}
Let $C$ be a Cartan matrix of Dynkin type. 
In \cite{HL}, an algebra $A=A(C)$ was introduced, by means of an infinite quiver with potential. 
Certain finite-dimensional $A$-modules (the \emph{generic kernels} $K^{(i)}_{k,m}$, see \cite[Definition~4.5]{HL}) 
were shown to encode the $q$-characters of the Kirillov-Reshetikhin modules of the quantum loop algebra 
$U_q(L\mathfrak{g})$, where $\mathfrak{g}$ is the complex simple Lie algebra with Cartan matrix $C$. 
The connection with the algebras considered in this article is the following:
Let $\widetilde{\Pi}(C)$ denote the algebra $K\overline{Q}/\widetilde{I}$, where $\widetilde{I}$ is the two-sided ideal 
defined by the relations (P2) and (P3) only. 
(Thus, the nilpotency relation (P1) is omitted.)
Then $A(C)$ is a truncation of a 
$\Z$-covering of $\wti{\Pi}(C^*)$, where $C^*$ 
is the transposed Cartan matrix, in other words, the Cartan matrix of the Langlands dual 
$\mathfrak{g}^L$ of $\mathfrak{g}$.
In particular, for $m\ll 0$ the generic kernels 
$K^{(i)}_{c_i,m}$ of $A(C)$ 
coincide with the indecomposable projective $\Pi(C^*,D)$-modules
regarded as $\Z$-graded $\wti{\Pi}(C^*)$-modules 
(compare for instance \cite[Section 6.5]{HL} to Figure~\ref{Fig:projPC3} below).
This generalizes \cite[Example~4.7]{HL}.

\subsubsection{}
The algebras $\wti{\Pi}(C)$ mentioned in
Section~\ref{sec1.5.2} were defined and studied independently 
by Cecotti \cite[Section~3.4]{C} and Cecotti and del Zotto 
\cite[Section~5.1]{CD}.
In \cite{C} they are called \emph{generalized preprojective algebras}.

\subsubsection{}
For $(C,\Omega)$ let 
\[
W(C,\Omega) :=
\sum_{(j,i) \in \ov{\Omega}} \sum_{g=1}^{g_{ji}} \sgn(i,j)
\vep_i^{f_{ji}}\alpha_{ij}^{(g)} \alpha_{ji}^{(g)}.
\]
Then the cyclic derivatives of the potential $W(C,\Omega)$ yield the defining 
relations (P2) and (P3) of $\Pi(C,D)$, compare \cite{C,CD,HL}, where these 
relations are also encoded via potentials.

\subsubsection{}
After the first version of this article appeared on arXiv, we were informed by 
D. Yamakawa of the following connection between $\Pi(C,D)$ and some quiver 
varieties for quivers with multiplicities introduced in \cite{Yam}. 
Suppose that $C$ is a generalized Cartan matrix of the form 
$C = (c_{ij}) = 2I_n - AD$, where $I_n$ is the unit matrix, $A=(a_{ij})$ is a 
symmetric matrix with $a_{ij}\in\N$ and $a_{ii}=0$, and 
$D = \diag(c_1,\ldots,c_n)$ with positive integers $c_i$. 
Then $C$ is symmetrizable with symmetrizer $D$. 
(Note that not every symmetrizable Cartan matrix is of this form. 
For instance Cartan matrices of type $C_n$ are of this form, but not Cartan 
matrices of type $B_n$.) 
Assume further that $k_{ij} = \gcd(c_i,c_j) = 1$ whenever $c_{ij} < 0$. 
Then the defining relations (P2) of $\Pi(C,D)$ are redundant, and the mesh 
relations (P3) can be regarded as the vanishing of the moment map of some 
Hamiltonian space considered in \cite{Yam}. 
As a result, in this case, the isomorphism classes of locally free 
$\Pi(C,D)$-modules of rank vector ${\mathbf r}$ are parametrized by the set 
theoretical quotient $\mathcal{N}^{\mathrm set}(\mathbf{r},0)$ 
of \cite{Yam}.

\subsection{Future directions}

This article intends to provide the
foundation for 
generalizing many of the connections between path algebras,
preprojective algebras, Lie algebras and cluster algebras from the symmetric
to the symmetrizable case.

In particular, since the algebras $H$ and $\Pi$ are defined via quivers with
relations, one can study their module varieties over an arbitrary field $K$.
Taking $K=\C$, one can hope to
generalize Lusztig's nilpotent varieties and Nakajima quiver varieties from
the symmetric to the symmetrizable case.

As a first step in this direction, in 
\cite{GLS2} we
construct the enveloping algebra of the positive part of an arbitrary 
simple finite-dimensional complex Lie algebra as an algebra of constructible functions on varieties of locally free $H$-modules.

\subsection{}
The article is organized as follows.
In Section~\ref{sec2} we recall some definitions and basic facts on
Cartan matrices, quadratic forms and Weyl groups.
A description of the projective and injective $H$-modules and
some fundamental results on locally free $H$-modules
are obtained in Section~\ref{sec3}.
In particular, Section~\ref{sec3} contains the proof of Theorem~\ref{thm:main0} (combine Proposition~\ref{vpdim} and Corollary~\ref{Gorenstein}).
In Section~\ref{sec4} we show that the quadratic form $q_C$ associated with
a Cartan matrix $C$ coincides with the restriction of the 
homological Euler form of $H$ to the subcategory of locally free $H$-modules.
The representation theory of the algebras $H$ and $\Pi$
can be reformulated in terms of a generalization of the representation
theory of modulated graphs.
This point of view, which is of central importance for proving (and in part also for formulating) our main results, is explained in Section~\ref{sec6}.
An interpretation of $H$ and $\Pi$
as tensor algebras is discussed in Section~\ref{sec7}.
Section~\ref{newsec7} provides a bimodule resolution of $H$.
In Section~\ref{newsec8} we introduce a trace pairing for $H_i$-modules
and relate it with the adjunction isomorphisms.
Section~\ref{sec8} contains some fundamental properties of
generalizations of BBK-reflection functors to our algebras $\Pi$.
(The letters BBK stand for Baumann and Kamnitzer \cite{BK1,BK2} and 
Bolten \cite{Bo}. 
Independently from each other they introduced reflection
functors for the classical preprojective algebras associated with quivers.)
The reflection functors for $\Pi$ restrict to reflection functors for $H$.
We show that the latter are generalized versions of
APR-tilting functors.
The intimate relation between Coxeter functors and the Auslander-Reiten translation for $H$ is studied in Section~\ref{sec10}.
Theorem~\ref{thm:main2} follows from
Theorem~\ref{thm:functoriso} and Proposition~\ref{prop:10.9}.
We also prove some crucial properties of the
algebras $\Pi$.
In particular, 
Theorem~\ref{thm:main3} corresponds to Corollary~\ref{cor:10.7}.
In Section~\ref{sec12} we use the previous constructions for proving 
Theorem~\ref{thm:main1} (see Theorem~\ref{Dynkinroots}).
The proof of Theorem~\ref{thm:main3a} can be found in Section~\ref{sec:10.3}.
We also obtain some first results on the Auslander-Reiten theory of $H$.
Section~\ref{sec9} contains the construction of a bimodule resolution
for $\Pi$.
This resolution plays an important part in the study of locally free representations.
In particular, Section~\ref{sec:extsymm} contains the proof of
Theorem~\ref{thm:main5} (see Theorem~\ref{thm:9.11}).
Finally, Section~\ref{sec13} contains a collection of examples.

\subsection{Notation}
By a \emph{subcategory} we always mean a full subcategory which is closed
under isomorphisms and direct summands.
By an \emph{algebra} we mean an associative $K$-algebra with $1$.
For a $K$-algebra $A$ let $\md(A)$ be the category of 
finite-dimensional left $A$-modules.
If $A = KQ/I$ is the path algebra of a quiver $Q$ modulo some
ideal $I$, then $\rep(A)$ denotes the category of finite-dimensional
representations of $(Q,I)$.
By definition these are the representations of $Q$ which are annihilated by
$I$.
We often identify $\md(A)$ and $\rep(A)$.
Let $\proj(A)$ and $\inj(A)$ be the full subcategories of $\rep(A)$
with objects the projective and injective $A$-modules, respectively.
Let $\Du := \Hom_K(-,K)$ be the usual $K$-duality.
For a finite-dimensional algebra $A$ let $\tau(-) = \tau_A(-)$ be the Auslander-Reiten
translation of $A$.
For a module $X$ we denote by $\add(X)$ the subcategory of modules 
which are isomorphic to finite direct sums of direct summands of $X$.
As a general reference for the representation theory of finite-dimensional
algebras we refer to the books \cite{ARS,Ri2}.
The composition of maps $f\df X \to Y$ and $g\df Y \to Z$
is denoted by $gf\df X \to Z$.
For arrows $\alpha\df i \to j$ and $\beta\df j \to k$ in a quiver, 
we write their composition as $\beta\alpha\df i \to k$.
By $\N$ we denote the natural numbers including $0$.


\section{Cartan matrices and the Weyl group}\label{sec2}


\subsection{Cartan matrices and valued graphs}
Let $C = (c_{ij}) \in M_n(\Z)$ be a 
Cartan matrix, and let $D = \diag(c_1,\ldots,c_n)$ 
be a symmetrizer of $C$.
The \emph{valued graph} $\GG(C)$ of $C$ 
has vertices $1,\ldots,n$ and an (unoriented) edge between $i$ and $j$
if and only if $c_{ij} < 0$. 
An edge $
\xymatrix{
i \ar@{-}[r] & j
}
$
has \emph{value} $(|c_{ji}|,|c_{ij}|)$.
In this case, we display this valued edge as
\[
\xymatrix@-0.5pc{
i \ar@{-}[rr]^{(|c_{ji}|,|c_{ij}|)} && j
}
\]
and we just write  
$
\xymatrix{
i \ar@{-}[r] & j
}
$
if $(|c_{ji}|,|c_{ij}|) = (1,1)$.

A Cartan matrix $C$ is \emph{connected} if $\GG(C)$ is a connected graph.
In this case, the symmetrizer $D$ is
uniquely determined up to multiplication with a positive integer.
More precisely, if $D$ is a minimal symmetrizer of a connected Cartan matrix 
$C$, then the other symmetrizers of $C$
are given by $mD$ with $m \ge 1$.

\subsection{The quadratic form}
Define the quadratic form $q_C\df \Z^n \to \Z$ of $C$ by
\[
q_C := \sum_{i=1}^n c_i X_i^2 - \sum_{i<j} c_i|c_{ij}| X_iX_j.
\]
(Recall that $c_i|c_{ij}| = c_j|c_{ji}| $.)
The quadratic form $q_C$ plays a crucial role in the representation theory
of the quivers of Cartan type $C$ and more generally of the species 
(see for example~\cite{DR1}) of type $C$.

A Cartan matrix $C$ is of \emph{Dynkin} or \emph{Euclidean type}
if $q_C$ is positive definite or positive semidefinite, respectively.
It is well known that $C$ is of Dynkin type if and only if $\GG(C)$
is a disjoint union of Dynkin graphs.
(The Dynkin graphs are listed in Section~\ref{secDynkin}.)

\subsection{The Weyl group}
As before let $C = (c_{ij})$ be a Cartan matrix, and let
$\alpha_1,\ldots,\alpha_n$ be the positive simple roots of the
Kac-Moody algebra $\g(C)$ associated with $C$.
For $1 \le i,j \le n$ define
\[
s_i(\alpha_j) := \alpha_j - c_{ij} \alpha_i.
\]
This yields a reflection $s_i\df \Z^n \to \Z^n$ on the \emph{root lattice}
\[
\Z^n = \sum_{i=1}^n \Z \alpha_i
\]
where we identify $\alpha_i$ with the $i$th standard basis vector of $\Z^n$.
The \emph{Weyl group} $W(C)$ of $\g(C)$ is the subgroup of
${\rm Aut}(\Z^n)$ generated by $s_1,\ldots,s_n$.
The Weyl group is finite if and only if $C$ is of Dynkin type.

\subsection{Roots}
Let 
\[
\Delta_\real(C) := \bigcup_{i=1}^n W(\alpha_i)
\] 
be the set of \emph{real roots} of $C$.

Let 
\[
(-,-)_C\df \Z^n \times \Z^n \to \Z
\] 
be the 
\emph{symmetric bilinear form} of $C$ defined by 
$(\alpha_i,\alpha_j)_C := c_ic_{ij}$.
The \emph{fundamental region} of $C$ is defined by
\[
F := \{ d \in \N^n \mid d \not= 0,\; {\rm supp}(d) \text{ connected}, 
(d,\alpha_i)_C \le 0 \text{ for all } 1 \le i \le n \},
\]
where ${\rm supp}(d)$ is the full subgraph of $\GG(C)$ given by the
vertices $i$ with $d_i \not= 0$.
Then 
\[
\Delta_\imag(C) := W(F) \cup W(-F)
\] 
is the set of \emph{imaginary roots} of $C$.

Let
\[
\Delta_\real^+(C) := \Delta_\real(C) \cap \N^n
\text{\;\;\; and \;\;\;}
\Delta_\imag^+(C) := \Delta_\imag(C) \cap \N^n
\] 
be the set of \emph{positive real roots} and \emph{positive
imaginary roots}, respectively.
It turns out that 
\[
\Delta_\real(C) = \Delta_\real^+(C) \cup
-\Delta_\real^+(C)
\text{\;\;\; and \;\;\;}
\Delta_\imag(C) = \Delta_\imag^+(C) \cup
-\Delta_\imag^+(C).
\]
Finally, let
\[
\Delta(C) := \Delta_\real(C) \cup \Delta_\imag(C)
\] 
be the set of \emph{roots} of $C$, and 
\[
\Delta^+(C) := \Delta(C) \cap \N^n = \Delta_\real^+(C) \cup
\Delta_\imag^+(C)
\] 
is the set of \emph{positive roots}.

By definition, for $w \in W(C)$ and $d \in \Delta(C)$ we have
$w(d) \in \Delta(C)$.
We have $q_C(d) = c_i$ if $d \in W(\alpha_i)$ is a real root, and 
$q_C(d) \le 0$ if $d$ is an imaginary root.
The following are equivalent:
\begin{itemize}

\item[(i)]
$C$ is of Dynkin type;

\item[(ii)]
$\Delta(C)$ is finite;

\item[(iii)]
$\Delta_\real(C) = \Delta(C)$.

\end{itemize}
%
\subsection{Coxeter transformations}\label{sec:ct}
For an orientation $\Omega$ of $C$ and some $1 \le i \le n$ let
\[
s_i(\Omega) :=
\{ (r,s) \in \Omega \mid i \notin \{r,s\} \}
\cup \{ (s,r) \in \Omega^* \mid i \in \{r,s\}  \}.
\]
If $i$ is a sink or source in $Q^\circ(C,\Omega)$, then $s_i(\Omega)$
is again an orientation of $C$.
A sequence $\bi = (i_1,\ldots,i_n)$ is a $+$-\emph{admissible sequence} 
for $(C,\Omega)$
if
$\{i_1,\ldots,i_n\} = \{1,\ldots,n\}$, $i_1$ is a sink in $Q^\circ(C,\Omega)$
and
$i_k$ is a sink in the acyclic quiver 
$Q^\circ(C,s_{i_{k-1}} \cdots s_{i_1}(\Omega))$
for $2 \le k \le n$.
For such a sequence $\bi$ let 
\[
\beta_{\bi,k} := \beta_k :=
\begin{cases} 
\alpha_{i_1} & \text{if $k=1$},\\
s_{i_1}s_{i_2} \cdots s_{i_{k-1}}(\alpha_{i_k}) & \text{if $2 \le k \le n$}
\end{cases}
\]
where $s_1,\ldots,s_n \in W(C)$.
Similarly, define 
\[
\gamma_{\bi,k} := \gamma_k :=
\begin{cases}
s_{i_n} \cdots s_{i_{k+1}}(\alpha_{i_k}) & \text{if $1 \le k \le n-1$},\\
\alpha_{i_n} & \text{if $k=n$}.
\end{cases}
\]

Let
\[
c^+ : = s_{i_n}s_{i_{n-1}} \cdots s_{i_1}\df \Z^n \to \Z^n
\text{\;\;\; and \;\;\;}
c^- : = s_{i_1}s_{i_2} \cdots s_{i_n}\df \Z^n \to \Z^n.
\]
be the \emph{Coxeter transformations}.
For $k \in \Z$ we set
\[
c^k := 
\begin{cases}
(c^+)^k& \text{if $k > 0$},\\
(c^-)^{-k}& \text{if $k < 0$},\\
\id & \text{if $k=0$}.
\end{cases}
\]
We get
\begin{align*}
c^+(\beta_{\bi,k}) &= (s_{i_n}s_{i_{n-1}} \cdots s_{i_1})(s_{i_1}s_{i_2} 
\cdots s_{i_{k-1}}(\alpha_{i_k}))\\
&= s_{i_n}s_{i_{n-1}} \cdots s_{i_k}(\alpha_{i_k})\\
&= - s_{i_n}s_{i_{n-1}} \cdots s_{i_{k+1}}(\alpha_{i_k})\\
&= - \gamma_{\bi,k}.
\end{align*}

The following two lemmas are well known.
For example, they are a consequence of the study of
preprojective and preinjective representations of species without oriented
cycles.

\begin{Lem}\label{lem:ct1}
Suppose $C$ is not of Dynkin type.
Then the elements $c^{-r}(\beta_i)$ and $c^s(\gamma_j)$
with $r,s \ge 0$ and $1 \le i,j \le n$ are pairwise different elements in $\Delta_\real^+(C)$.  
\end{Lem}

Let $C$ be of Dynkin type.
For $1 \le i \le n$ let $p_i \ge 1$ be minimal with
$c^{-{p_i}}(\beta_i) \notin \N^n$, and let
$q_j \ge 1$ be minimal such that $c^{q_j}(\gamma_j) \notin \N^n$.
It is well known that such $p_i$ and $q_j$ exist.
The elements $c^{-r}(\beta_i)$ with $1 \le i \le n$ and $0 \le r \le p_i-1$
are pairwise different, and the elements $c^s(\gamma_j)$ with $1 \le j \le n$ 
and $0 \le s \le q_j-1$ are pairwise different.
 
\begin{Lem}\label{lem:ct2}
Assume that $C$ is of Dynkin type.
Then 
\[
\Delta^+(C) = 
\{ c^{-r}(\beta_i) \mid 1 \le i \le n,\, 0 \le r \le p_i-1 \} = 
\{ c^s(\gamma_j) \mid 1 \le j \le n,\, 0 \le s \le q_j-1 \}.
\]
\end{Lem}


\section{Locally free $H$-modules}
\label{sec3}


For the whole section, let $H = H(C,D,\Omega)$ and 
$Q = Q(C,\Omega)$ as defined in Section~\ref{def:HPi}.

\subsection{Modules defined by idempotents}\label{sec3.0}
Let $M \in \rep(H)$.
For $1 \le i \le n$ let $M_i := e_iM$.
The $H$-module structure on $M$ is described by 
the spaces $M_i$ with $1 \le i \le n$ and by $K$-linear maps
$M(\alpha)\df M_{s(\alpha)} \to M_{t(\alpha)}$ with
$\alpha$ running through the arrows of the quiver $Q$.
(Of course, these maps need to satisfy the defining relations for $H$.)

For a non-empty subset $J$ of $\{ 1,\ldots,n \}$ let
$e := \sum_{j \in J} e_j$.
Thus $e$ is an idempotent in $H$.
We get a vector space decomposition
\[
eM = \bigoplus_{j \in J} M_j.
\]
For $1 \le i \le n$ we set 
\[
(eM)_i := 
\begin{cases}
M_i & \text{if $i \in J$},
\\
0 & \text{otherwise},
\end{cases}
\]
and for each arrow $\alpha$ of $Q$ we define a map 
$(eM)(\alpha)\df (eM)_{s(\alpha)} \to  (eM)_{t(\alpha)}$
by
\[
(eM)(\alpha) := 
\begin{cases}
M(\alpha) & \text{if $s(\alpha),t(\alpha) \in J$},
\\
0 & \text{otherwise}.
\end{cases}
\]
This defines an $H$-module structure on $eM$.
This follows from the nature of the defining relations 
for $H$.
Namely, for a given $i$, there are no relations passing through $i$, i.e. any 
of the relations either starts in $i$, ends in $i$,
or does not involve $i$.

\subsection{Description of the 
projective and injective modules}\label{sec3.1}
The algebra $H$ 
is by definition a path algebra modulo an admissible ideal generated
by zero relations and commutativity relations.
This implies that each indecomposable projective $H$-module $P_i := He_i$ has a basis $B_i$
with the following properties: 
For each path $p$ in $Q$ and each $b \in B_i$ we have $p \cdot b \in B_i \cup \{0\}$.
In particular, we can visualize $P_i$ by drawing a graph with vertices the elements
in $B_i$, and an arrow $b \xrightarrow{a} b'$ if for an arrow $a \in Q_1$ and
$b,b' \in B_i$ we have
$a \cdot b = b'$.
We say that $P_i$ has a \emph{multiplicative basis}.
Similarly, the indecomposable injective $H$-modules $I_i := \Du(e_iH)$ have
a multiplicative basis.

Let $i$ be a vertex of $Q$.
We define an $H$-module ${_-}H_i$ as follows:
A basis of $e_i{_-}H_i$ is given by vectors $a_{i,c}$ with
$1 \le c \le c_i$, and for $j \in \Omega(-,i)$
a basis of $e_j{_-}H_i$ is given by vectors $b_{j,c}^{f,g}$ with $1 \le c \le c_j$, $1 \le f \le f_{ji}$ and $1 \le g \le g_{ji}$, and for 
$s \notin \Omega(-,i)$ define $e_s{_-}H_i := 0$.
The
arrows of $Q$ act as follows:
\begin{align*}
\vep_ia_{i,c} &:= 
{\begin{cases}
a_{i,c-1} & \text{if $c\ge 2$},\\
0 & \text{if $c=1$},
\end{cases}}\\
\vep_jb_{j,c}^{f,g} &:= 
{\begin{cases}
b_{j,c-1}^{f,g} & \text{if $c\ge 2$},\\
0 & \text{if $c=1$},
\end{cases}}
\end{align*}
and for $0 \le k < k_{ji}$, $0 \le f < f_{ji}$ and $1 \le g \le g_{ji}$ 
we have
\[
\alpha_{ji}^{(g)} a_{i,c_i-f-kf_{ji}} := b_{j,c_j-kf_{ij}}^{f_{ji}-f,g}.
\]

For $c_i=9$, $c_j=6$, $f_{ji}=3$, $f_{ij}=2$ (and therefore $k_{ji} = 3$)
we display a part of the module $(e_i+e_j){_-}H_i$ in Figure~\ref{fig:P(i,j,g)}.
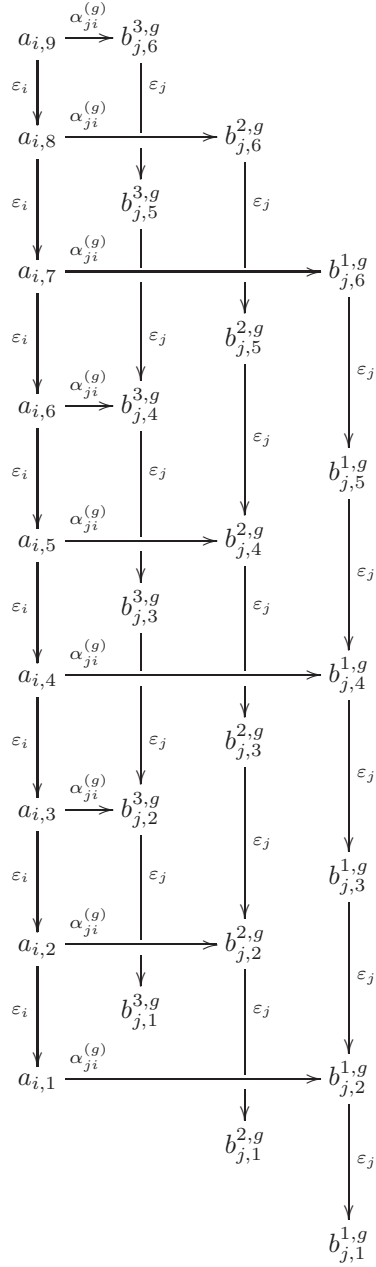
\begin{figure}[!htb]
\small
\[
\xymatrix@-1.5pc{
a_{i,9} \ar[dd]_{\vep_i}\ar[rr]^<<<<{\alpha_{ji}^{(g)}} && b_{j,6}^{3,g} 
\ar@{->}'[dd]^{\vep_j}[ddd]\\
&&\\
a_{i,8} \ar[dd]_{\vep_i}\ar[rrrr]^<<<<{\alpha_{ji}^{(g)}} &&&& b_{j,6}^{2,g} 
\ar@{->}'[dd]^{\vep_j}[ddd]\\
&& b_{j,5}^{3,g} \ar@{->}'[d][ddd]^{\vep_j}\\
a_{i,7} \ar[dd]_{\vep_i}\ar[rrrrrr]^<<<<{\alpha_{ji}^{(g)}} &&&&&& b_{j,6}^{1,g} \ar[ddd]^{\vep_j}\\
&&&& b_{j,5}^{2,g}  \ar[ddd]^{\vep_j}\\
a_{i,6} \ar[dd]_{\vep_i}\ar[rr]^<<<<{\alpha_{ji}^{(g)}} && b_{j,4}^{3,g} \ar@{->}'[dd]^{\vep_j}[ddd]\\
&&&&&& b_{j,5}^{1,g} \ar[ddd]^{\vep_j}\\
a_{i,5} \ar[dd]_{\vep_i}\ar[rrrr]^<<<<{\alpha_{ji}^{(g)}} &&&& b_{j,4}^{2,g} \ar@{->}'[dd]^{\vep_j}[ddd]\\
&& b_{j,3}^{3,g} \ar@{->}'[d][ddd]^{\vep_j}\\
a_{i,4} \ar[dd]_{\vep_i}\ar[rrrrrr]^<<<<{\alpha_{ji}^{(g)}} &&&&&& b_{j,4}^{1,g} \ar[ddd]^{\vep_j} \\
&&&& b_{j,3}^{2,g} \ar[ddd]^{\vep_j} \\
a_{i,3} \ar[dd]_{\vep_i}\ar[rr]^<<<<{\alpha_{ji}^{(g)}} && b_{j,2}^{3,g} \ar@{->}'[dd]^{\vep_j}[ddd]\\
&&&&&& b_{j,3}^{1,g} \ar[ddd]^{\vep_j}\\
a_{i,2} \ar[dd]_{\vep_i}\ar[rrrr]^<<<<{\alpha_{ji}^{(g)}} &&&& b_{j,2}^{2,g} \ar@{->}'[dd]^{\vep_j}[ddd]\\
&& b_{j,1}^{3,g} \\
a_{i,1} \ar[rrrrrr]^<<<<{\alpha_{ji}^{(g)}}&&&&&& b_{j,2}^{1,g} \ar[ddd]^{\vep_j}\\
&&&& b_{j,1}^{2,g}&&\\
&&&&&&\\
&&&&&& b_{j,1}^{1,g}
}
\]
\caption{Construction of the $H$-module $(e_i+e_j){_-}H_i$ with $c_i=9$, $c_j=6$, $f_{ji}=3$, $f_{ij}=2$ and $k_{ji} = 3$.}
\label{fig:P(i,j,g)}
\end{figure}
The module ${_-}H_i$ has one $i$-\emph{column} with basis
$(a_{i,1},\ldots,a_{i,c_i})$, and for each $j \in \Omega(-,i)$ it has a $j$-\emph{column} with basis
$(b_{j,1}^{f,g},\ldots,b_{j,c_j}^{f,g})$
for each $1 \le f \le f_{ji}$ and each $1 \le g \le g_{ji}$.
By definition we have
\[
\dim {_-}H_i = c_i + \sum_{j \in \Omega(-,i)} f_{ji}g_{ji}c_j.
\]
The number of $j$-columns of ${_-}H_i$ is $f_{ji}g_{ji} = |c_{ji}|$.

Suppose $j$ is a sink in $Q^\circ$.
The module $P_j = e_jP_j$ has a basis $a_{j,1},\ldots,a_{j,c_j}$ such that
\[
\vep_ja_{j,c} := 
\begin{cases}
a_{j,c-1} & \text{if $c\ge 2$},\\
0 & \text{if $c=1$},
\end{cases}.
\]
Then $(a_{j,1},\ldots,a_{j,c_j})$ is the $j$-\emph{column} of $P_j$.

Next, assume that $i$ is a vertex of $Q$ such that for each 
$j \in \Omega(-,i)$
the projective module $P_j$ is already constructed, and
$P_j$ has a distinguished basis including a $j$-column
$(a_{j,1},\ldots,a_{j,c_j})$, which forms a basis of $e_jP_j$.

Then $P_i$ is constructed as follows:
We take 
for each $j \in \Omega(-,i)$, $1 \le f \le f_{ji}$ and 
$1 \le g \le g_{ji}$
a copy $P_j^{f,g}$ of $P_j$ and
identify the $j$-column $(b_{j,1}^{f,g},\ldots,b_{j,c_j}^{f,g})$
of ${_-}H_i$ with the $j$-column $(a_{j,1},\ldots,a_{j,c_j})$ of $P_j^{f,g}$.
The resulting module is our indecomposable projective $H$-module 
$P_i$,
and by definition its $i$-\emph{column} is the $i$-column 
$(a_{i,1},\ldots,a_{i,c_i})$ of the module ${_-}H_i$.

The indecomposable injective $H$-modules $I_j$ are constructed dually by gluing
modules ${_j}H_-$, which are dual to the modules ${_-}H_i$.
Again, for 
$c_i=9$, $c_j=6$, $f_{ji}=3$, $f_{ij}=2$ and $k_{ji} = 3$
we display a part of the module $(e_i+e_j){_j}H_-$ in Figure~\ref{fig:I(i,j,g)}.
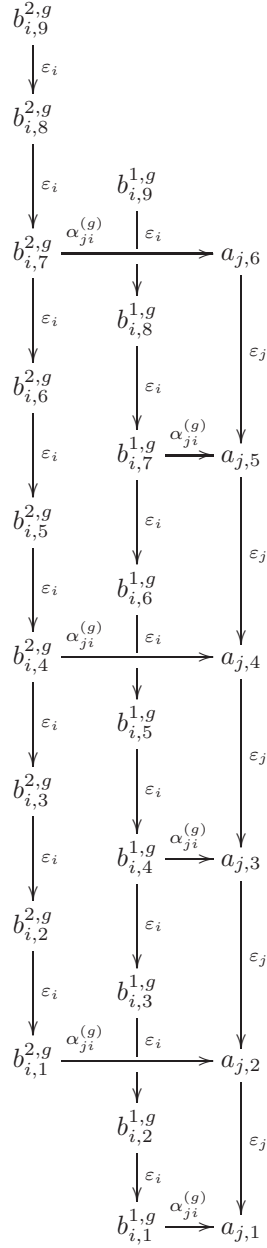
\begin{figure}[!htb]
\small
\[
\xymatrix@-1.5pc{
b_{i,9}^{2,g} \ar[dd]^{\vep_i} &&&&\\
&&&&\\
b_{i,8}^{2,g} \ar[dd]^{\vep_i} &&&&\\
&&  b_{i,9}^{1,g} \ar@{->}'[d]^<<<<{\vep_i}[dd] &&\\
b_{i,7}^{2,g} \ar[rrrr]^<<<<{\alpha_{ji}^{(g)}} \ar[dd]^{\vep_i} &&&& a_{j,6} \ar[ddd]^{\vep_j}\\
&& b_{i,8}^{1,g} \ar[dd]^{\vep_i} &&\\
b_{i,6}^{2,g} \ar[dd]^{\vep_i} &&&&\\
&& b_{i,7}^{1,g} \ar[rr]^{\alpha_{ji}^{(g)}} \ar[dd]^{\vep_i} && a_{j,5} \ar[ddd]^{\vep_j}\\
b_{i,5}^{2,g} \ar[dd]^{\vep_i} &&&&\\
&& b_{i,6}^{1,g} \ar@{->}'[d]^<<<<{\vep_i}[dd]&&\\
b_{i,4}^{2,g} \ar[rrrr]^<<<<{\alpha_{ji}^{(g)}} \ar[dd]^{\vep_i} &&&& a_{j,4} \ar[ddd]^{\vep_j}\\
&& b_{i,5}^{1,g} \ar[dd]^{\vep_i} &&\\
b_{i,3}^{2,g} \ar[dd]^{\vep_i} &&&&\\
&& b_{i,4}^{1,g} \ar[rr]^{\alpha_{ji}^{(g)}} \ar[dd]^{\vep_i} && a_{j,3} \ar[ddd]^{\vep_j}\\
b_{i,2}^{2,g} \ar[dd]^{\vep_i} &&&&\\
&& b_{i,3}^{1,g} \ar@{->}'[d]^<<<<{\vep_i}[dd]&&\\
b_{i,1}^{2,g}\ar[rrrr]^<<<<{\alpha_{ji}^{(g)}} &&&& a_{j,2} \ar[ddd]^{\vep_j}\\
&& b_{i,2}^{1,g} \ar[dd]^{\vep_i} &&\\
&&&&\\
&& b_{i,1}^{1,g} \ar[rr]^{\alpha_{ji}^{(g)}} && a_{j,1}
}
\]
\caption{Construction of the $H$-module $(e_i+e_j){_j}H_-$ with $c_i=9$, $c_j=6$, $f_{ji}=3$, $f_{ij}=2$ and $k_{ji} = 3$.}
\label{fig:I(i,j,g)}
\end{figure}

Recall that the notion of a locally free module can be found in
Definition~\ref{def:locallyfree}.
Let $S_1,\ldots,S_n$ be the simple $H$-modules with
$\dimv(S_i) = \alpha_i$, and let
$E_1,\ldots,E_n$ be the (indecomposable) locally free $H$-modules
with $\rkv(E_i) = \alpha_i$.
(Here $\alpha_1,\ldots,\alpha_n$ is the standard basis of $\Z^n$.)
We refer to the $E_i$ as the \emph{generalized simple} $H$-modules.
Thus $E_i$ corresponds to the regular representation of $H_i$.
More precisely, we have $E_i = e_iE_i$, and 
$e_iE_i$ has a basis
$a_{i,1},\ldots,a_{i,c_i}$ such that
\[
\vep_ia_{i,c} := 
\begin{cases}
a_{i,c-1} & \text{if $c\ge 2$},\\
0 & \text{if $c=1$}.
\end{cases}
\]
In particular, if $i$ is a sink in $Q^\circ(C,\Omega)$, then $E_i = P_i$.
Dually, if $i$ is a source in $Q^\circ(C,\Omega)$, then $E_i = I_i$.

It follows from our construction of $P_i$ and $I_i$ 
that these modules are locally free.
Furthermore, we get the following result, which again follows directly
from our construction.

\begin{Prop}\label{resolutionE}
For every $i\in Q_0$, the canonical exact sequence
\[
0 \to \bigoplus_{j \in \Omega(-,i)} P_j^{|c_{ji}|} \to P_i \to E_i \to 0
\]
is a minimal projective resolution of $E_i$,
and the canonical exact sequence
\[
0 \to E_i \to I_i \to \bigoplus_{j \in \Omega(i,-)} I_j^{|c_{ji}|} \to 0
\]
is a minimal injective resolution of $E_i$.
\end{Prop}

\subsection{Example}
Let 
\[
C = \left(\bbm 2&-1&0\\-2&2&-1\\0&-1&2 \ebm\right) 
\]
with symmetrizer $D = \diag(4,2,2)$ and $\Omega = \{ (2,1), (3,2) \}$.
Thus $C$ is a Cartan matrix of type $C_3$.
Then $H = H(C,D,\Omega)$ is given by the quiver
\[
\xymatrix{
1 \ar@(ul,ur)^{\vep_1} \ar[r]_{\alpha_{21}}& 2 \ar@(ul,ur)^{\vep_2} \ar[r]_{\alpha_{32}}& 3 
\ar@(ul,ur)^{\vep_3}
}
\]
with relations $\vep_1^4 = 0$, $\vep_2^2 = \vep_3^2 = 0$,
$\vep_2\alpha_{21} = \alpha_{21}\vep_1^2$ and
$\vep_3\alpha_{32} = \alpha_{32}\vep_2$.
The indecomposable projective $H$-module $P_1$ and the
indecomposable injective $H$-module $I_3$ are displayed in
Figure~\ref{fig:exampleProjInj}.
The modules $P_2$ and $P_3$ are submodules of $P_1$, and
$I_1$ and $I_2$ are factor modules of $I_3$.
\begin{figure}[!htb]
\small
\[
\xymatrix@-0.5pc{
P_1\df& 1 \ar[d]\ar[r] & 2 \ar@{->}'[d][dd]\ar[r] & 3 \ar@{->}'[d][dd] &&&&
I_3\df & 1 \ar[d]
\\
& 1\ar[d]\ar[rrr] &&& 2 \ar[dd]\ar[r] & 3 \ar[dd] &&
& 1 \ar[d]\ar[r] & 2 \ar[dd]\ar[r] & 3 \ar[dd]
\\
& 1 \ar[d]\ar[r] & 2 \ar[r]  & 3 &&&&
& 1 \ar[d]
\\
& 1 \ar[rrr] &&& 2 \ar[r] & 3 &&
& 1 \ar[r]& 2 \ar[r] & 3 
}
\]
\caption{A projective and an injective $H$-module for type $C_3$} 
\label{fig:exampleProjInj}.
\end{figure}
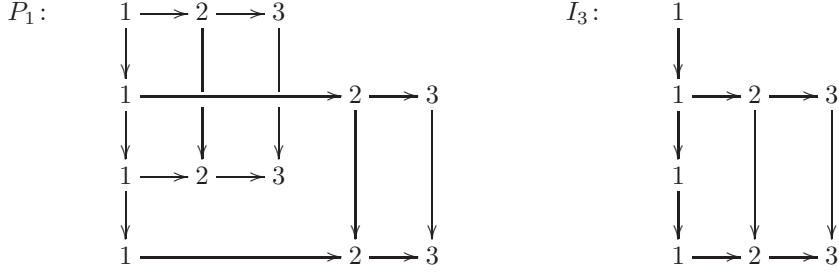
%

\subsection{The rank vectors of 
projective and injective modules}\label{sec3.3}
Assume that $\bi = (i_1,\ldots,i_n)$ is a $+$-admissible sequence
for $(C,\Omega)$.
Without loss of generality, assume that $i_k = k$ for $1 \le k \le n$.
Recall that we defined some positive roots $\beta_k,\gamma_k \in \Delta_\real^+(C)$ in Section~\ref{sec:ct}.

\begin{Lem}\label{dimP_k}
We have $\rkv(P_k) = \beta_k$.
\end{Lem}

\begin{proof}
By our construction of the 
indecomposable projective $H$-modules $P_i$
we get
\begin{align*}
\rkv(P_k) &= \rkv(E_k) + \sum_{j \in \Omega(-,k)} g_{jk}f_{jk} \rkv(P_j) 
\\
&= 
\rkv(E_k) + \sum_{j \in \Omega(-,k)} |c_{jk}| \rkv(P_j).
\end{align*}
For $k=1$ we have $P_k \cong E_k$.
Thus we have $\rkv(P_k) = \alpha_1 = \beta_1$.
For $k \ge 2$ we have
\begin{align*}
\beta_k &= s_1 \cdots s_{k-2}(\alpha_k-c_{k-1,k}\alpha_{k-1})\\
&= s_1 \cdots s_{k-2}(\alpha_k) - c_{k-1,k}\beta_{k-1}\\
&= \alpha_k - \sum_{j \in \Omega(-,k)} c_{jk} \beta_j\\
&= \alpha_k + \sum_{j \in \Omega(-,k)} |c_{jk}| \beta_j
\end{align*}
The claim follows by induction.
\end{proof}

The proof of the next result is similar to the proof of Lemma~\ref{dimP_k}.

\begin{Lem}\label{dimI_k}
We have $\rkv(I_k) = \gamma_k$.
\end{Lem}

As a consequence of Lemmas~\ref{dimP_k} and \ref{dimI_k}
we get the following result.

\begin{Prop}\label{projinjroots}
We have $\rkv(P_i),\rkv(I_i) \in \Delta_\real^+(C)$.
\end{Prop}

\subsection{The Coxeter matrix}\label{subsec-Coxeter}
The
\emph{Cartan matrix} $C_H$ of $H$ is the $(n\times n)$-matrix
with $k$th column the dimension vector $\dimv(P_k)$, $1 \le k \le n$.
(This is not to be confused with the Cartan matrix $C$.)
It follows that the $k$th row of $C_H$ is $\dimv(I_k)$, $1 \le k \le n$,
see for example \cite[Section~2.4, p.70]{Ri3}.
The matrix $C_H$ is invertible over $\Q$ (but not necessarily over
$\Z$). 
(We can choose a numbering of the vertices of $Q(C,\Omega)$ 
such that $C_H$ is an upper triangular matrix with only non-zero entries on 
the diagonal.)
The \emph{Coxeter matrix} of $H$ is defined as
\[
\Phi_H := -C_H^TC_H^{-1}
\]
where $C_H^T$ denotes the transpose of $C_H$.
It follows that
\[
\Phi_H(\dimv(P_k)) = - \dimv(I_k)
\]
(Here we treat $\dimv(P_k)$ as a column vector.)

Next, let $C_{H,P}$ be the $(n\times n)$-matrix
with $k$th column the rank vector $\rkv(P_k)$, 
and let $C_{H,I}$ be the $(n\times n)$-matrix
with $k$th row the rank vector $\rkv(I_k)$, 
$1 \le k \le n$.
We have 
\[
C_{H,P} = D^{-1}C_H
\text{\;\;\; and \;\;\;}
C_{H,I} = C_HD^{-1}.
\]
We get
\[
D^{-1}\Phi_HD = -D^{-1}C_H^TC_H^{-1}D =
- C_{H,I}^TC_{H,P}^{-1}
\]
and this matrix satisfies 
\[
D^{-1}\Phi_HD(\rkv(P_k)) := -\rkv(I_k).
\]
Thus by Lemmas~\ref{dimP_k} and \ref{dimI_k}
we can identify
$D^{-1}\Phi_HD$ with the Coxeter transformation $c^+$.

\subsection{Homological characterization of locally free modules}

\begin{Prop}\label{vpdim}
For $M \in \rep(H)$ the following are equivalent:
\begin{itemize}

\item[(i)]
$\pdim(M) \le 1$;

\item[(ii)]
$\idim(M) \le 1$;

\item[(iii)]
$\pdim(M) < \infty$;

\item[(iv)]
$\idim(M) < \infty$;

\item[(v)]
$M$ is locally free.

\end{itemize}
\end{Prop}

\begin{proof}
For $M= 0$, all properties ${\rm (i)},\ldots,{\rm (v)}$ hold.
Thus we assume that $M$ is non-zero.

Let $M$ be locally free.
Then there exists a vertex $i$ of the quiver $Q(C,\Omega)$
of $H$ such that $e_iM \not= 0$ and $e_jM = 0$ for all
$j \in \Omega(-,i)$.
(Here we used that $Q^\circ(C,\Omega)$ is acyclic.)
It follows that $e_iM$ is a submodule of $M$, and $(1-e_i)M$ is
isomorphic to the factor module $M/e_iM$ of $M$.
So we get a short exact sequence
\[
0 \to e_iM \to M \to (1-e_i)M \to 0
\]
of $H$-modules.
Note that $e_iM$ and $(1-e_i)M$ are both locally free.
If $(1-e_i)M = 0$, then $e_iM = M$.
In this case, we have $M \cong E_i^m$ for some $m \ge 1$,
and Proposition~\ref{resolutionE} yields
that ${\rm (i)},\ldots,{\rm (iv)}$ hold for $M$.
If $(1-e_i)M \not= 0$, then by induction on the dimension we get that
${\rm (i)},\ldots,{\rm (iv)}$ hold for $e_iM$ and for $(1-e_i)M$.
Now one uses long exact homology sequences associated with the short exact sequence above to show that 
${\rm (i)},\ldots,{\rm (iv)}$ also hold for $M$.

Next, assume that $M$ is not locally free.
Let $i$ be a vertex of $Q$ such that $e_iM$ is not a free
$H_i$-module.
Any projective resolution
\begin{equation}\label{resol1}
\cdots \to P_2 \to P_1 \to P_0 \to M \to 0
\end{equation}
of $M$ yields a projective resolution of $H_i$-modules
\begin{equation}\label{resol2}
\cdots \to e_iP_2 \to e_iP_1 \to e_iP_0 \to e_iM \to 0.
\end{equation}
But $H_i$ is a selfinjective algebra, and $e_iM$ is not a 
projective $H_i$-module.
Thus the resolution (\ref{resol2}) and therefore also the resolution (\ref{resol1})
has to be infinite.
This implies $\pdim(M) = \infty$.
Dually, one shows that $\idim(M) = \infty$.
\end{proof}

For a finite-dimensional algebra $A$, let $\tau = \tau_A$ denote its Auslander-Reiten translation. 
Recall that for $X,Y \in \md(A)$ there are functorial
isomorphisms
\[
\Ext_A^1(X,Y) \cong \Du\ov{\Hom}_A(Y,\tau(X)) \cong
\Du\underline{\Hom}_A(\tau^-(Y),X),
\]
see for example \cite[Section~2.4]{Ri3} for details.
These isomorphisms are often referred to as 
\emph{Auslander-Reiten formulas}.
If $\pdim(X) \le 1$, we get a functorial isomorphism
\[
\Ext_A^1(X,Y) \cong \Du\Hom_A(Y,\tau(X)),
\] 
and if
$\idim(Y) \le 1$, then 
\[
\Ext_A^1(X,Y) \cong \Du\Hom_A(\tau^-(Y),X).
\]
Recall that an $A$-module $X$ is $\tau$-\emph{rigid}
(resp. $\tau^-$-\emph{rigid}) if
$\Hom_A(X,\tau(X)) = 0$ (resp. $\Hom_A(\tau^-(X),X) = 0$) \cite{AIR}.
Clearly, if $X$ is $\tau$-rigid or $\tau^-$-rigid, then $X$ is rigid.

\begin{Cor}\label{cor:3.9}
For $M \in \rep_\vp(H)$ the following are equivalent:
\begin{itemize}

\item[(i)]
$M$ is rigid;

\item[(ii)]
$M$ is $\tau$-rigid;

\item[(iii)]
$M$ is $\tau^-$-rigid.
\end{itemize}
\end{Cor}
Combining Propositions~\ref{resolutionE} and \ref{vpdim} yields the following
result.

\begin{Cor}\label{Gorenstein}
The algebra $H$ is a $1$-Iwanaga-Gorenstein algebra.
\end{Cor}

\begin{Lem}\label{vpclosed}
The subcategory $\rep_\vp(H)$ is closed under extensions,
kernels of epimorphisms and cokernels of monomorphisms.
\end{Lem}

\begin{proof}
Let 
\[
0 \to X \xrightarrow{f} Y \xrightarrow{g} Z \to 0 
\] 
be a short exact sequence
in $\rep(H)$.
For each $1 \le i \le n$ this induces a short exact sequence 
\[
0 \to e_iX \to e_iY \to e_iZ \to 0
\]
of $H_i$-modules.
Recall that $M \in \rep(H)$ is locally free if and only if $e_iM$ is
a projective (and therefore also an injective) $H_i$-module for all $i$.
It follows that if any two of the three modules $e_iX$, $e_iY$ and $e_iZ$ are
projective $H_i$-modules, then the third module is also projective as
an $H_i$-modules.
This finishes the proof.
\end{proof}

For the following definitions, see for example \cite{AS}.
Let $A$ be a finite-dimensional $K$-algebra, and let
$\cU$ be a subcategory of $\md(A)$. 
Then $\cU$ is a \emph{resolving subcategory}
if the following hold:
\begin{itemize}

\item[(i)]
${_A}A \in \cU$;

\item[(ii)]
$\cU$ is closed under extensions (i.e. for a short exact sequence 
$0 \to X \to Y \to Z \to 0$ of $A$-modules, 
if $X,Z \in \cU$, then $Y \in \cU$); 

\item[(iii)]
$\cU$ is closed under kernels of epimorphisms.

\end{itemize}
Dually, $\cU$ is \emph{coresolving} if
\begin{itemize}

\item[(i)]
$\Du(A_A) \in \cU$;

\item[(ii)]
$\cU$ is closed under extensions;

\item[(iii)]
$\cU$ is closed under cokernels of monomorphisms.

\end{itemize}
For $X \in \md(A)$
a homomorphism $f\df X \to U$ is a \emph{left $\cU$-approximation} of $X$ if
$U \in \cU$ and
\[
\Hom_A(U,U') \xrightarrow{\Hom_A(f,U')}  \Hom_A(X,U') \to 0 
\]
is exact for all $U' \in \cU$.
Dually,
a homomorphism
$g\df U \to X$ is a \emph{right $\cU$-approximation} of $X$ if
$U \in \cU$ and
\[
\Hom_A(U',U) \xrightarrow{\Hom_A(U',g)} \Hom_A(U',X) \to 0
\]
is exact for all $U' \in \cU$.
The subcategory $\cU$ is 
\emph{covariantly finite} if 
every $X \in \md(A)$ has a left $\cU$-approximation.
Dually, $\cU$ is \emph{contravariantly finite} if
every $X \in \md(A)$ has a right $\cU$-approximation.
Finally, $\cU$ is \emph{functorially finite}
if $\cU$ is covariantly and contravariantly finite.

\begin{Thm}
The subcategory $\rep_\vp(H)$ is resolving, coresolving and
functorially finite.
In particular, $\rep_\vp(H)$ has Auslander-Reiten sequences.
\end{Thm}

\begin{proof}
By Lemma~\ref{vpclosed} and Proposition~\ref{resolutionE} we get that
$\rep_\vp(H)$ is a resolving and coresolving subcategory of $\rep(H)$.
Furthermore, by Proposition~\ref{vpdim} we know that $\rep_\vp(H)$ coincides
with the subcategory of all $H$-modules with projective dimension $1$.
Thus $\rep_\vp(H)$ is covariantly finite by \cite[Proposition~4.2]{AR}.
Since $\rep_\vp(H)$ also coincides
with the subcategory of all $H$-modules with injective dimension $1$, the dual
of \cite[Proposition~4.2]{AR} yields that
$\rep_\vp(H)$ is contravariantly finite.
Thus $\rep_\vp(H)$ is functorially finite in $\rep(H)$.
Now it follows from \cite[Theorem~2.4]{AS} that $\rep_\vp(H)$ has
Auslander-Reiten sequences. 
\end{proof}


\section{The homological bilinear form}\label{sec4}


As before, let
$H  = H(C,D,\Omega)$.
For $M,N \in \rep_\vp(H)$ define
\begin{align*}
\bil{M,N}_H &:= \dim \Hom_H(M,N) - \dim \Ext_H^1(M,N),\\
(M,N)_H &:= \bil{M,N}_H + \bil{N,M}_H,\\
q_H(M) &:= \bil{M,M}_H.
\end{align*}

\begin{Prop}\label{quadratic1}
For $M,N \in \rep_\vp(H)$ we have
\[
\bil{M,N}_H = \sum_{i=1}^n c_ia_ib_i - \sum_{(j,i) \in \Omega} c_i|c_{ij}| a_ib_j
\]
where
$\rkv(M) = (a_1,\ldots,a_n)$ and
$\rkv(N) = (b_1,\ldots,b_n)$.
\end{Prop}

\begin{proof}
Let $Q = Q(C,\Omega)$.
Let $i_1$ be a sink in $Q^\circ$, and let $i_n$ be a source of $Q^\circ$.
We get short exact sequences
\begin{equation}\label{seq1}
0 \to E_{i_1}^{a_{i_1}} \xrightarrow{f_1} M \xrightarrow{f_2} M' \to 0
\end{equation}
and
\begin{equation}\label{seq2}
0 \to N' \xrightarrow{g_1} N \xrightarrow{g_2} E_{i_n}^{b_{i_n}} \to 0
\end{equation}
where $f_1$ is the obvious canonical inclusion, $f_2$ is the canonical
projection onto $\Coker(f_1)$, $g_2$ is the obvious canonical projection, and
$g_1$ is the canonical inclusion of $\Ker(g_2)$.
Applying $\Hom_H(-,N)$ to sequence (\ref{seq1}) and $\Hom_H(M,-)$ to
the sequence (\ref{seq2}) yields the long exact cohomology sequences
\begin{multline}
0 \to \Hom_H(M',N) \to \Hom_H(M,N) \to \Hom_H(E_{i_1}^{a_{i_1}},N) \\
\to \Ext_H^1(M',N) \to \Ext_H^1(M,N) \to \Ext_H^1(E_{i_1}^{a_{i_1}},N) \to 0
\end{multline}
and
\begin{multline}
0 \to \Hom_H(M,N') \to \Hom_H(M,N) \to \Hom_H(M,E_{i_n}^{b_{i_n}}) \\
\to  \Ext_H^1(M,N') \to \Ext_H^1(M,N) \to \Ext_H^1(M,E_{i_n}^{b_{i_n}}) \to 0.
\end{multline}
For the exactness of the first cohomology sequence we used that $\idim(N) \le 1$, and for the 
second sequence we needed that $\pdim(M) \le 1$, compare
Proposition~\ref{vpdim}.
The first sequence implies that 
\[
\bil{M,N}_H = \bil{M',N}_H + \bil{E_{i_1}^{a_{i_1}},N}_H,
\]
and the second sequence  yields
\[
\bil{M,N}_H = \bil{M,N'}_H + \bil{M,E_{i_n}^{b_{i_n}}}_H.
\]
Thus by induction we get
\[
\bil{M,N}_H = \sum_{1 \le i,j \le n} a_ib_j \bil{E_i,E_j}_H.
\]
For $1 \le j \le n$ we have 
\[
\dim \Hom_H(E_i,E_j) = \dim \Hom_H(P_i,E_j) = 
\begin{cases}
c_i & \text{if $i=j$},\\
0 & \text{otherwise}.
\end{cases}
\]
Recall that the minimal projective resolution of $E_i$ has the form
\begin{equation}\label{seq3}
0 \to \bigoplus_{j \in \Omega(-,i)} P_j^{|c_{ji}|} \to P_i \to E_i \to 0.
\end{equation}
Applying $\Hom_H(-,E_j)$ for $1 \le j \le n$ yields
\[
\dim \Ext_H^1(E_i,E_j) = 
\begin{cases}
c_j|c_{ji}| & \text{if $j \in \Omega(-,i)$},\\
0 & \text{otherwise}.
\end{cases}
\]
Since $c_jc_{ji} = c_ic_{ij}$,
the result follows.
\end{proof}

Proposition~\ref{quadratic1} shows that for $M, N \in \rep_\vp(H)$ the 
number $\bil{M,N}_H$ depends only on the rank vectors 
$\rkv(M)$ and $\rkv(N)$. This implies:

\begin{Cor}\label{quadratic2}
The map $(M,N) \mapsto \bil{M,N}_H$ descends to the Grothendieck group
$\Z^n$ of $\rep_\vp(H)$ and induces a bilinear form 
$\Z^n \times \Z^n \to \Z$ still denoted by $\bil{-,-}_H$.
This bilinear form is characterized 
by $\bil{\alpha_i,\alpha_j}_H = \bil{E_i,E_j}_H$,
where $\alpha_1,\ldots,\alpha_n$ is the standard basis of $\Z^n$.
\end{Cor}

Let 
\[
(-,-)_H\df \Z^n \times \Z^n \to \Z
\] 
be the symmetrization of $\bil{-,-}_H$ defined by
$(a,b)_H := \bil{a,b}_H + \bil{b,a}_H$, and let
$q_H\df \Z^n \to \Z$ be the quadratic form defined by
$q_H(a) := \bil{a,a}_H$.
The forms $q_H$ and $\bil{-,-}_H$ are called the \emph{homological bilinear 
forms} of $H$.

\begin{Cor}\label{quadratic3}
We have $q_H = q_C$ and $(-,-)_H = (-,-)_C$.
\end{Cor}

\begin{proof}
By definition we have
\[
q_C := \sum_{i=1}^n c_i X_i^2 - \sum_{i<j} c_i|c_{ij}| X_iX_j, 
\]
and we know from Proposition~\ref{quadratic1} that
\[
q_H = \sum_{i=1}^n c_iX_i^2 - \sum_{(j,i) \in \Omega} c_i|c_{ij}| X_iX_j.
\]
Note that $q_H$ does not depend on the orientation $\Omega$, since 
$c_ic_{ij} = c_jc_{ji}$ for all $i,j$.
Thus we have
$q_H = q_C$.
Similarly, one also shows easily that $(-,-)_H = (-,-)_C$.
\end{proof}


\section{An analogy to the representation theory of 
modulated graphs}\label{sec6}


The constructions and results of this section form a crucial part 
of this article.
For example, it contains the foundation for defining reflection functors and Coxeter functors for the algebras $H(C,D,\Omega)$.

\subsection{The bimodules ${_j}H_i$}\label{sec5.1}
Let $C = (c_{ij}) \in M_n(\Z)$ be a Cartan matrix with symmetrizer 
$D = \diag(c_1,\ldots,c_n)$, and
let $\Omega$ be an orientation of $C$, and let $\Omega^*$ be the opposite orientation.
Let $H := H(C,D,\Omega)$ and $H^* := H(C,D,\Omega^*)$.
Recall that for $1 \le i \le n$ we have 
\[
H_i := e_iHe_i = K[\vep_i]/(\vep_i^{c_i}).
\]
In the following we write $\otimes_i$ for a tensor product $\otimes_{H_i}$
over $H_i$.
If there is no danger of misunderstanding, we also just write $\otimes$
instead of $\otimes_i$.

For $(j,i) \in \Omega$ we define
\[
{_j}H_i := H_j {\rm Span}_K(\alpha_{ji}^{(g)} \mid 1 \le g \le g_{ji})H_i 
= {\rm Span}_K(\vep_j^{f_j}\alpha_{ji}^{(g)}\vep_i^{f_i} \mid
f_j,f_i \ge 0, 1 \le g \le g_{ji}).
\]
Our considerations in Section~\ref{sec3.1} show that
${_j}H_i$
is an $H_j$-$H_i$-bimodule, which is free as a left $H_j$-module
and free as a right $H_i$-module.
Let ${_i}H_j$ be the corresponding $H_i$-$H_j$-bimodule
coming from $H^*$. 
We get
\[
{_j}H_i = \bigoplus_{g=1}^{g_{ij}} \bigoplus_{f=0}^{f_{ji}-1} H_j(\alpha_{ji}^{(g)}\vep_i^f)
= \bigoplus_{g=1}^{g_{ij}}\bigoplus_{f=0}^{f_{ij}-1} (\vep_j^f\alpha_{ji}^{(g)})H_i
\]
and 
\[
{_i}H_j = \bigoplus_{g=1}^{g_{ij}} \bigoplus_{f=0}^{f_{ij}-1} H_i(\alpha_{ij}^{(g)}\vep_j^f)
= \bigoplus_{g=1}^{g_{ij}}\bigoplus_{f=0}^{f_{ji}-1} (\vep_i^f\alpha_{ij}^{(g)})H_j.
\]
So we have
\begin{align*}
{_{H_j}}({_j}H_i) &\cong H_j^{|c_{ji}|} \cong ({_i}H_j)_{H_j},
\\
{_{H_i}}({_i}H_j) &\cong H_i^{|c_{ij}|} \cong ({_j}H_i)_{H_i}.
\end{align*}
Define
\begin{align*}
{_j}L_i &:= \{ \alpha_{ji}^{(g)}, \alpha_{ji}^{(g)}\vep_i,\ldots,
\alpha_{ji}^{(g)}\vep_i^{f_{ji}-1} \mid 1 \le g \le g_{ij} \},
\\
{_i}L_j &:= \{ \alpha_{ij}^{(g)}, \alpha_{ij}^{(g)}\vep_j,\ldots,
\alpha_{ij}^{(g)}\vep_j^{f_{ij}-1} \mid 1 \le g \le g_{ij} \},
\\
{_j}R_i &:= \{ \alpha_{ji}^{(g)}, \vep_j\alpha_{ji}^{(g)},\ldots,
\vep_j^{f_{ij}-1}\alpha_{ji}^{(g)} \mid 1 \le g \le g_{ij} \},
\\
{_i}R_j &:= \{ \alpha_{ij}^{(g)}, \vep_i\alpha_{ij}^{(g)},\ldots,
\vep_i^{f_{ji}-1}\alpha_{ij}^{(g)} \mid 1 \le g \le g_{ij} \}.
\end{align*}
Then ${_j}L_i$ (resp. ${_j}R_i$) is a basis of ${_j}H_i$
as a left $H_j$-modules (resp. as a right $H_i$-module).
We have
$|{_j}L_i| = |{_i}R_j| = |c_{ji}|$ and
$|{_i}L_j| = |{_j}R_i| = |c_{ij}|$.

Let $({_j}L_i)^*$ and $({_j}R_i)^*$ be the dual basis of $\Hom_{H_j}({_j}H_i,H_j)$ and $\Hom_{H_i}({_j}H_i,H_i)$, respectively.
For $b \in {_j}L_i$ or $b \in {_j}R_i$ let $b^*$ be 
the corresponding dual basis vector.
Similarly, define $({_i}L_j)^*$ and $({_i}R_j)^*$.

There is an $H_i$-$H_j$-bimodule isomorphism
\[
\rho\df {_i}H_j \to \Hom_{H_j}({_j}H_i,H_j) 
\]
given by
\[
\rho\left(\vep_i^{f_{ji}-1-f}\alpha_{ij}^{(g)}\right) = (\alpha_{ji}^{(g)}\vep_i^f)^* 
\]
for $0 \le f \le f_{ji}-1$ and $1 \le g \le g_{ij}$.
Indeed, for the left $H_i$-module structure on $\Hom_{H_j}({_j}H_i,H_j)$ 
one has 
\[
\vep_i \cdot (\alpha_{ji}^{(g)}\vep_i^f)^* = 
\begin{cases}
(\alpha_{ji}^{(g)}\vep_i^{f-1})^*& \text{if $f>0$,}\\[2mm]
(\alpha_{ji}^{(g)}\vep_i^{f_{ji}-1})^*\cdot \vep_j^{f_{ij}} & \text{if $f=0$.}
\end{cases}
\]
Similarly there is an $H_i$-$H_j$-bimodule isomorphism
\[
\lambda\df {_i}H_j \to \Hom_{H_i}({_j}H_i,H_i)
\]
given by
\[
\lambda\left(\alpha_{ij}^{(g)}\vep_j^{f_{ij}-1-f}\right) = 
(\vep_j^f\alpha_{ji}^{(g)})^* 
\]
for $0 \le f \le f_{ij}-1$ and $1 \le g \le g_{ij}$.
In particular, we get $ \rho({_i}R_j) = ({_j}L_i)^*$ and
$\lambda({_i}L_j) = ({_j}R_i)^*$.
In the following, we sometimes identify the spaces 
$\Hom_{H_j}({_j}H_i,H_j)$, ${_i}H_j$ and $\Hom_{H_i}({_j}H_i,H_i)$ via
$\rho$ and $\lambda$.
For example, for $b \in {_j}L_i$, we consider 
$b^* \in \Hom_{H_j}({_j}H_i,H_j)$ as an element in ${_i}H_j$.

If $N_j$ is an $H_j$-module, then we have a natural isomorphism of $H_i$-modules
\[
\Hom_{H_j}({_j}H_i,N_j) \to {_i}H_j \otimes_j N_j
\]
defined by 
\[
f \mapsto \sum_{b \in {_j}L_i} b^* \otimes_j f(b).
\]
Now, if in addition $M_i$ is an $H_i$-module, the adjunction map gives an 
isomorphism of $K$-vector spaces:
\[
\Hom_{H_j}({_j}H_i\otimes_i M_i, N_j) \to \Hom_{H_i}(M_i, \Hom_{H_j}({_j}H_i,N_j)). 
\]
Combining these two maps we get a functorial isomorphism of $K$-vector spaces
\[
\ad_{ji} := \ad_{ji}(M_i,N_j)\df 
\Hom_{H_j}({_j}H_i \otimes_i M_i,N_j) \to  \Hom_{H_i}(M_i,{_i}H_j \otimes_j N_j)
\]
given by 
\[
f \mapsto \left(f^\vee\df m \mapsto \sum_{b \in {_j}L_i} b^* \otimes_j f(b \otimes_i m)\right).
\]
The inverse $\ad_{ji}^{-1}$ of $\ad_{ji}$ is given by
\[
g \mapsto \left(g^\vee\df h \otimes_i m \mapsto \sum_{b \in {_j}L_i} b^*(h) (g(m))_b\right)
\]
where the elements  
$(g(m))_b \in N_j$ are uniquely determined by
\[
g(m) = \sum_{b \in {_j}L_i} b^* \otimes_j (g(m))_b. 
\]
Here we used that each element $x$ in ${_i}H_j \otimes_j N_j$ can be written 
uniquely as a sum of the form 
\[
x = \sum_{b \in {_j}L_i} b^* \otimes_j x_b.
\]

\subsection{Representation theory of modulated graphs}\label{subsec5.2}
The tuple
$(H_i,{_i}H_j,{_j}H_i)$ defined in Section~\ref{sec5.1}
is called a \emph{modulation} of $C$
and is denoted by $\cM(C,D)$.

For an orientation $\Omega$ of $C$, a \emph{representation}
$M = (M_i,M_{ij})$ 
of $(\cM(C,D),\Omega)$ is given by a finite-dimensional
$H_i$-module $M_i$ for each $1 \le i \le n$ and an $H_i$-linear map
\[
M_{ij}\df {_i}H_j \otimes_j M_j \to M_i
\]
for each $(i,j) \in \Omega$.
A morphism $f\df M \to N$ of representations $M = (M_i,M_{ij})$ and 
$N = (N_i,N_{ij})$ of $(\cM(C,D),\Omega)$ is a tuple $f = (f_i)_i$ of 
$H_i$-linear maps
$f_i\df M_i \to N_i$ for $1 \le i \le n$ such that for each $(i,j) \in \Omega$ 
the diagram
\[
\xymatrix{
{_i}H_j \otimes_j M_j \ar[d]^{M_{ij}} \ar[r]^{1 \otimes_j f_j} & 
{_i}H_j \otimes_j N_j \ar[d]^{N_{ij}} 
\\
M_i \ar[r]^{f_i} & N_i
}
\]
commutes.
One easily checks that the representations of $(\cM(C,D),\Omega)$ form
an abelian category $\rep(C,D,\Omega)$.

For $(M_i,M_{ij}) \in \rep(C,D,\Omega)$
define a representation 
\[
(M_i,M(\alpha_{ij}^{(g)}),M(\vep_i))
\] 
of
$H(C,D,\Omega)$ as follows:
Define a $K$-linear map $M(\vep_i)\df M_i \to M_i$ by
\[
M(\vep_i)(m) := \vep_i m.
\]
(Here we use that $M_i$ is an $H_i$-module.)
Let $(i,j) \in \Omega$.
Recall that ${_i}H_j$ has an $H_i$-basis
\[
{_i}L_j = \{ \alpha_{ij}^{(g)}, \alpha_{ij}^{(g)}\vep_j,\ldots,
\alpha_{ij}^{(g)}\vep_j^{f_{ij}-1} \mid 1 \le g \le g_{ij} \}.
\]
Define a $K$-linear map $M(\alpha_{ij}^{(g)})\df M_j \to M_i$ by
\[
M(\alpha_{ij}^{(g)})(m) := M_{ij}(\alpha_{ij}^{(g)} \otimes_j m).
\]
Now one can check that the relations (H1) and (H2) are satisfied.
In other words, $(M_i,M(\alpha_{ij}^{(g)}),M(\vep_i))$ is a
representation of $H(C,D,\Omega)$.

Conversely, let $(M_i,M(\alpha_{ij}^{(g)}),M(\vep_i))$ be  a
representation of $H(C,D,\Omega)$.
Note that $M_i$ is an $H_i$-module via the map $M(\vep_i)$.
For $(i,j) \in \Omega$ define an $H_i$-linear map
\[
M_{ij}\df {_i}H_j \otimes_j M_j \to M_i
\]
by 
\[
M_{ij}(\alpha_{ij}^{(g)}\vep_j^f \otimes m) := 
(M(\alpha_{ij}^{(g)}) \circ M(\vep_j)^f)(m).
\]
Then $(M_i,M_{ij}) \in \rep(C,D,\Omega)$.

These two constructions yield obviously mutually inverse bijections between the
representations of $(\cM(C,D),\Omega)$ and $H(C,D,\Omega)$.
It is also clear how to associate to a morphism in $\rep(C,D,\Omega)$ a morphism
in $\rep(H(C,D,\Omega))$ and vice versa.
Now it is straightforward to verify the following statement.

\begin{Prop}
The categories $\rep(C,D,\Omega)$ and $\rep(H(C,D,\Omega))$ are isomorphic. 
\end{Prop}

Thus the representation theory of the algebras $H(C,D,\Omega)$
shows a striking analogy to the representation theory of modulated
graphs in the sense of Dlab and Ringel \cite{DR1}.
The main difference is that in Dlab and Ringel's theory, the 
rings $H_i$ would be division rings, whereas in our case they are
commutative symmetric algebras, or more precisely, truncations of polynomial
rings.
Generalizations of the representation theory of modulated graphs
have been formulated already in \cite{Li}.

\subsection{}\label{Minn}
Next, we want to interpret the category $\rep(\Pi(C,D))$ of finite-dimensional
representations of $\Pi(C,D)$ as a category of representations of modulated graphs.
Let $\rep(C,D,\overline{\Omega})$ be the category with objects
$M = (M_i,M_{ij},M_{ji})$ with $(i,j) \in \Omega$
such that
$(M_i,M_{ij}) \in \rep(C,D,\Omega)$
and $(M_i,M_{ji}) \in \rep(C,D,\Omega^*)$.
Given two such objects $M$ and $N$ a tuple $f = (f_i)_i$ is a homomorphism $f\df M \to N$ if 
$f$ is both a homomorphism $(M_i,M_{ij}) \to (N_i,N_{ij})$ in
$\rep(C,D,\Omega)$ and a 
homomorphism $(M_i,M_{ji}) \to (N_i,N_{ji})$ in
$\rep(C,D,\Omega^*)$.

For an object $M = (M_i,M_{ij},M_{ji})$ in $\rep(C,D,\overline{\Omega})$ let
\[
M_{i,\inn} := (\sgn(i,j)M_{ij})_j\df \bigoplus_{j \in \overline{\Omega}(i,-)} {_i}H_j \otimes M_j \to M_i 
\]
and 
\[
M_{i,\out} := (M_{ji}^\vee)_j\df 
M_i \to \bigoplus_{j \in \overline{\Omega}(-,i)} {_i}H_j \otimes M_j.
\]
These are both $H_i$-module homomorphisms.
(Recall that $M_{ji}^\vee = \ad_{ji}(M_{ji})$, see Section~\ref{sec5.1}.)
Set
\[
\widetilde{M_i} := \bigoplus_{j \in \overline{\Omega}(i,-)} {_i}H_j \otimes M_j.
\]
Since $\overline{\Omega}(i,-) = \overline{\Omega}(-,i)$, we have
\[
\bigoplus_{j \in \overline{\Omega}(i,-)} {_i}H_j \otimes M_j
=
\bigoplus_{k \in \overline{\Omega}(-,i)} {_i}H_k \otimes M_k.
\]
Thus we get a diagram
\[
\widetilde{M_i} \xrightarrow{M_{i,\inn}} 
M_i \xrightarrow{M_{i,\out}} \widetilde{M_i}.
\]

\begin{Prop}\label{prop3.3}
The category $\rep(\Pi(C,D))$ is isomorphic to the full subcategory of
$\rep(C,D,\overline{\Omega})$ with objects 
$M = (M_i,M_{ij},M_{ji})$ such that 
\[
M_{i,\inn} \circ M_{i,\out} = 0
\]
for all $i$.
\end{Prop}

\begin{proof}
For an object $M = (M_i,M_{ij},M_{ji})$ in $\rep(C,D,\overline{\Omega})$,
the composition 
\[
M_{i,\inn} \circ M_{i,\out} = 
\sum_{j \in \overline{\Omega}(-,i)} \sgn(i,j)M_{ij} \circ M_{ji}^\vee
\]
is in $\End_{H_i}(M_i)$ and maps an element $m \in M_i$ to
\[
M_{i,\inn} \circ M_{i,\out}(m) = 
\sum_{j \in \overline{\Omega}(-,i)} 
\sgn(i,j)\sum_{b \in {_j}L_i} M_{ij}\left(b^* \otimes_j M_{ji}(b \otimes_i m)\right).
\]
Let $b \in {_j}L_i$.
Thus we have
$b = \alpha_{ji}^{(g)}\vep_i^{f_{ji}-1-f}$ for some $0 \le f \le f_{ji}-1$.
This implies that 
$b^* = \vep_i^f\alpha_{ij}^{(g)} \in {_i}R_j$.
It follows that
\[
\sgn(i,j)M_{ij}\left(b^* \otimes_j M_{ji}(b \otimes_i m)\right)
=
\sgn(i,j)M(\vep_i)^fM(\alpha_{ij}^{(g)})M(\alpha_{ji}^{(g)})M(\vep_i)^{f_{ji}-1-f}(m).
\]
In view of the defining relation (P3) of $\Pi(C,D)$, this yields the result.
\end{proof}


\section{The algebras $H$ and $\Pi$ are tensor algebras}
\label{sec7}

Let $A$ be a $K$-algebra, and let $M = {_A}M_A$ be an $A$-$A$-bimodule.
The \emph{tensor algebra} $T_A(M)$ is defined as
\[
T_A(M) := \bigoplus_{k \ge 0} M^{\otimes k}
\]
where $M^0 := A$, and 
$M^{\otimes k}$ is the $k$-fold tensor product of $M$
for $k \ge 1$.
The multiplication of $T_A(M)$ is defined as follows:
For $r,s \ge 1$, $m_i,m_i' \in M$ and $a,a' \in A$ let 
\[
(m_1 \otimes \cdots \otimes m_r) \cdot (m_1' \otimes \cdots \otimes m_s') := (m_1 \otimes \cdots \otimes m_r \otimes m_1' \otimes \cdots \otimes m_s')
\]
and 
\[
a(m_1 \otimes \cdots \otimes m_r)a' := (am_1 \otimes \cdots \otimes m_ra').
\]
Recall that the modules over a tensor algebra $T_A(M)$ are
given by the $A$-module homomorphisms
$M \otimes_A X \to X$, where
$X$ is an $A$-module. 

Let $A$ be a $K$-algebra, $A_0$ a subalgebra and $A_1$ an
$A_0$-$A_0$-subbimodule of $A$.
Following \cite{BSZ} we say that
$A$ is \emph{freely generated by $A_1$ over $A_0$}
if the following holds:
For every $K$-algebra $B$ and any pair $(f_0,f_1)$ with
$f_0\df A_0 \to B$ an algebra homomorphism, and
$f_1\df A_1 \to B$ an $A_0$-$A_0$-bimodule homomorphism
(with the $A_0$-$A_0$-bimodule structure on $B$ given by $f_0$)
there exists a unique $K$-algebra homomorphism $f\df A \to B$
which extends $f_0$ and $f_1$.
The following two lemmas can be found in \cite[Section~1]{BSZ}.

\begin{Lem}\label{bimodule1}
For any $K$-algebra $A$ and any $A$-$A$-bimodule $M$ the
tensor algebra $T_A(M)$ is freely generated by $M$ over $A$.
\end{Lem}

\begin{Lem}\label{bimodule2}
Let $A$ be a $K$-algebra which is freely generated by $A_1$ over $A_0$.
Then $A$ is isomorphic to the tensor algebra $T_{A_0}(A_1)$.
\end{Lem}

Let $Q$ be a finite quiver, and let
$w\df Q_1 \to \{0,1\}$ be a map assigning to each arrow
of $Q$ a \emph{degree}.
Then the path algebra $KQ$ is naturally $\N$-graded:
Each path gets as degree the sum of the degrees of its arrows.
By definition the paths of length $0$ have degree $0$.
Let $r_1,\ldots,r_m$ be a set of relations for $KQ$ which are
homogeneous with respect to this grading.
Suppose that there is some $1 \le l \le m$ such that 
$\deg(r_i) = 0$ for $1 \le i \le l$ and
$\deg(r_j) = 1$ for $l+1 \le j \le m$.

Let $A := KQ/I$, where $I$ is the ideal generated by $r_1,\ldots,r_m$.
Clearly, $A$ is again $\N$-graded.
Let $A_i$ be the subspace of elements with degree $i$.
Observe that $A_1$ is naturally an $A_0$-$A_0$-bimodule.
Now Lemmas~\ref{bimodule1},\ref{bimodule2} yield the following result.

\begin{Prop}\label{bimodule3}
$A$ is isomorphic to the tensor algebra $T_{A_0}(A_1)$.
\end{Prop}

As before, let $H = H(C,D,\Omega)$.
Define
\[
S := \prod_{i=1}^n H_i
\text{\;\;\; and \;\;\;}
B := \bigoplus_{(i,j) \in \Omega} {_i}H_j.
\]
Clearly, $B$ is an $S$-$S$-bimodule.

\begin{Prop}\label{prop:tensor1}
$H \cong T_S(B)$.
\end{Prop}

\begin{proof}
The algebra $H$ is graded by defining $\deg(\vep_i) := 0$ and
$\deg(\alpha_{ij}^{(g)}) := 1$ for all $(i,j) \in \Omega$ and
all $g$.
The defining relations for $H$ are homogeneous,
$S$ is the subalgebra of elements of degree $0$, and
$B$ is the subspace of elements of degree $1$.
Now we can apply Proposition~\ref{bimodule3}.
\end{proof}

Let $\Pi = \Pi(C,D,\Omega)$ be the preprojective algebra.
Define $\deg(\vep_i) := 0$ for all $i$, and for $(i,j) \in \Omega$ let
$\deg(\alpha_{ij}^{(g)}) := 0$ and $\deg(\alpha_{ji}^{(g)}) := 1$
for all $g$.
Let 
\[
\Pi_1 := \Pi(C,D,\Omega)_1
\] 
be the subspace of $\Pi$ consisting of the elements of 
degree $1$.
Note that $\Pi_1$ is an $H$-$H$-bimodule.
Again we can apply Proposition~\ref{bimodule3} and get the following
result.

\begin{Prop}\label{prop:7.5}
$\Pi \cong T_H(\Pi_1)$.
\end{Prop}

Define
\[
\overline{B} := \bigoplus_{(i,j) \in \overline{\Omega}} {_i}H_j.
\]
Next, for $1 \le i \le n$ let
\[
\rho_i := \sum_{j \in \overline{\Omega}(-,i)} \sgn(i,j)\sum_{b \in {_j}L_i}
b^* \otimes_j b \in T_S(\ov{B}).
\]
Every $b \in {_j}L_i$ is of the form $b = \alpha_{ji}^{(g)}\vep_i^{f_{ji}-1-f}$ 
for some $0 \le f \le f_{ji}-1$ and $1 \le g \le g_{ij}$.
Then $b^* \in {_i}R_j$ is equal to $\vep_i^f\alpha_{ij}^{(g)}$.
Thus $\rho_i$ translates to the defining relation (P3)
\[
\sum_{j\in \ov{\Omega}(-,i)} \sum_{g=1}^{g_{ij}} \sum_{f=0}^{f_{ji}-1} \sgn(i,j)
\vep_i^f\alpha_{ij}^{(g)} \alpha_{ji}^{(g)} \vep_i^{f_{ji}-1-f} = 0.
\]
of $\Pi$.

The algebra $T_S(\overline{B})/(\rho_1,\ldots,\rho_n)$ is an
analogue of Dlab and Ringel's \cite{DR3} 
definition of a preprojective algebra
of a modulated graph.

\begin{Prop}
$\Pi \cong T_S(\overline{B})/(\rho_1,\ldots,\rho_n)$.
\end{Prop}

\begin{proof}
Similarly as in the proof of Proposition~\ref{prop:tensor1} one shows
that $T_S(\ov{B})$ is isomorphic to the path algebra
$K\ov{Q}$ modulo the defining relations (P1) and (P2) of $\Pi$.

Let $M$ be a module over the tensor algebra
$T_S(\overline{B})$. 
Then $M$ is defined by the structure maps
\[
M_{ij}\df {_i}H_j \otimes_j M_j \to M_i
\]
for each $(i,j) \in \overline{\Omega}$.
This yields maps
\[
M_{iji} := M_{ij} \circ (\id_{{_i}H_j} \otimes_j M_{ji})\df 
{_i}H_j \otimes_j {_j}H_i \otimes_i M_i \to M_i.
\]
Now $M$ is a module over $T_S(\overline{B})/(\rho_1,\ldots,\rho_n)$ if and only
if 
for each vertex $i$ and each $m \in M_i$ we have
$\rho_im = 0$.
This is equivalent to 
\[
\sum_{j \in \overline{\Omega}(-,i)} \sgn(i,j)M_{iji}(\sum_{b\in {_j}L_i} b^* \otimes_j b \otimes_i m) 
= 0.
\]
It follows from the definitions that
\[
\sum_{j \in \overline{\Omega}(-,i)} \sgn(i,j)M_{iji}(\sum_{b\in {_j}L_i} b^* \otimes_j b \otimes_i m) =
(M_{i,\inn} \circ M_{i,\out})(m).
\]
Now Proposition~\ref{prop3.3} yields the result. 
\end{proof}


\section{Projective resolutions of $H$-modules}
\label{newsec7}


\begin{Prop}\label{prop:resolH}
We have a short exact sequence
of $H$-$H$-bimodules
\begin{equation}
P_\bullet: \;\;\;\;\;\;
0 \to \bigoplus_{(j,i) \in \Omega} He_j \otimes_j {_j}H_i \otimes_i e_iH
\xrightarrow{d} \bigoplus_{k=1}^n He_k \otimes_k e_kH 
\xrightarrow{{\rm mult}} H \to 0
\end{equation}
where 
\[
d(p \otimes_j h \otimes_i q) := ph \otimes_i q - p \otimes_j hq.
\]
\end{Prop}

\begin{proof}
We know that $H = T_S(B)$.
The sequence $P_\bullet$ is isomorphic to 
the sequence
\[
0 \to H \otimes_S B \otimes_S H \xrightarrow{d} H \otimes_S H \xrightarrow{{\rm mult}} H
\to 0
\]
of $H$-$H$-bimodules,
where
$d(h \otimes b \otimes h') := (hb \otimes h' - h \otimes bh')$.
Now the statement follows from
\cite[Theorems~10.1 and 10.5]{Sch}.
\end{proof}

The components of $P_\bullet$ are projective as left $H$-modules and as right 
$H$-modules.
However, the components are not projective as $H$-$H$-bimodules.
(Except, if $S$ is semisimple, then the first two components are in fact
projective bimodules.)
In any case, viewed as a short exact sequence of left or right modules, 
$P_\bullet$ splits as an exact sequence of projective modules.

\begin{Cor}\label{cor:resolHM}
If $M \in \rep_\vp(H)$, then $P_\bullet \otimes_H M$ is a projective
resolution of $M$.
Explicitly, $P_\bullet \otimes_H M$ looks as follows
\begin{equation}
0 \to
\bigoplus_{(j,i) \in \Omega} H e_j \otimes_j {_j}H_i \otimes_i M_i
\xrightarrow{d \otimes M}
\bigoplus_{k=1}^n He_k \otimes_k M_k \xrightarrow{{\rm mult}} M \to 0
\end{equation}
where
\[
(d \otimes M)(p \otimes_j h \otimes_i m) = ph \otimes_i m - p \otimes_j 
M_{ji}(h \otimes_i m).
\]
(Here $M_{ji}\df {_j}H_i \otimes_i M_i \to M_j$ is the $H_j$-linear structure map
of $M$ associated with $(j,i) \in \Omega$.)
\end{Cor}

\begin{proof}
By the remarks above, $P_\bullet \otimes_H M$ is always exact.
If $M$ is locally free, then 
$e_kH \otimes_H M = e_kM$ and 
${_j}H_i \otimes_i e_iH \otimes_H M = {_j}H_i \otimes_i M_i$
are free $H_k$- resp. $H_j$-modules.
Thus the relevant components of $P_\bullet \otimes_H M$ are indeed
projective. 
\end{proof}


\section{The trace pairing}\label{newsec8}


\subsection{The trace pairing for homomorphisms between free $H_i$-modules}
\label{sec:trace1}
For each $i=1,2,\ldots, n$ we have the $K$-linear map 
\[
t_i^{\rm max}\df H_i \to K
\]
defined by 
\[
\sum_{j=0}^{c_i-1} \lambda_j\vep_i^j \mapsto \lambda_{c_i-1}.
\]
For  free $H_i$-modules $U$ and $V$, with $V$ finitely generated, 
the \emph{trace pairing}
is the non-degenerate, bilinear form
\[
\Tr := \Tr_{U,V}\df \Hom_{H_i}(U,V) \times \Hom_{H_i}(V,U) \to K
\]
defined by 
\[
(f,g) \mapsto t_i^{\rm max}(\operatorname{Tr}_{H_i}(f \circ g)).
\]
It induces an isomorphism
\[
\Hom_{H_i}(U,V) \ra \Du\Hom_{H_i}(V,U), f\mapsto \Tr(f,-).
\]
Note, that for $U=\bigoplus_{j\in J} H_i u_j$ and $V=\bigoplus_{k=1}^r H_i v_k$ we have
\begin{align*}
\Hom_{H_i}(U,V)&=\prod_{j\in J}\bigoplus_{k=1}^r\Hom_{H_i}(H_iu_j,H_iv_k) \text{ and}\\
\Du\Hom_{H_i}(V,U) &= \Du(\bigoplus_{j\in J}\bigoplus_{k=1}^r\Hom_{H_i}(H_i v_k, H_iu_j))\\&=
\prod_{j\in J}\bigoplus_{k=1}^r\Du\Hom_{H_i}(H_i v_k, H_iu_j).
\end{align*}
Let $W$ be another finitely generated free $H_i$-module.
The following lemma is easily verified:

\begin{Lem}\label{lem:dualmap}
For $f \in \Hom_{H_i}(V,W)$ the following diagram of natural morphisms
commutes:
\[\xymatrix@C+5ex{
 {\Hom_{H_i}(U,V)}\ar[r]^{\Hom(U,f)}\ar[d] &{\Hom_{H_i}(U,W)}\ar[d]\\
{\Du\Hom_{H_i}(V,U)} \ar[r]^{\Du\Hom(f,U)}     &{\Du\Hom_{H_i}(W,U)}
}\]
In other words, under the trace pairing the transpose of $\Hom_{H_i}(f,U)$ is
identified with $\Hom_{H_i}(U,f)$.
\end{Lem}

\subsection{Adjunction and trace pairing}
\label{sec:trace2}

Recall from Section~\ref{sec5.1} that for $(j,i)\in\overline{\Omega}$ we have 
isomorphisms of $H_i\text{-}H_j$-bimodules
\[
\ad_{ji}\df \Hom_{H_j}({_jH_i}, H_j) \ra \Hom_{H_i}(H_i,{_iH_j}),
\]
where we abbreviate ${_jH_i}\otimes_i H_i={_jH_i}$ and 
${_iH_j}\otimes_j H_j = {_iH_j}$.

\begin{Lem} \label{lem:adj-com}
The  diagram of natural isomorphisms
\[\xymatrix{
\Hom_{H_j}({_jH_i}, H_j)\ar[r]^{\ad_{ji}}\ar[d] &\Hom_{H_i}(H_i, {_iH_j})\ar[d]\\
\Du\Hom_{H_j}(H_j,{_jH_i})\ar[r]^{\Du(\ad_{ji})} & \Du\Hom_{H_i}({_iH_j},H_i)
}\]
with the vertical isomorphisms induced by the respective trace pairings, 
commutes.
\end{Lem}

\begin{proof}
We have to show that for any $\phi \in \Hom_{H_j}({_jH_i},H_j)$ and 
$\psi \in \Hom_{H_i}({_iH_j},H_i)$ we have
\begin{equation}\label{eq:trace-id}
\Tr_{H_i,{_iH_j}}(\ad_{ji}(\phi),\psi)=
\Tr_{{_jH_i},H_j}(\phi,\ad_{ij}(\psi))
\end{equation}
To this end, write 
\begin{alignat*}{2}
\phi &=\sum_{g=1}^{g_{ij}} \sum_{k=0}^{f_{ji}-1} 
(\alp_{ji}^{(g)}\eps_i^k)^*\phi^{(g,k)} &\quad\text{with }
\phi^{(g,k)}&=\sum_{l=0}^{c_j-1} \phi^{(g,k)}_l \eps_j^l\in H_j \text{ and}\\
\psi &=\sum_{g=1}^{g_{ij}}\sum_{k'=0}^{f_{ij}-1} 
(\alp_{ij}^{(g)}\eps_j^{k'})^*\psi^{(g,k')} &\quad\text{with }
\psi^{(g,k')}&=\sum_{l'=0}^{c_i-1} \psi^{(g,k')}_{l'} \eps_i^{l'}\in H_i,\\
\end{alignat*}
where we use heavily the notation from Sections~\ref{def:HPi} and~\ref{sec5.1}.
Now a straightforward, though tedious, calculation shows that both sides 
of~\eqref{eq:trace-id} yield
\[
\sum_{g=1}^{g_{ij}} \sum_{k'=0}^{f_{ij}-1}\sum_{k=0}^{f_{ji}-1} \sum_{l=0}^{k_{ij}-1}
\psi^{(g,k')}_{k+f_{ji}l}\; \phi^{(g,k)}_{k'+f_{ij}(k_{ij}-1-l)}.
\]
\end{proof}

As a direct consequence we obtain the following:

\begin{Prop}\label{prop:adj1}
Let $M$ be a free $H_i$-module, and $N$ a finitely generated free $H_j$-module,
and denote by
$\Du(\ad_{ij})\colon 
\Du\Hom_{H_j}(N,{_j}H_i\otimes_i M) \to \Du\Hom_{H_i}({_i}H_j\otimes_j N, M)$ the 
transpose of $\ad_{ij}=\ad_{ij}(N,M)$. We have a commutative diagram
\[
\xymatrix{
{\Hom_{H_j}({_j}H_i\otimes_i M , N)}\ar[r]^{\ad_{ji}}\ar[d]&
{\Hom_{H_i}(M, {_i}H_j\otimes_j N)}\ar[d] \\
{\Du\Hom_{H_j}(N,{_j}H_i\otimes_i M)}  \ar[r]^{\Du(\ad_{ij})} & 
{\Du\Hom_{H_i}({_i}H_j\otimes_j N, M)}
}\]
where the vertical arrows are the isomorphisms induced by the trace pairings.
\end{Prop}

\begin{proof}
In fact, this follows from Lemma~\ref{lem:adj-com} since $\ad_{ji}$ and the
isomorphisms induced by the trace pairing are in fact natural transformations
between the corresponding bifunctors defined on pairs of free modules.
\end{proof}


\section{Reflection functors}\label{sec8}


In this section, let $H = H(C,D,\Omega)$ and $\Pi = \Pi(C,D)$.

\subsection{Reflection functors for $\Pi$}\label{sec8.1}
We keep the notations of Section~\ref{Minn}.
Let $M \in \rep(\Pi)$.
Thus we have
$M = (M_i,M_{ij},M_{ji})$, where $(i,j)$ runs over $\Omega$, such that 
$M_{i,\inn} \circ M_{i,\out} = 0$ for each $i$.
Hence for every $i$, we have $M_{i,\out}(M_i) \subseteq \Ker(M_{i,\inn})$.

Generalizing the construction in \cite[Section~2.2]{BK1}, see also \cite{Bo}, we fix some
vertex $i$ and construct a new $\Pi$-module by replacing the diagram
\[
\widetilde{M_i} \xrightarrow{M_{i,\inn}} 
M_i \xrightarrow{M_{i,\out}} \widetilde{M_i}
\]
by
\[
\widetilde{M_i} \xrightarrow{\overline{M}_{i,\out}M_{i,\inn}} 
\Ker(M_{i,\inn}) \xrightarrow{\can} \widetilde{M_i}
\]
where $\ov{M}_{i,\out}\df M_i \to \Ker(M_{i,\inn})$ is induced by $M_{i,\out}$ and 
$\can$ is the canonical inclusion.
Gluing this new datum with the remaining part of $M$ gives
a new $\Pi$-module $\Sigma_i^+(M)$.

Similarly, replacing 
\[
\widetilde{M_i} \xrightarrow{M_{i,\inn}} 
M_i \xrightarrow{M_{i,\out}} \widetilde{M_i}
\]
by
\[
\widetilde{M_i} \xrightarrow{\can} 
\Coker(M_{i,\out})  \xrightarrow{M_{i,\out}\ov{M}_{i,\inn}} \widetilde{M_i}
\]
where $\ov{M}_{i,\inn}\df \Coker(M_{i,\out}) \to M_i$ is induced by $M_{i,\inn}$ 
and $\can$ is the canonical projection.
Gluing this new datum with the remaining part of $M$ gives
a new $\Pi$-module denoted by $\Sigma_i^-(M)$.

The above constructions are obviously functorial.
It is straightforward to show that $\Sigma_i^+$ is left exact,
and $\Sigma_i^-$ is right exact. 
Both functors are covariant, $K$-linear and additive.

The commutative diagram
\[
\xymatrix{
\widetilde{M_i} \ar@{=}[d]\ar[rr]^{\can}&& 
\Coker(M_{i,\out}) \ar[rr]^<<<<<<<<{M_{i,\out}\overline{M}_{i,\inn}}\ar[d]^{\overline{M}_{i,\inn}}&&
\widetilde{M_i} \ar@{=}[d]\\
\widetilde{M_i} \ar@{=}[d]\ar[rr]^{M_{i,\inn}} && M_i \ar[rr]^{M_{i,\out}}\ar[d]^{\overline{M}_{i,\out}} &&
\widetilde{M_i} \ar@{=}[d]\\
\widetilde{M_i} \ar[rr]^<<<<<<<<<{\overline{M}_{i,\out}M_{i,\inn}} && \Ker(M_{i,\inn}) \ar[rr]^\can &&
\widetilde{M_i} 
}
\]
of $H_i$-module homomorphisms
summarizes the situation and also shows the existence of canonical
homomorphisms $\Sigma_i^-(M) \to M \to \Sigma_i^+(M)$.

For $M \in \rep(\Pi)$ let $\sub_i(M)$ be the largest submodule
$U$ of $M$ such that $e_iU = U$, and let
$\soc_i(M)$ be the largest submodule $V$ of $M$ such that
$V$ is isomorphic to a direct sum of copies of $S_i$.
For example, we have $\sub_i(E_i) = E_i$ and $\soc_i(E_i) \cong S_i$.
Dually, let $\fac_i(M)$ be the largest factor module $M/U$ of
$M$ such that $e_i(M/U) = M/U$, and
let $\tp_i(M)$ be the largest factor module $M/V$ of $M$ such that
$M/V$ is isomorphic to a direct sum of copies of $S_i$.
All these constructions are functorial.

The proof of the following proposition follows almost word by word the
proof of Baumann and Kamnitzer \cite[Proposition~2.5]{BK1}, who deal with 
classical preprojective algebras associated with Dynkin quivers.
One difference is that we need to work with $\sub_i$ and $\fac_i$
instead of $\soc_i$ and $\tp_i$.

\begin{Prop}\label{prop:refl1}
For each $i$ the following hold:
\begin{itemize}

\item[(i)]
The pair $(\Sigma_i^-,\Sigma_i^+)$ is a pair of adjoint functors, i.e.
there is a functorial isomorphism
\[
\Hom_\Pi(\Sigma_i^-(M),N) \cong
\Hom_\Pi(M,\Sigma_i^+(N)).
\]

\item[(ii)]
The adjunction morphisms $\id \to \Sigma_i^+\Sigma_i^-$ and 
$\Sigma_i^-\Sigma_i^+ \to \id$ can be inserted in functorial short
exact sequences
\[
0 \to \sub_i \to \id \to \Sigma_i^+\Sigma_i^- \to 0
\]
and
\[
0 \to \Sigma_i^-\Sigma_i^+ \to \id \to \fac_i \to 0. 
\]
\end{itemize}
\end{Prop}

\begin{proof}
To establish (i), it is enough to define a pair of mutually inverse bijections between $\Hom_\Pi(\Sigma_i^-(M),N)$ and
$\Hom_\Pi(M,\Sigma_i^+(N))$ 
for any $\Pi$-modules $M$ and $N$, which are functorial 
in $M$ and $N$. 
The construction looks as follows. 
Consider a morphism $f\df M \to \Sigma_i^+(N)$. 
By definition, this is a collection of $H_j$-module homomorphisms
\[
f_j\df M_j \to (\Sigma_i^+(N))_j
\] 
with $1 \le j \le n$ such that the diagram
\[
\xymatrix{
{_i}H_j \otimes_j M_j \ar[r]^{1 \otimes f_j}\ar[d]^{M_{ij}} &
{_i}H_j \otimes_j (\Sigma_i^+(N))_j \ar[d]^{(\Sigma_i^+(N))_{ij}}
\\
M_i \ar[r]^{f_i} & (\Sigma_i^+(N))_i
}
\]
commutes for all $(i,j) \in \ov{\Omega}$.
Recall that
\[
\wti{M}_i = \bigoplus_{j \in \ov{\Omega}(i,-)} {_i}H_j \otimes_j M_j.
\]
Set 
\[
\widetilde{f_i} := \bigoplus_{j \in \ov{\Omega}(i,-)} 1 \otimes f_j 
\df \widetilde{M_i} \to  
\widetilde{N_i}. 
\]
In the diagram
\[
\xymatrix@-0.8pc{
\widetilde{M_i} \ar[dd]^{\widetilde{f_i}}\ar[rrr]^{M_{i,\inn}} &&& 
M_i \ar[dd]^{f_i}\ar[rr]^{M_{i,\out}} && 
\widetilde{M_i} \ar[dd]^{\widetilde{f_i}}\ar[rr]^<<<<<<<{\pi} &&
\Coker(M_{i,\out}) \ar@{-->}[dd]^{g_i}\ar[rrr]^<<<<<<<<<<{M_{i,\out}\overline{M}_{i,\inn}} &&& \widetilde{M_i}\ar[dd]^{\widetilde{f_i}}
\\
&&&&&&&&&&&\\
\widetilde{N_i} \ar[rrr]^<<<<<<<<<<{\overline{N}_{i,\out}N_{i,\inn}} &&& 
\Ker(N_{i,\inn}) \ar[rr]^<<<<<<{\iota} && 
\widetilde{N_i} \ar[rr]^{N_{i,\inn}} &&
N_i \ar[rrr]^{N_{i,\out}} &&& \widetilde{N_i}
}
\]
the two left squares commute. 

There is thus a unique map $g_i$ making the third square commutative.
(Observe that $N_{i,\inn}\wti{f_i}M_{i,\out} = N_{i,\inn}\iota f_i = 0$.
Thus $N_{i,\inn}\wti{f_i}$ factors through the cokernel of $M_{i,\out}$.)

The fourth square also commutes. 
Thus if we set $g_j := f_j$ for all vertices $j \not= i$
we get a homomorphism $g\df \Sigma_i^-(M) \to N$.
Conversely, consider a homomorphism $g\df \Sigma_i^-(M) \to N$ and set
\[
\widetilde{g_i} := \bigoplus_{j \in \ov{\Omega}(-,i)} 1 \otimes g_j\df 
\widetilde{M_i} \to  
\widetilde{N_i}. 
\]
In the diagram
\[
\xymatrix@-0.8pc{
\widetilde{M_i} \ar[dd]^{\widetilde{g_i}}\ar[rrr]^{M_{i,\inn}} &&& 
M_i \ar@{-->}[dd]^{f_i}\ar[rr]^{M_{i,\out}} && 
\widetilde{M_i} \ar[dd]^{\widetilde{g_i}}\ar[rr]^<<<<<<<{\pi} &&
\Coker(M_{i,\out}) \ar[dd]^{g_i}\ar[rrr]^<<<<<<<<<<{M_{i,\out}\overline{M}_{i,\inn}} &&& \widetilde{M_i}\ar[dd]^{\widetilde{g_i}}
\\
&&&&&&&&&&&\\
\widetilde{N_i} \ar[rrr]^<<<<<<<<<<{\overline{N}_{i,\out}N_{i,\inn}} &&& 
\Ker(N_{i,\inn}) \ar[rr]^<<<<<<{\iota} && 
\widetilde{N_i} \ar[rr]^{N_{i,\inn}} &&
N_i \ar[rrr]^{N_{i,\out}} &&& \widetilde{N_i}
}
\]
the two right squares commute. 
Thus there is a unique map $f_i$ making the second square commutative. 
The first square then also commutes.
Thus if we set $f_j := g_j$ for all the vertices 
$j \not= i$, we get a morphism $f\df M \to \Sigma_i^+(N)$.
To establish (ii), one checks that $\Sigma_i^-\Sigma_i^+(M)$ is the $\Pi$-module obtained by replacing in $M$ the part
\[
\widetilde{M_i} \xrightarrow{M_{i,\inn}} 
M_i \xrightarrow{M_{i,\out}} \widetilde{M_i}
\]
with
\[
\widetilde{M_i} \xrightarrow{M_{i,\inn}} \Ima(M_{i,\inn}) \xrightarrow{M_{i,\out}} 
\widetilde{M_i}
\]
and that $\Sigma_i^+\Sigma_i^-(M)$ is the $\Pi$-module obtained by replacing in
$M$ the part 
\[
\widetilde{M_i} \xrightarrow{M_{i,\inn}} 
M_i \xrightarrow{M_{i,\out}} \widetilde{M_i}
\]
with
\[
\widetilde{M_i} \xrightarrow{M_{i,\inn}} \Ima(M_{i,\out}) \xrightarrow{M_{i,\out}} 
\widetilde{M_i}.
\]
It remains to observe that as vector spaces, $\fac_i(M) \cong \Coker(M_{i,\inn})$
and $\sub_i(M) \cong \Ker(M_{i,\out})$.
\end{proof}

For the following corollary, observe that $\sub_i(M) = 0$ if and only if
$\soc_i(M) = 0$. Dually, $\fac_i(M) = 0$ if and only if $\tp_i(M) = 0$.

\begin{Cor}\label{cor:refl2}
The functors $\Sigma_i^+\df \cT_i \to \cS_i$ and $\Sigma_i^-\df \cS_i \to \cT_i$
define inverse equivalences of the subcategories 
\[
\cT_i := \left\{ M \in \rep(\Pi) \mid \tp_i(M) = 0 \right\}
\]
and 
\[
\cS_i :=
\left\{ M \in \rep(\Pi) \mid \soc_i(M) = 0 \right\}.
\]
\end{Cor}

\begin{Cor}\label{cor:refl3}
For $M,N \in \rep(\Pi)$ the following hold:
\begin{itemize}

\item[(i)]
If $M,N \in \cT_i$, then $\Sigma_i^+$ induces an isomorphism
\[
\Ext_\Pi^1(M,N) \cong \Ext_\Pi^1(\Sigma_i^+(M),\Sigma_i^+(N)).
\]

\item[(ii)]
If $M,N \in \cS_i$, then $\Sigma_i^-$ induces an isomorphism
\[
\Ext_\Pi^1(M,N) \cong \Ext_\Pi^1(\Sigma_i^-(M),\Sigma_i^-(N)).
\]

\end{itemize}
\end{Cor}

\begin{Prop}\label{prop:refl4}
For $M \in \rep(\Pi)$ 
the following are equivalent:
\begin{itemize}

\item[(i)]
$\tp_i(M) = 0$;

\item[(ii)]
$M \cong \Sigma_i^-\Sigma_i^+(M)$.

\end{itemize}
Furthermore, if $M \in \rep_\vp(\Pi)$, then {\rm (i)} and {\rm (ii)} are
equivalent to the following:
\begin{itemize}

\item[(iii)]
$\rkv(\Sigma_i^+(M)) = s_i(\rkv(M))$.

\end{itemize}
Dually, the following are equivalent:
\begin{itemize}

\item[(i)]
$\soc_i(M) = 0$;

\item[(ii)]
$M \cong \Sigma_i^+\Sigma_i^-(M)$.

\end{itemize}
Furthermore, if $M \in \rep_\vp(\Pi)$, then {\rm (i)} and {\rm (ii)} are
equivalent to the following:
\begin{itemize}

\item[(iii)]
$\rkv(\Sigma_i^-(M)) = s_i(\rkv(M))$.

\end{itemize}
\end{Prop}

\begin{proof}
The equivalence of (i) and (ii) follows directly from
Proposition~\ref{prop:refl1} and Corollary~\ref{cor:refl2}.

Suppose (i) holds for some vertex $i$ of $Q(C,\Omega)$.
Let $a = (a_1,\ldots,a_n) = \rkv(M)$. 
Recall that we have the $H_i$-module homomorphism
\[
M_{i,\inn}\df \bigoplus_{j \in \ov{\Omega}(i,-)} {_i}H_j \otimes_j M_j 
\to M_i.
\]
Since $\tp_i(M) = 0$, the map $M_{i,\inn}$ is surjective.
This implies that $\Sigma_i^+(M)$ is again locally free with
\[
(\rkv(\Sigma_i^+(M)))_i 
= \sum_{j \in \ov{\Omega}(i,-)} |c_{ij}|a_j - a_i 
=
(s_i(\rkv(\Sigma_i^+(M))))_i.
\]
(Here we used that ${_i}H_j \otimes_j M_j$ is a free $H_i$-module of rank
$|c_{ij}|a_j$.)
Thus (iii) holds.

Vice versa, the
equality 
\[
(\rkv(\Sigma_i^+(M)))_i = (s_i(\rkv(\Sigma_i^+(M))))_i
\]
implies that $M_{i,\inn}$ is surjective.
Thus (iii) implies (i).
\end{proof}

Let $M = (M_i,M_{ij},M_{ji}) \in \rep(C,D,\overline{\Omega})$.
We say that $M_{i,\inn}$ (resp. $M_{i,\out}$) 
\emph{splits} if the image of $M_{i,\inn}$  (resp. 
$M_{i,\out}$) 
is a free $H_i$-module.
The following lemma is straightforward.

\begin{Lem}\label{lem:8.5}
Let $M \in \rep(\Pi)$ be locally free.
For each $i$ the following hold:
\begin{itemize}

\item[(i)]
$\Sigma_i^+(M)$ is locally free if and only if $M_{i,\inn}$ splits.

\item[(ii)]
$\Sigma_i^-(M)$ is locally free if and only if $M_{i,\out}$ splits.
\end{itemize}
\end{Lem}

\subsection{Reflection functors for $H$}
We keep the notations of Section~\ref{subsec5.2}.
Recall that
for an orientation $\Omega$ of $C$ and some $1 \le i \le n$ we defined
\[
s_i(\Omega) :=
\{ (r,s) \in \Omega \mid i \notin \{r,s\} \}
\cup \{ (s,r) \in \Omega^* \mid i \in \{r,s\}  \}.
\]
Define
\begin{equation}\label{eq:s_i(H)}
s_i(H) := s_i(H(C,D,\Omega)) := H(C,D,s_i(\Omega)).
\end{equation}

Now let
$k$ be a sink in $Q^\circ(C,\Omega)$.
Then $\Sigma_k^+$ obviously restricts to a 
\emph{reflection functor}
\[
F_k^+\df \rep(H) \to \rep(s_k(H))
\]
which can also be described as follows.
Let $M = (M_i,M_{ij}) \in \rep(C,D,\Omega)$.
Recall that
\[
M_{k,\inn} = (\sgn(k,j)M_{kj})_j\df \bigoplus_{j \in \Omega(k,-)} {_k}H_j \otimes M_j \to M_k. 
\]
(Note that $\Omega(k,-) = \ov{\Omega}(k,-)$, since $k$ is a sink.)
Let $N_k := \Ker(M_{k,\inn})$.
We obtain an exact sequence
\[
0 \to N_k 
\to
\bigoplus_{j \in \Omega(k,-)}{_k}H_j \otimes_j M_j \xrightarrow{M_{k,\inn}} M_k. 
\]
Let us denote by $(N_{jk}^\vee)_j$ the inclusion map 
$N_k \to \bigoplus_{j \in \Omega(k,-)}{_k}H_j \otimes_j M_j$.
Then we have $F_k^+(M) = (N_r,N_{rs})$ with $(r,s) \in s_k(\Omega)$, where
\[
N_r :=
\begin{cases}
M_r & \text{if $r \not= k$},\\
N_k & \text{if $r=k$}
\end{cases}
\quad \mbox{and} \quad 
N_{rs} :=
\begin{cases} 
M_{rs} & \text{if $(r,s) \in \Omega$ and $r \not= k$},\\
(N_{rs}^\vee)^\vee & \text{if $(r,s) \in \Omega^*$ and $s=k$}.
\end{cases}
\]
Similarly, if $k$ is a source in $Q^\circ(C,\Omega)$, then
$\Sigma_k^-$ restricts to a 
\emph{reflection functor}
\[
F_k^-\df \rep(H) \to \rep(s_k(H)).
\]

\begin{Prop}\label{prop:reflvp}
Let $M$ be locally free and rigid in $\rep(H)$.
Then $F_k^\pm(M)$ is locally free and rigid.
\end{Prop}

\begin{proof}
Without loss of generality assume that $1$ is a sink and $n$ is a source
in $Q^\circ(C,\Omega)$.
Let $M$ be a locally free and rigid $H$-module.
To get a contradiction, assume
that $F_1^+(M)$ is not locally free and that $M$ is of minimal dimension
with this property. 
Recall that $F_1^+(E_1) = 0$. 
Thus by the minimality of its dimension, $M$ does not have any direct summand isomorphic to $E_1$.
We can also assume that $e_iM \not= 0$ for all $1 \le i \le n$.
(Otherwise $M$ can be considered as a module over an algebra
$H(C',D',\Omega')$ with fewer vertices.)
Since $F_1^+(M)$ is not locally free, we get $\Hom_H(M,E_1) \not= 0$,
see Lemma~\ref{lem:8.5}.

We have a short exact sequence
\begin{equation}\label{eq:8.6}
0 \to M' \to M \to e_nM \to 0
\end{equation}
where $M' := (e_1 + \cdots + e_{n-1})M$.
Clearly, $M'$ and $e_nM$ are both locally free. 
In particular, $e_nM \cong E_n^s$ for some $s \ge 1$.
We have $\Hom_H(M',e_nM) = 0$.
Thus applying $\Hom_H(M',-)$ to (\ref{eq:8.6}) we get an embedding
$\Ext_H^1(M',M') \to \Ext_H^1(M',M)$.
Applying $\Hom_H(-,M)$ to (\ref{eq:8.6}) and using that
$\Ext_H^2(e_nM,M) = 0$ we get $\Ext_H^1(M',M) = 0$.
This shows that $M'$ is rigid.
Applying $\Hom_H(-,E_1)$ to the sequence (\ref{eq:8.6}) yields that 
$\Hom_H(M',E_1) \not= 0$.
Since $M'$ is locally free and rigid, the minimality of $M$ implies that
$M' \cong E_1^r \oplus U$ for some $r \ge 1$ and some locally free
and rigid module $U$ with $\Hom_H(U,E_1) = 0$.
This yields short exact sequences
\begin{equation}\label{eq:8.6b}
0 \to U \xrightarrow{f} M \xrightarrow{g} V \to 0
\end{equation}
and
\begin{equation}\label{eq:8.6c}
0 \to E_1^r \to V \to e_nM \to 0
\end{equation}
where $f$ is the obvious embedding, and $g$ is the obvious projection
onto $V := \Coker(f)$.
Note that $V$ is also locally free, and that
$\Hom_H(U,E_n) = 0$.
Applying $\Hom_H(U,-)$ to (\ref{eq:8.6c}) implies
that $\Hom_H(U,V) = 0$.
Now we apply $\Hom_H(M,-)$ and $\Hom_H(-,V)$ to (\ref{eq:8.6b})
and get similarly as before that $V$ is rigid.
Applying $\Hom_H(-,E_1)$ to (\ref{eq:8.6b}) implies
$\Hom_H(M,E_1) \cong \Hom_H(V,E_1) \not= 0$.
By the minimality of $M$ it now follows that $U=0$ and $V=M$.
Thus we have 
$n=2$.

Since $M$ is locally free, Proposition~\ref{vpdim} implies that it has a 
minimal projective
resolution of the form
\[
0 \to P'' \to P' \to M \to 0.
\]
Since $M$ is rigid and locally free, we know that $M$ is also 
$\tau$-rigid, see Corollary~\ref{cor:3.9}.
Thus by
\cite[Proposition~2.5]{AIR} we get that $\add(P') \cap \add(P'') = 0$.
(It can be easily checked that
\cite[Proposition~2.5]{AIR} is true for arbitrary ground fields.)
We obviously have $\Hom_H(M,E_2) \not= 0$, since $2$ is a source
in $Q^\circ(C,\Omega)$.
By assumption we have $\Hom_H(M,E_1) \not= 0$.
This implies that $\Hom_H(M,S_i) \not= 0$ for the simple $H$-modules
$S_i$, where $i = 1,2$.
Thus $P'$ contains both $P_1$ and $P_2$ as direct summands.
Since $\add(P') \cap \add(P'') = 0$, it follows that $P'' = 0$.
In other words, $M$ is projective.
But $P=E_1$ is the only indecomposable projective $H$-module
with $\Hom_H(P,E_1) \not= 0$.
Thus $M$ contains a direct summand
isomorphic to $E_1$, a contradiction.

Altogether we proved that $\Hom_H(M,E_1) = 0$ for any locally
free and rigid $H$-module $M$, which does not have a direct summand
isomorphic to $E_1$.
In this case, we have
$\tp_1(M) = 0$ and $M_{1,\inn}$ is surjective.
Thus $F_1^+(M)$ is rigid by Corollary~\ref{cor:refl3}, and $F_1^+(M)$ is locally free by Lemma~\ref{lem:8.5}.
The corresponding statement for $F_n^-(M)$ is proved dually.
This finishes the proof.
\end{proof}

\subsection{Reflection functors and APR-tilting}\label{sec:apr}
As before, let $H = H(C,D,\Omega)$ and $Q = Q(C,\Omega)$.
Let $i$ be a sink in $Q^\circ$, and define
\[
T := {_H}H/E_i  \oplus \tau_H^-(E_i) 
\text{\;\;\; and \;\;\;} 
B := \End_H(T)^\op.
\]
We also assume that $i$ is not a source in $Q^\circ$.
We have $E_i = He_i$.
For $1 \le j \le n$ set $P_j := He_j$ and $I_j := \Du(e_jH)$.
Let $S_j$ be the corresponding simple $H$-module with
$S_j \cong \tp(P_j) \cong \soc(I_j)$.
Furthermore, let
\[
T_j := 
\begin{cases}
He_j & \text{if $j \not= i$},
\\ 
\tau_H^-(E_i) & \text{if $j=i$}.
\end{cases}
\]
Thus we have $T = T_1 \oplus \cdots \oplus T_n$.
The indecomposable projective $B$-modules are (up to isomorphism) 
$\Hom_H(T,T_j)$ where $1 \le j \le n$.
Let $e_j$ denote the primitive idempotent in $B$ obtained 
by composing the canonical projection $T \to T_j$ with the
canonical inclusion $T_j \to T$.
Then $\Hom_H(T,T_j) \cong Be_j$.
Finally, let $S_j$ denote the simple $B$-module associated with $e_j$.
We are using the same notation for the idempotents in 
$H$, $s_i(H)$ and $B$ and also for the associated simple modules.
However, the context will always save us from confusion.

Recall that a finite-dimensional module $T$ over a finite-dimensional
algebra $A$ is a \emph{tilting module} if the following hold:
\begin{itemize}

\item[(i)]
$\pdim(T) \le 1$;

\item[(ii)]
$\Ext_A^1(T,T) = 0$;

\item[(iii)]
The number of isomorphism classes of indecomposable direct
summands of $T$ is equal to the number of isomorphism classes
of simple $A$-modules.

\end{itemize}

The following result is inspired by Auslander, Platzeck and Reiten's \cite{APR} ground breaking interpretation of BGP-reflection
functors as homomorphism functors of certain tilting
modules.
Their result can be seen as the beginning of tilting theory.

\begin{Thm}\label{thm:APR}
With the notation above, the $H$-module $T$ is a tilting module,
and the functors 
\[
F_i^+(-)\df \rep(H) \to \rep(s_i(H))
\text{\;\;\; and \;\;\;}
\Hom_H(T,-)\df \rep(H) \to \rep(B)
\]
are \emph{equivalent}, i.e. 
there exists an equivalence 
\[
S\df \rep(s_i(H)) \to \rep(B)
\]
such that we have an isomorphism of functors 
$ \circ F_i^+ \cong \Hom_H(T,-)$.
\end{Thm}

As a preparation for the proof of Theorem~\ref{thm:APR}
we construct a minimal injective resolution 
of the $H$-module $E_i$, see (\ref{eq9.6}) below.
This yields via the inverse Nakayama functor a minimal projective resolution (\ref{eq9.7}) of $\tau_H^-(E_i)$.

Let $E_i'$ be the right $H$-module such that 
$\Du(E_i') \cong E_i$.
Similarly to Proposition~\ref{resolutionE}
there is a minimal projective resolution 
\begin{equation}\label{eq:11.1}
0 \to \bigoplus_{j \in \Omega(i,-)} {_i}H_j \otimes_j e_jH \xrightarrow{\mu}
e_iH \to E_i' \to 0
\end{equation}
of the right $H$-module $E_i'$,
where the map $e_iH \to E_i'$ is the canonical projection, and for $j \in \Omega(i,-)$
the $j$th component of the map $\mu$ is just the multiplication
map
\[
H_{ij}\df {_i}H_j \otimes_j e_jH \to e_iH.
\]
(The maps $H_{ij}$ are the structure maps of the (left) $H$-module structure
on $H$.)
Note that ${_i}H_j \otimes_j e_jH \cong e_jH^{|c_{ji}|}$.

There is a chain of $H_j$-$H_i$-bimodule isomorphisms
\[
{_j}H_i \cong \Hom_{H_i}({_i}H_j,H_i) \cong \Du\Hom_{H_j}({_j}H_i,H_j)
\cong \Du({_i}H_j),
\]
where the first and third isomorphisms are constructed as in 
Section~\ref{sec5.1}, and the second isomorphism is defined as in 
Lemma~\ref{lem:adj-com}. Thus 
for $r \in {_i}R_j$ we can consider $r^*$ now as an element in $\Du({_i}H_j)$.

Under the isomorphism
\[
\Du(e_jH) \otimes_j \Du({_i}H_j) \to \Du({_i}H_j \otimes_j e_jH)
\]
defined by 
\[
\iota \otimes \eta \mapsto 
\left(h \otimes e_jh' \mapsto \eta(h)\iota(e_jh')\right)
\]
the dual of the multiplication map $H_{ij}$ is identified with the map
\[
H_{ij}^*\df \Du(e_iH) \to \Du(e_jH) \otimes_j \Du({_i}H_j)
\]
defined by 
\[
\varphi \mapsto \sum_{r \in {_i}R_j} \varphi(r \cdot -) \otimes r^*.
\]
We have $\Du(e_iH) = I_i$ and 
$\Du(e_jH) \otimes_j \Du({_i}H_j) \cong I_j \otimes_j {_j}H_i \cong I_j^{|c_{ji}|}$.
Applying the duality $\Du$ to (\ref{eq:11.1}) we get a minimal injective
resolution
\begin{equation}\label{eq9.6}
0 \to E_i \to I_i \xrightarrow{(H_{ij}^*)_{j \in \Omega(i,-)}}
\bigoplus_{j \in \Omega(i,-)} I_j \otimes_j {_j}H_i \to 0.
\end{equation}
For $j \in \Omega(i,-)$ set
\[
\theta_{ij} := \nu_H^{-1}(H_{ij}^*) \df E_i \to P_j \otimes_j {_j}H_i
\]
which is given by
$he_i \mapsto \sum_{r \in {_i}R_j} hr \otimes r^*$.
Here $\nu_H^{-1}\df \inj(H) \to \proj(H)$ is the inverse Nakayama functor obtained via restriction from $\Hom_H(\Du(H_H),-)$.
Let $\theta := (\theta_{ij})_j$ where $j \in \Omega(i,-)$. 
Recall that $E_i = P_i$, since $i$ is a sink in $Q^\circ$.
We get that the exact sequence
\begin{equation}\label{eq9.7}
0 \to E_i \xrightarrow{\theta} \bigoplus_{j \in \Omega(i,-)}
P_j \otimes_j {_j}H_i \to \tau_H^-(E_i) \to 0
\end{equation}
is a minimal projective resolution of $\tau_H^-(E_i)$.
Note that (\ref{eq9.7}) is not an Auslander-Reiten sequence if $E_i$
is not simple.

\begin{Lem}\label{lem:9.8}
The following hold:
\begin{itemize}

\item[(i)]
$\tau_H^-(E_i)$ is locally free;

\item[(ii)]
$\Hom_H(\tau_H^-(E_i),H) = 0$;

\item[(iii)]
$\End_H(\tau_H^-(E_i)) \cong \End_H(E_i) \cong H_i$.

\end{itemize}
\end{Lem}

\begin{proof}
The existence of the exact sequence
(\ref{eq9.7}) implies that $\tau_H^-(E_i)$ is locally free.
Since $\idim(E_i) \le 1$, we have $\Hom_H(\tau_H^-(E_i),H) = 0$.
The $H$-module $E_i$ is indecomposable non-injective with
$\pdim(\tau_H^-(E_i)) = 1$ and $\idim(E_i) = 1$.
Thus by the Auslander-Reiten formulas we get 
$\End_H(\tau_H^-(E_i)) \cong \End_H(E_i) \cong H_i$.
This finishes the proof.
\end{proof}

\begin{Lem}\label{lem:APR1}
The $H$-module $T$ is a tilting module.
\end{Lem}

\begin{proof}
By the above considerations we know that
\[
0 \to E_i \xrightarrow{\theta} \bigoplus_{j \in \Omega(i,-)}
P_j \otimes_j {_j}H_i \to \tau_H^-(E_i) \to 0
\]
is a minimal projective resolution of $\tau_H^-(E_i)$.
It follows that $\pdim(T) = 1$.
Furthermore, we have
\begin{align*}
\Ext_H^1(T,T) &= \Ext_H^1(\tau_H^-(E_i),T)\\
&\cong \Du\Hom_H(T,E_i)\\
&= \Du\Hom_H(\tau_H^-(E_i),E_i)\\
&= 0.
\end{align*}
The equality in the first line of the above equations holds since $\tau_H^-(E_i)$ is the only
non-projective indecomposable direct summand of $T$.
The isomorphism in the second line follows by the Auslander-Reiten formulas
and the fact that $\pdim(\tau_H^-(E_i)) \le 1$.
The equality in the third line holds since $i$ is a sink. 
Finally, the equality in the fourth line follows from 
Lemma~\ref{lem:9.8}(ii).
The module $H$ has exactly $n$ pairwise non-isomorphic indecomposable direct summands.
This finishes the proof.
\end{proof}

\begin{Lem}\label{lem:APR3}
We have ${_B}\Du(T) \cong \Du(B)/\Du(e_iB) \oplus \tau_B(\Du(e_iB))$.
\end{Lem}

\begin{proof}
We have 
\[
{_B}\Du(T) \cong \Hom_H(T,\Du(H)) = \bigoplus_{j=1}^n \Hom_H(T,\Du(e_jH)).
\]
For $j \not= i$ there are $B$-module isomorphisms
\[
\Hom_H(T,\Du(e_jH)) \cong \Du\Hom_H(He_j,T) = \Du\Hom_H(T_j,T)
\cong \Du(e_jB).
\]
It follows that
\[
{_B}\Du(T) \cong \Du(B)/\Du(e_iB) \oplus \Hom_H(T,\Du(e_iH)).
\]
Since $\Ext_H^1(T,E_i) \not= 0$, the Connecting Lemma \cite[Section~2.3]{HR} 
implies that
\[
\Hom_H(T,\Du(e_iH)) \cong \tau_B(\Ext_H^1(T,E_i)).
\]
We have 
\[
\Ext_H^1(T,E_i) \cong \Du\Hom_H(\tau_H^-(E_i),T) \cong \Du(e_iB).
\]
The first isomorphism is obtained from the Auslander-Reiten formulas
and the fact that $\idim(E_i) \le 1$.
The second isomorphism is obvious.
Note that both isomorphism are $B$-module isomorphisms.
This finishes the proof.
\end{proof}

For any of the algebras $A \in \{ H, s_i(H), B \}$ and any simple 
$A$-module $S_j$ let
\begin{align*}
\cT_j^A := \{ M \in \rep(A) \mid \Hom_A(M,S_j) = 0 \},
\\
\cS_j^A := \{ M \in \rep(A) \mid \Hom_A(S_j,M) = 0 \}.
\end{align*}
Using the notation from Section~\ref{sec8.1}, for $A \in \{ H,s_i(H) \}$ 
we have
$\cT_j^A = \cT_j \cap \rep(A)$ and $\cS_j^{A} = \cS_j \cap \rep(A)$.

\begin{Lem}\label{lem:APR4}
The functors
$F := \Hom_H(T,-)$ and $G := T \otimes_{B} -$
induce mutually quasi-inverse equivalences
$F\df \cT \to \cY$ and $G\df \cY \to \cT$, where
\begin{align*}
\cT &:= \{ M \in \rep(H) \mid \Ext_H^1(T,M) = 0\} 
\\
&= \{ M \in \rep(H) \mid \Du\Hom_H(M,E_i) = 0 \} 
\\
&= \cT_i^H
\end{align*}
and
\begin{align*}
\cY &:= \{ N \in \rep(B) \mid {\rm Tor}_B^1(T,N) = 0 \}
\\
&= \{ N \in \rep(B) \mid \Du\Ext_B^1(N,\Du(T)) = 0 \}
\\
&= \{ N \in \rep(B) \mid \Hom_B(\Du(e_iB),N) = 0 \}
\\
&= \cS_i^B.
\end{align*}
Furthermore, we have $F(\rep(H)) \subseteq \cY$ and
$G(\rep(B)) \subseteq \cT$.
\end{Lem}

\begin{proof}
This follows mainly from classical tilting theory (Brenner-Butler Therorem, 
see for example~\cite[Section~2]{HR}) and the Auslander-Reiten formulas.
The third equality in the description of $\cY$ uses the Auslander-Reiten 
formulas in combination with Lemma~\ref{lem:APR3}. 
For $j \not= i$ we have 
\[
\Hom_B(\Du(e_iB),\Du(e_jB)) \cong \Du\Ext_B^1(\Du(e_jB),\Du(T)) = 0.
\]
Thus every composition factor of $\Du(e_iB)$ is isomorphic to
$S_i$.
This implies the fourth equality in the description of $\cY$.
\end{proof}

\begin{Cor}\label{cor:APRd}
Via restriction of
the functors $F_i^+$, $F \circ F_i^-$ and $F$ we get a commutative diagram
\[
\xymatrix{
\cT_i^H \ar[r]^{F_i^+} \ar@{=}[d] & \cS_i^{s_i(H)} \ar[d]^{F \circ F_i^-}
\\
\cT_i^H \ar[r]^F & \cS_i^B
}
\]
of equivalences of subcategories.
\end{Cor}

\begin{proof}
Combine Lemma~\ref{lem:APR4} and
Corollary~\ref{cor:refl2}.
\end{proof}

\begin{Cor}\label{cor:APRe}
Via restriction of the functor $F \circ F_i^-$ we get a commutative diagram
\[
\xymatrix{
\rep(H) \ar[r]^{F_i^+} \ar@{=}[d] & \cS_i^{s_i(H)} \ar[d]^{F \circ F_i^-}
\\
\rep(H) \ar[r]^F & \cS_i^B
}
\]
with $F \circ F_i^-$ an equivalence of subcategories.
\end{Cor}

\begin{proof}
This follows from Corollary~\ref{cor:APRd}, the definition of $F_i^+$ and the last
statement in Lemma~\ref{lem:APR4}.
\end{proof}

\begin{Lem}\label{lem:APRb}
The following hold:
\begin{itemize}

\item[(i)]
$\proj(s_i(H)) \subseteq \cS_i^{s_i(H)}$;

\item[(ii)]
$\proj(B) \subseteq \cS_i^B$.

\end{itemize}
\end{Lem}

\begin{proof}
Since $i$ is a source in $Q(C,s_i(\Omega))^\circ$, we get that
$\proj(s_i(H)) \subseteq \cS_i^{s_i(H)}$.
Part (ii) is obvious, since the modules $\Hom_T(T,T_j)$ are (up 
to isomorphism) the indecomposable projective $B$-modules.
\end{proof}

Theorem~\ref{thm:APR} follows now from \cite[Lemma~2.2]{APR}. (Corollary~\ref{cor:APRe} and Lemma~\ref{lem:APRb} and the fact
that $\cS_i^A$ is closed under submodules for $A \in \{ s_i(H),B\}$
ensure that the assumptions of \cite[Lemma~2.2]{APR} are
satisfied.)

As a consequence, 
via restriction of $F \circ F_i^-$, we get an equivalence of subcategories 
\[
\proj(s_i(H)) \to \proj(B).
\]
In particular, we have $F_i^+(T_j) \cong s_i(H)e_j$ for all $1 \le j \le n$.
We also get an algebra isomorphism $s_i(H) \cong B$.

We leave it as an exercise to formulate a dual version of Theorem~\ref{thm:APR}.

\subsection{Coxeter functors}\label{sec:6.3}
Let $Q = Q(C,\Omega)$.

Given a $+$-admissible sequence $(i_1,\ldots,i_n)$ for $(C,\Omega)$
let
\[
C^+ := F_{i_n}^+ \circ \cdots \circ F_{i_1}^+\df \rep(H) \to \rep(H).
\]
Dually, one defines $-$-\emph{admissible sequences}
$(j_1,\ldots,j_n)$ and $C^- := F_{j_n}^- \circ \cdots \circ F_{j_1}^-$.
We call $C^+$ and $C^-$ \emph{Coxeter functors}.
Similarly as in the classical case one proves the following result,
compare \cite{BGP}.

\begin{Lem}
The functors $C^+$ and $C^-$ do not depend on the chosen 
admissible sequences for $(C,\Omega)$.
\end{Lem}

The next lemma is a consequence of Proposition~\ref{prop:refl1},
Corollary~\ref{cor:refl2}
 and
Lemma~\ref{lem:8.5}.

\begin{Lem}\label{lem:8.7}
Let $M$ be an indecomposable locally free $H$-module.
Let $(i_1,\ldots,i_n)$ be a $+$-admissible sequence for $(C,\Omega)$.
Assume that  $F_{i_s}^+ \cdots F_{i_1}^+(M)$ is locally free and non-zero
for some $1 \le s \le n$.
Then we have
\[
F_{i_k}^+ \cdots F_{i_1}^+(M) \cong
(F_{i_{k+1}}^- \cdots F_{i_s}^-)(F_{i_s}^+ \cdots F_{i_{k+1}}^+)
F_{i_k}^+ \cdots F_{i_1}^+(M)
\]
for $1 \le k \le s-1$, and
$F_{i_k}^+ \cdots F_{i_1}^+(M)$ is indecomposable and locally free
for $1 \le k \le s$.
\end{Lem}

There is also an obvious dual of Lemma~\ref{lem:8.7}.


\section{Coxeter functors and Auslander-Reiten translations}
\label{sec10}


\subsection{Overview}

As before we fix $H=H(C,D,\Omega)$. 
Our aim is to compare
the Coxeter functors $C^+$ and $C^-$ introduced in Section~\ref{sec:6.3} with the
Auslander-Reiten translations $\tau$ and $\tau^-$. 

Without loss of
generality we assume that for each
$(i,j) \in \Omega$ we have $i<j$. 
Thus, 
\[
C^+ = F_n^+ \circ \cdots \circ F_1^+.
\]
Recall that we defined the 
twist automorphism $T$ of $H$ by 
\[
T(\vep_i) = \vep_i,\qquad T(\alpha_{ij}^{(g)}) = -\alpha_{ij}^{(g)}, \qquad 
(i\in Q_0,\ (i,j)\in\Omega,\ 1\le g \le g_{ij}).
\]
The twist by $T$ defines
an automorphism of $\rep(H)$ which we denote also by $T$. 
More explicitly,
for $M = (M_i,M_{ij}) \in \rep(H)$ we have $(TM)_i = M_i$ and  
$(TM)_{ij} = -M_{ij}$. 

Following \cite[Section 5]{Ga3}, we start in Section~\ref{ssec:GaRie} by constructing a new algebra
$\wti{H}$ containing two subalgebras $H_{(0)}$ and $H_{(1)}$ canonically isomorphic to $H$.
Denoting by 
\[
\Res_a \df \rep(\wti{H}) \to \rep(H_{(a)}), \quad (a\in\{0,1\})
\]
the corresponding restriction functors, we will show that 
$C^+ \cong\Res_1\circ \Res_0^*$, where 
\[
\Res_0^*\df \rep(H_{(0)})\to \rep(\wti{H})
\] 
is right adjoint to $\Res_0$.
This will follow from a factorization
\[
\Res_0^* = \Res^*_{(n-1,n)}\circ\cdots\circ\Res^*_{(1,2)}\circ\Res^*_{(0,1)} 
\]
similar to the definition of $C^+$, and from a comparison of the functors 
$\Res^*_{(i-1,i)}$ and $F^+_i$ obtained in Lemma~\ref{lem9.2}.

After that, we will give a different description of the adjoint functor $\Res_0^*$,
which will allow  to show that, for $M\in\rep_\vp (H)$, 
the $H$-module $\Res_1\circ \Res_0^*(M)$ is the kernel of a certain map $d^*_M$.
On the other hand, it follows from Corollary~\ref{cor:resolHM}
that $\tau(TM)$ is the kernel of the map $\Du\Hom_H(d\otimes TM,H)$. 
We will then show that, under the trace pairings, the maps $d^*_M$ 
and $\Du\Hom_H(d\otimes TM,H)$ can be identified, hence 
\[
 C^+(M) \cong \Res_1\circ \Res_0^*(M) \cong \tau(TM).
\]
A more detailed statement of our results will be given in Theorem~\ref{thm:functoriso},
whose proof is carried out in 
Sections~\ref{pf:fiso2ab} to \ref{pf:fiso2cd}.

The remaining subsections present direct applications of Theorem~\ref{thm:functoriso}.
In \ref{subsec:9.6} we give another description of the preprojective algebra $\Pi=\Pi(C,D)$ as a
tensor algebra. In \ref{subsec:9.7} we adapt to our setting a description of the category
$\rep(\Pi)$ due to Ringel in the classical case, in terms of $H$-module homomorphisms $M\to TC^+(M)$.
Finally, in \ref{subsec9.8} we show that the subcategory of 
Gorenstein-projective $H$-modules coincides 
with the kernel of the Coxeter functor $C^+$.

\subsection{An analogue of the Gabriel-Riedtmann construction} \label{ssec:GaRie}
The following is an adaptation of 
\cite[Section 5]{Ga3} to our situation.

\subsubsection{The algebra $\wti{H}$}
To our fixed datum $(C,D,\Omega)$ we attach a new algebra $\wti{H}$ defined 
by a quiver with relations. The quiver $\wti{Q}$ has set of vertices 
\[
 \wti{Q}_0 := \{(i,a)\mid i\in Q_0,\ a\in\{0,1\}\},
\]
and set of arrows
\begin{align*}
 \wti{Q}_1 := &\ \{\alpha_{(i,a)(j,a)}^{(g)}\df (j,a) \to (i,a) \mid (i,j)\in\Omega,\ 1\le g \le g_{ij},\ a\in\{0,1\}\}\\[2mm]
&\ \cup\ 
\{\alpha_{(j,0)(i,1)}^{(g)}\df (i,1) \to (j,0) \mid (i,j)\in\Omega,\ 1\le g \le g_{ij}\}\\[2mm]
&\ \cup\  \{\vep_{(i,a)}\df (i,a) \to (i,a) \mid (i,a) \in \wti{Q}_0\}.
\end{align*}
Accordingly we put
\[
\wti{\Omega}:=
\{((i,a),(j,a)), ((j,0),(i,1))\mid (i,j)\in\Omega, a\in\{0,1\}\}. 
\]
Let 
\[
\wti{H} := K\wti{Q}/\wti{I}
\]
where $\wti{I}$ is the ideal of $K\wti{Q}$ defined by the following relations:
\begin{itemize}

\item[($\wti{\mathrm{H}}1$)]
For each $(i,a)\in \wti{Q}_0$ we have 
\[
\vep_{(i,a)}^{c_i} = 0.
\]

\item[($\wti{\mathrm{H}}2$)]
For each $((i,a),(j,b)) \in \wti{\Omega}$ and each $1 \le g \le g_{ij}$ we have
\[
\vep_{(i,a)}^{f_{ji}}\alpha_{(i,a)(j,b)}^{(g)} = \alpha_{(i,a)(j,b)}^{(g)}\vep_{(j,b)}^{f_{ij}}.
\]

\item[($\wti{\mathrm{H}}3$)]
For each $i\in Q_0$ we have
\begin{align*}
 \sum_{j\in\Omega(i,-)}\sum_{g=1}^{g_{ij}}\sum_{f=0}^{f_{ji}-1}
 \vep_{(i,0)}^f\alpha_{(i,0)(j,0)}^{(g)}\alpha_{(j,0)(i,1)}^{(g)} \vep_{(i,1)}^{f_{ji}-1-f}\qquad& \\
 \qquad+ \sum_{j\in\Omega(-,i)}\sum_{g=1}^{g_{ij}}\sum_{f=0}^{f_{ji}-1}
 \vep_{(i,0)}^f\alpha_{(i,0)(j,1)}^{(g)}\alpha_{(j,1)(i,1)}^{(g)} \vep_{(i,1)}^{f_{ji}-1-f} &= 0.
\end{align*}

\end{itemize}

When $C$ is symmetric and $D$ is minimal the algebra $\wti{H}$ coincides with the bounded quiver 
denoted by $\wti{QQ}$ in \cite[Section 5.3]{Ga3}.

\subsubsection{Example} \label{sssec:GabRie5}
Let $H = H(C,D,\Omega)$ be defined by the quiver 
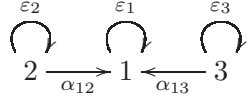
\[
\xymatrix{
2 \ar@(ul,ur)^{\vep_2} \ar[r]_{\alpha_{12}} & 1 \ar@(ul,ur)^{\vep_1} & 3 \ar@(ul,ur)^{\vep_3} \ar[l]^{\alpha_{13}}
}
\]
with relations $\vep_2 = 0$, $\vep_1^2 = \vep_3^2 = 0$ and 
$\vep_1\alpha_{13} = 
\alpha_{13} \vep_3$.
Here $C$ is a Cartan matrix of Dynkin type $B_3$, and $D$ is the minimal symmetrizer.
Then $\wti{H}$ is defined by the quiver
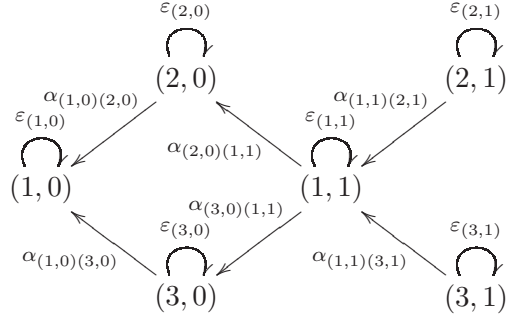
\[
\xymatrix{
& (2,0) \ar@(ul,ur)^{\vep_{(2,0)}}\ar[dl]_<<<<{\alpha_{(1,0)(2,0)}} && (2,1) \ar@(ul,ur)^{\vep_{(2,1)}}\ar[dl]_<<<<{\alpha_{(1,1)(2,1)}}
\\
(1,0) \ar@(ul,ur)[]^{\vep_{(1,0)}} &&
(1,1) \ar@(ul,ur)[]^{\vep_{(1,1)}} 
\ar[ul]^{\alpha_{(2,0)(1,1)}}\ar[dl]_<<<<{\alpha_{(3,0)(1,1)}}
\\
& (3,0) \ar@(ul,ur)[]^{\vep_{(3,0)}} 
\ar[ul]^{\alpha_{(1,0)(3,0)}} &&
(3,1) \ar[ul]^{\alpha_{(1,1)(3,1)}}\ar@(ul,ur)[]^{\vep_{(3,1)}}
}
\]
bound by the relations 
\begin{align*}
&\vep_{(2,a)} = 0,\\
&\vep_{(1,a)}^2 = \vep_{(3,a)}^2 = 0,\\
&\vep_{(1,a)} \alpha_{(1,a)(3,a)} = 
\alpha_{(1,a)(3,a)}\vep_{(3,a)},\\
&\vep_{(3,0)} \alpha_{(3,0)(1,1)} = 
\alpha_{(3,0)(1,1)}\vep_{(1,1)},
\end{align*}
with $a \in\{0,1\}$, and
\begin{align*}
&\alpha_{(2,0)(1,1)}\alpha_{(1,1)(2,1)} = 0,\\[1mm]
&\vep_{(1,0)}\alpha_{(1,0)(2,0)}\alpha_{(2,0)(1,1)} + \alpha_{(1,0)(2,0)}\alpha_{(2,0)(1,1)}\vep_{(1,1)}\\
&+\ \vep_{(1,0)}\alpha_{(1,0)(3,0)}\alpha_{(3,0)(1,1)} + \alpha_{(1,0)(3,0)}\alpha_{(3,0)(1,1)}\vep_{(1,1)} =0,\\[1mm]
&\vep_{(3,0)}\alpha_{(3,0)(1,1)}\alpha_{(1,1)(3,1)} + \alpha_{(3,0)(1,1)}\alpha_{(1,1)(3,1)}\vep_{(3,1)} =0.
\end{align*}
%

\subsubsection{$\wti{H}$ as a quotient of a tensor algebra} \label{sssec:GabRie1}
It will be useful to have a more intrinsic description of $\wti{H}$ 
in the spirit of Section~\ref{sec7}.
Define
$\wti{C} = (\wti{c}_{(i,a),(j,b)}) \in \Z^{\wti{Q}_0\times\wti{Q}_0} = M_{2n}(\Z)$
by
\[
\wti{c}_{(i,a),(j,b)} :=
\begin{cases}
c_{ij} &\text{if } (a=b),\\
       &\text{or }(a,b)=(0,1)\text{ and } (i,j)\in\Omega^*,\\
       &\text{or }(a,b)=(1,0)\text{ and } (i,j)\in\Omega,\\
0      &\text{otherwise}.
\end{cases}
\]
Clearly, $\wti{C}$ is a Cartan matrix with
symmetrizer $\wti{D} = \diag(c_1,\ldots,c_n,c_1,\ldots,c_n)$,
where $D = \diag(c_1,\ldots,c_n)$ is our symmetrizer for $C$,
and $\wti{\Omega}$ is an orientation of $\wti{C}$. 
Moreover, if $\wti{c}_{(i,a),(j,b)} < 0$ then 
\[
\wti{g}_{(i,a),(j,b)} = g_{ij}, \qquad
\wti{f}_{(i,a),(j,b)} = f_{ij}, \qquad
\wti{k}_{(i,a),(j,b)} = k_{ij}.
\]
As before, one defines the corresponding algebra
$H(\wti{C},\wti{D},\wti{\Omega})$.
Let 
\[
H_{(i,a)} = K[\vep_{(i,a)}]/(\vep_{(i,a)}^{c_i}).
\]
We have isomorphisms 
\[
\eta_{(i,a)}\df H_i\rightarrow H_{(i,a)}
\]
defined by 
$\vep_i\mapsto\vep_{(i,a)}$,
and as before for each $((i,a),(j,b)) \in \wti{\Omega}$ we get an
$H_{(i,a)}$-$H_{(j,b)}$-bimodule
${_{(i,a)}}H_{(j,b)}$
and an 
$H_{(j,b)}$-$H_{(i,a)}$-bimodule
${_{(j,b)}}H_{(i,a)}$.
There are
bimodule isomorphisms
\[
{_{(i,a)}}H_{(j,b)} \cong
\begin{cases}
{_i}H_j & \text{if $a=b$ and $(i,j) \in \Omega$},\\
{_i}H_j & \text{if $(a,b)=(0,1)$ and $(i,j) \in \Omega^*$}.
\end{cases}
\]
via $\eta_{(i,a)}$ and $\eta_{(j,b)}$.

Set 
\[
\wti{S} := \prod_{(i,a) \in \wti{Q}_0} H_{(i,a)}.
\]
Then 
\[
\wti{B} := \bigoplus_{(\bi,\bj) \in \wti{\Omega}} 
{_\bi}H_\bj
\]
is an $\wti{S}$-$\wti{S}$-bimodule, and we have an isomorphism 
\[
T_{\wti{S}}(\wti{B}) \cong H(\wti{C},\wti{D},\wti{\Omega}).
\] 

In case $(i,j)\in\Omega$ we abbreviate ${_iR_j^0}$ for the standard right basis 
of ${_{(i,0)}}H_{(j,0)}$ and  
${_i}L^1_j$ for the standard left basis of ${_{(i,1)}H_{(j,1)}}$. 
Moreover, in this
case we can identify in a obvious way ${_{(j,0)}H_{(i,1)}}$ with 
$\Hom_{H_{(j,0)}}(_{(i,0)}H_{(j,0)},H_{(j,0)})$ and obtain an $H_{(j,0)}$-basis
$(r^*_+)_{r\in {_iR_j^0}}$ of ${_{(j,0)}H_{(i,1)}}$ which is under this identification
dual to ${_iR_j^0}$. 
Similarly, we obtain a dual $H_{(i,1)}$-basis
$(\ell^*_-)_{\ell\in {_iL_j^1}}$ of ${_{(j,0)}H_{(i,1)}}$.

For $j\in Q_0$, define
\begin{equation} \label{eq:Def-wtiH}
\wti{\rho}_j := \sum_{\substack{i \in \Omega(-,j) \\ \ell \in {_i}L^1_j}} \ell_-^* \otimes \ell +  
\sum_{\substack{k \in \Omega(j,-) \\ r \in {_j}R^0_k}} r \otimes r_+^*.
\end{equation}
We have
$\wti{\rho}_j\in e_{(j,0)}\wti{B}\otimes_{\wti{S}}\wti{B}e_{(j,1)}$.
Now, arguing as in Section~\ref{sec7}, we obtain:
\[
\wti{H} \cong T_{\wti{S}}(\wti{B})/(\wti{\rho}_j \mid j \in Q_0).
\]
Similarly to the case of preprojective algebras, 
for $M\in\rep(H(\wti{C},\wti{D},\wti{\Omega}))$ and $j\in Q_0$ 
we can define maps
\begin{align*}
\wti{M}_{j,\inn} 
&= 
(M_{(j,0),(k,c)})_{(k,c)} \df   
\bigoplus_{(k,c)\in\wti{\Omega}((j,0),-)} {_{(j,0)}H_{(k,c)}}\otimes_{H_{(k,c)}}  M_{(k,c)}\ra M_{(j,0)},
\\ 
\wti{M}_{j,\out} 
&=
(\ad_{(i,a),(j,1)}(M_{(i,a),(j,1)}))_{(i,a)}
\df    M_{(j,1)}\ra
\bigoplus_{(i,a)\in\wti{\Omega}(-,(j,1))}  {_{(j,1)}H_{(i,a)}} \otimes_{H_{(i,a)}}   M_{(i,a)}.
\end{align*}
Note that $\wti{\Omega}((j,0),-) = \wti{\Omega}(-,(j,1))$ and thus,
if we identify by a slight abuse $H_{(j,0)}$ with $H_{(j,1)}$,
we can write
\[
\bigoplus_{(k,c)\in\wti{\Omega}((j,0),-)} {_{(j,0)}H_{(k,c)}}\otimes_{H_{(k,c)}}  M_{(k,c)} \;\;\;=
\bigoplus_{(i,a)\in\wti{\Omega}(-,(j,1))} {_{(j,1)}H_{(i,a)}}\otimes_{H_{(i,a)}}  M_{(i,a)}.
\]
With this setup $M\in\rep(H(\wti{C},\wti{D},\wti{\Omega}))$ belongs to
$\rep(\wti{H})$ if and only if 
\begin{equation}\label{eq:tilde H}
 \wti{M}_{j,\inn}\circ\wti{M}_{j,\out}=0
\end{equation}
for all $j\in Q_0$.

\subsubsection{The subalgebras $H_{(0)}$ and $H_{(1)}$} \label{sssec:GabRie4}
For $a=0,1$ set 
\[
\bbo_a:= \sum_{i\in I}{e_{(i,a)}},\qquad
H_{(a)} := \bbo_a \wti{H} \bbo_a. 
\]
Clearly, $H_{(a)}$ is a (non-unitary) subalgebra of $\wti{H}$, and we have
natural isomorphisms $\eta_a\colon H\rightarrow H_{(a)}$ with 
$\eta_a(\vep_i)=\vep_{(i,a)}$ and 
$\eta_a(\alpha_{ij}^{(g)})=\alpha_{(i,a)(j,a)}^{(g)}$.

We obtain for $a\in\{0,1\}$  exact restriction functors 
\[
\Res_a\df \rep(\wti{H}) \to \rep(H_{(a)}), \quad
M\mapsto \bbo_a\wti{H}\otimes_{\wti{H}} M= \bbo_a M.
\]
We will use several times the elementary fact that the functor
\[
\Res_0^*\df \rep(H_{(0)})\ra \rep(\wti{H}), \quad N\mapsto 
\Hom_{H_{(0)}}(\bbo_0\wti{H},N)
\]
is uniquely characterized up to isomorphism as the right adjoint of $\Res_0$.
It is not hard to see that 
\[
\Res_0\circ\Res_0^* \cong \id_{\rep(H_{(0)})}. 
\]

\subsubsection{The $H$-$H$-bimodule $X$}
Define 
\[
X = X(C,D,\Omega) := \bbo_0\wti{H}\bbo_1.
\] 
We regard $X$ as an $H$-$H$-bimodule
via the maps $\eta_0$ and $\eta_1$, 
that is,
\[
hxh' := \eta_0(h)x\eta_1(h'),\qquad (h, h' \in H,\ x\in X). 
\]
Similarly, using $\eta_0$ and $\eta_1$ we can regard $\Res_1\circ\Res_0^*$ as a functor
from $\rep(H)$ to $\rep(H)$.
Then it is easy to see that we have an isomorphism of functors: 
\[
\Res_1\circ\Res_0^* \cong \Hom_{H}(X,-).
\] 
\begin{Thm}\label{thm:functoriso}
The following hold:
\begin{itemize}

\item[(a)] 
For each $M \in \rep(H)$
we have a functorial isomorphism 
\[
\Hom_H(X,M)\cong C^+(M).
\] 

\item[(b)]
For each $M \in \rep_\vp(H)$ we have
functorial isomorphisms
\[
\Hom_H(X,TM) \cong TC^+(M)\cong \tau(M).
\]

\item[(c)]
For each $M \in \rep(H)$ 
we have a functorial isomorphism 
\[
X\otimes_H M \cong C^-(M).
\] 

\item[(d)]
For each $M \in \rep_\vp(H)$ we have functorial isomorphisms
\[
X\otimes_H TM \cong TC^-(M) \cong \tau^-(M).
\]

\end{itemize}
\end{Thm}

\subsection{Proof of Theorem~\ref{thm:functoriso}(a)} 
\label{pf:fiso2ab}
We follow the hints from \cite[Section~5.5]{Ga3}.
For $l\in Q_0\cup\{0\}$ we define  idempotents in $\wti{H}$
\[
\bbo^{(l)}  := \sum_{i>l} e_{(i,0)} +  \sum_{i\leq l} e_{(i,1)},\qquad
\bbo^{(l)}_0:= \bbo_0+ \sum_{i\leq l} e_{(i,1)},
\]
and the corresponding  (non-unitary) subalgebras
\[
H^{(l)}:= \bbo^{(l)} \wti{H}\bbo^{(l)}, \qquad
\wti{H}^{(l)}:= \bbo^{(l)}_0 \wti{H}\bbo^{(l)}_0.
\]
Clearly, $H^{(0)}=\wti{H}^{(0)}=H_{(0)}$, $H^{(n)}=H_{(1)}$, $\wti{H}^{(n)}=\wti{H}$, and an easy calculation
shows that, using the notation of equation (\ref{eq:s_i(H)}),
\[
H^{(l)}= s_l \cdots s_2s_1(H^{(0)}), \qquad (l\in Q_0).
\]
Moreover $\bbo^{(l)}\in\wti{H}^{(l)}$ and thus 
$H^{(l)}\subset\wti{H}^{(l)}\supset  \wti{H}^{(l-1)}$ for $l\in Q_0$.
We study the corresponding restriction functors: 
\begin{align*}
\Res^{(l)} &\df \rep(\wti{H})\ra\rep(H^{(l)}), \ \qquad
M\mapsto \bbo^{(l)}\wti{H}\otimes_{\wti{H}} M, \\[1mm] 
\Res_{(l,m)} &\df \rep(\wti{H}^{(m)})\ra\rep(\wti{H}^{(l)}), \quad
M\mapsto \bbo_0^{(l)}\wti{H}^{(m)}\otimes_{\wti{H}^{(m)}} M, \quad (l<m).
\end{align*}
Obviously, $\Res_{(l,m)}$ admits a right adjoint 
\[
\Res_{(l,m)}^*(-) = \Hom_{\wti{H}^{(l)}}(\bbo_0^{(l)}\wti{H}^{(m)},-)
\] 
and 
\[
\Res_0=\Res_{(0,1)}\circ\Res_{(1,2)}\circ\cdots\circ\Res_{(n-1,n)}. 
\]
Thus we have
\[
\Res_0^* = \Res^*_{(n-1,n)}\circ\cdots\circ\Res^*_{(1,2)}\circ\Res^*_{(0,1)}.
\]

\begin{Lem} \label{lem9.2}
With the above notations we have functorial isomorphisms
\[
\Res^{(i)}\circ\Res^*_{(i-1,i)}(M)\cong F_i^+\circ\Res^{(i-1)}(M)
\]
for all $M\in\rep(\wti{H}^{(i-1)})$ and $i\in Q_0 = \{1,\ldots,n\}$.
\end{Lem}

\begin{proof}
Note that naturally 
\[
\Res_{(i-1,i)}\circ\Res^*_{(i-1,i)}(M) \cong M
\] 
for all $M\in\rep(\wti{H}^{(i-1)})$. 
Now there is a unique functor 
\[
R^*_{(i-1,i)}\df\rep(\wti{H}^{(i-1)}) \ra \rep(\wti{H}^{(i)})
\]
satisfying the two following conditions for all $M\in\rep(\wti{H}^{(i-1)})$: 
\[
\Res_{(i-1,i)}\circ R^*_{(i-1,i)}(M) = M,
\qquad
\Res^{(i)}\circ R^*_{(i-1,i)}(M) = F_i^+\circ\Res^{(i-1)}(M).
\]
Indeed, the first condition fixes the restriction of $R^*_{(i-1,i)}(M)$ to
$\wti{H}^{(i-1)}$ and the second one fixes the restriction of
$R^*_{(i-1,i)}(M)$ to $H^{(i)}$. Because of the definitions of $H^{(i)}$ and
$\wti{H}^{(i-1)}$, this determines completely the structure of
$R^*_{(i-1,i)}(M)$, and gives uniqueness. 
Note that the quivers of $\wti{H}^{(i-1)}$ and of
$H^{(i)}$ contain some common arrows, but the representations $M$ and
$F_i^+ \circ \Res^{(i-1)}(M)$ are the same for those arrows, by definition of
$F_i^+$. So $R^*_{(i-1,i)}(M)$ is indeed a representation of
$H(\wti{C},\wti{D},\wti{\Omega})$, supported on the vertices and arrows of $\wti{H}^{(i)}$.
Finally, this representation satisfies the relation (\ref{eq:tilde H}) for $j=i$, 
because again of the definition of $F_i^+$, so $R^*_{(i-1,i)}(M)\in \rep(\wti{H}^{(i)})$.

To prove the lemma, we have to show that the above functor $R^*_{(i-1,i)}$ is isomorphic to  
$\Res^*_{(i-1,i)}$, or equivalently, that $R^*_{(i-1,i)}$ is right adjoint to $\Res_{(i-1,i)}$.
To do so, let $N\in \rep(\wti{H}^{(i)})$ and $M\in \rep(\wti{H}^{(i-1)})$,
and consider the natural map
\[
\Hom_{\wti{H}^{(i)}}(N, R^*_{(i-1,i)}(M))
\to \Hom_{\wti{H}^{(i-1)}}(\Res_{(i-1,i)}(N), M) 
\]
obtained by restricting $f\df N \to R^*_{(i-1,i)}(M)$ to $\Res_{(i-1,i)}(N)$.
We have to show that this restriction is in fact bijective. 
That is, for
$g \in \Hom_{\wti{H}^{(i-1)}}(\Res_{(i-1,i)}(N), M)$
we have to show that there exists a unique
$g_{(i,1)} \in \Hom_{H_i} (N_{(i,1)}, \Ker(\wti{M}_{i,\inn}))$ which lifts
$g$ to an element of $\Hom_{\wti{H}^{(i)}}(N, R^*_{(i-1,i)}(M))$.

Now, let
\[
N_{i,+}:=
\bigoplus_{(k,c)\in\wti{\Omega}((i,0),-)} {}_{(i,0)}H_{(k,c)}\otimes_{H_{(k,c)}} N_{(k,c)}
\]
denote the domain of $\wti{N}_{i,\inn}$, and similarly let $M_{i,+}$ 
denote the domain of $\wti{M}_{i,\inn}$.
By the definition of $\wti{H}^{(i-1)}$-homomorphisms, we have a commutative 
diagram
\[
\xymatrix@-1ex{
& N_{(i,1)} \ar[r]^{\wti{N}_{i,\out}} &
N_{i,+}\ar[rr]^{\wti{N}_{i,\inn}}\ar[d]^{g_{i,+}}
&&
N_{(i,0)}\ar[d]^{g_{(i,0)}}
\\
0 \ar[r] & \Ker(\wti{M}_{i,\inn}) \ar[r] &
M_{i,+} \ar[rr]^{\wti{M}_{i,\inn}}
&& M_{(i,0)}
}
\]
where the bottom row is exact by construction, and in the top row the
composition is zero since $N$ is a $\wti{H}^{(i)}$-module.
Thus $\wti{M}_{i,\inn} \circ g_{i,+} \circ \wti{N}_{i,\out} = 0$.
By the universal property of $\Ker(\wti{M}_{i,\inn})$ there exists a unique 
morphism of $H_i$-modules $g_{(i,1)}\df N_{(i,1)} \to \Ker(\wti{M}_{i,\inn})$
which makes the left-hand square commutative.
\end{proof}

We can now finish the proof of Theorem~\ref{thm:functoriso}(a).
Using $n$ times Lemma~\ref{lem9.2}, for $M\in\rep(H)$ (regarded as a 
representation of $H_0$) we have
\begin{align*}
 \Hom_H(X,M) &= \Res_1\circ\Res_0^*(M)\\
 &= \Res^{(n)}\circ\Res_{(n-1,n)}^*\circ \cdots \circ \Res_{(0,1)}^*(M)\\
 &= F_n^+\circ\Res^{(n-1)}\circ \Res_{(n-2,n-1)}^*\circ \cdots \circ \Res_{(0,1)}^*(M)\\
 &= \quad\cdots \\
 &= F_n^+\circ F_{n-1}^+ \circ \cdots \circ F_1^+ \circ \Res^{(0)}(M) \\
 &= C^+(M).
 \end{align*}

\subsection{Proof of Theorem~\ref{thm:functoriso}(b)}\label{pf:fiso2b}
We follow the idea of \cite[Section~5.4]{Ga3}, and start by giving an 
alternative description of $\Res_0^*$. This is done by constructing in two steps
a functor $R_0^*\df \rep(H_{(0)})\to \rep(\wti{H})$, and then showing that $R_0^*$ 
is right adjoint to $\Res_0$.

Let $M\in\rep(H_{(0)})$. We first define 
$\wti{M}\in\rep(H(\wti{C},\wti{D},\wti{\Omega}))$
by requiring that
\begin{align*}
\Res_0(\wti{M}) &= M,
\\
\Res_1(\wti{M}) &= \bigoplus_{(k,l)\in\Omega} 
\Hom_{H_{(l,0)}}({_{(l,0)}H_{(k,1)}}\otimes_{H_{(k,1)}}e_{(k,1)}H_1, M_{(l,0)}).
\end{align*}
(Note that $H_{(0)}$ and $H_{(1)}$ can also be regarded as subalgebras of $H(\wti{C},\wti{D},\wti{\Omega})$,
so we allow ourselves, by some abuse of notation, to continue to denote the restriction functors
$\rep(H(\wti{C},\wti{D},\wti{\Omega})) \ra \rep(H_{(a)})$ by $\Res_a$.) 
It remains to define, for $(i,j)\in\Omega$, the structure map
\[
\wti{M}_{(j,0),(i,1)}\df {_{(j,0)}H_{(i,1)}}\otimes_{H_{(i,1)}} \wti{M}_{(i,1)}\longrightarrow
\wti{M}_{(j,0)}=M_{(j,0)}. 
\]
This is given by the following composition:
\begin{align*}
&{_{(j,0)}H_{(i,1)}}\otimes_{H_{(i,1)}} \left(\bigoplus_{(k,l)\in\Omega} 
\Hom_{H_{(l,0)}}({_{(l,0)}H_{(k,1)}}\otimes e_{(k,1)}H_1 e_{(i,1)},M_{(l,0)})\right)\\
&\overset{\text{proj.}}{\longrightarrow} {_{(j,0)}H_{(i,1)}}\otimes_{H_{(i,1)}}
\Hom_{H_{(j,0)}}({_{(j,0)}H_{(i,1)}}\otimes e_{(i,1)}H_1 e_{(i,1)},M_{(j,0)})\\
&=\quad {_{(j,0)}H_{(i,1)}}\otimes_{H_{(i,1)}}
\Hom_{H_{(j,0)}}({_{(j,0)}H_{(i,1)}},M_{(j,0)})
\overset{\text{eval.}}{\longrightarrow} M_{(j,0)},
\end{align*}
where the first map is the projection on the direct summand indexed by $(k,l) = (i,j)$
and the second map is the evaluation $h\otimes \varphi \mapsto \varphi(h)$.

Secondly, we define a subrepresentation $R_0^*(M)$ of $\wti{M}$ as follows. 
We set
\[
(R_0^*(M))_{(i,0)} =\wti{M}_{(i,0)}=M_{(i,0)}, \qquad (i\in Q_0),
\]
and we define $(R_0^*(M))_{(h,1)}$ as the subspace of $\wti{M}_{(h,1)}$
consisting of all 
\[
(\mu_{k,l}^h)_{(k,l)\in\Omega} \in \bigoplus_{(k,l)\in\Omega} 
\Hom_{H_{(l,0)}}({_{(l,0)}H_{(k,1)}}\otimes e_{(k,1)}H_{(1)} e_{(h,1)},M_{(l,0)})
\]
such that, for all $l\in Q_0$ and
$n^{(1)}\in e_{(l,1)} H_{(1)} e_{(h,1)}$ the following relation holds:
\begin{equation} \label{eq:R0*}
\sum_{\substack{k\in\Omega(-,l)\\\ell\in {_{k}L^1_{l}}}} 
\mu_{k,l}^h(\ell^*_-\otimes \ell\cdot n^{(1)}) 
+ \sum_{\substack{m \in\Omega(l,-)\\ r\in {_{l}R^0_{m}}}}
M_{(l,0),(m,0)} (r\otimes \mu_{l,m}^h(r^*_+\otimes n^{(1)})) = 0.
\end{equation}
Here, we use the notation from 
Section~\ref{sssec:GabRie1}. It is straightforward to
check that $R_0^*(M)$ is an $H(\wti{C},\wti{D},\wti{\Omega})$-subrepresentation
of $\wti{M}$. 
Moreover, $R^*_0(M)$ is in fact a representation
of $\wti{H}$. To see this, we check the defining relations~\eqref{eq:Def-wtiH}
with the help of the special case $n^{(1)}=e_{(l,1)}$ of  
equation~\eqref{eq:R0*}. In fact, if we apply $\wti{\rho}_j$ to
$\mu^{(j)}=(\mu_{k,l}^j)_{(k,l)\in\Omega}\in R_0^*(M)_{(j,1)}$ we deduce from
the definitions that
\begin{multline*}
\sum_{\substack{i\in\Omega(-,j)\\ \ell\in{_iL^1_j}}}\!\!
\wti{M}_{(j,0),(i,1)}(\ell_-^*\otimes\wti{M}_{(i,1),(j,1)}(\ell\otimes \mu^{(j)}))
\,+\!\sum_{\substack{k\in\Omega(j,-)\\ r\in{_jR^0_k}}}\!\!
M_{(j,0),(k,0)}(r\otimes\wti{M}_{(k,0),(j,1)}(r^*_+ \otimes\mu^{(j)}))
\\
=\sum_{\substack{i\in\Omega(-,j)\\ \ell\in{_iL^1_j}}}\!\mu_{i,j}^j(\ell_-^*\otimes \ell\cdot e_{(j,1)})
\,+\!\sum_{\substack{k\in\Omega(j,-)\\ r\in{_jR^0_k}}}\!
M_{(j,0),(k,0)}(r\otimes\mu_{j,k}^j(r_+^*\otimes e_{(j,1)})) =0,
\end{multline*}
as required.

Thus, we have obtained a functor 
$R_0^*\df \rep(H_{(0)}) \to \rep(\wti{H}), \ M\mapsto R_0^*(M)$.
It will follow from the next lemma that $R_0^*$ is isomorphic to $\Res_0^*$.

\begin{Lem} \label{lem:radj}
$R_0^*$ is right adjoint to $\Res_0$.
\end{Lem}

\begin{proof}
Let $N\in\rep(\wti{H})$ and $M\in\rep(H_{(0)})$. 
Consider $\chi\in \Hom_{\wti{H}}(N, R_0^*(M))$. Thus, $\chi$ is given by a family
of maps
\[
\chi^{(i,a)}\in \Hom_{H_{(i,a)}}(N_{(i,a)}, (R_0^*(M))_{(i,a)}),
\qquad ((i,a)\in\wti{Q}_0),
\]
subject to the usual commutativity relations. By the construction of $R_0^*(M)$
we have more explicitly for all $i\in Q_0$ and $n_{(i,1)}\in N_{(i,1)}$:
\begin{align*}
&\chi^{(i,0)} \in \Hom_{H_{(i,0)}}(N_{(i,0)},M_{(i,0)}),
\\
&\chi^{(i,1)}(n_{(i,1)}) \in \bigoplus_{(k,l)\in\Omega} 
\Hom_{H_{(l,0)}}({_{(l,0)}H_{(k,1)}}\otimes e_{(k,1)} H_1 e_{(i,1)}, M_{(l,0)}).
\end{align*}
Let us denote by $\chi^{(i,1)}_{k,l}(-,n_{(i,1)})$ the $(k,l)$-component of 
$\chi^{(i,1)}(n_{(i,1)})$.
These maps are subject to the following relations for $(i,j)\in\Omega$,
$\ell^{(a)}\in {_iL_j^a}$, $r\in {_iR_j^0}$:
\begin{align} 
\chi^{(i,0)}(N_{(i,0),(j,0)}(\ell^{(0)}\otimes n_{(j,0)})) &=
M_{(i,0),(j,0)}(\ell^{(0)}\otimes\chi^{(j,0)}(n_{(j,0)})), \label{eq:ad1}
\\
\chi^{(j,0)}(N_{(j,0),(i,1)}(r_+^*\otimes n_{(i,1)})) &=
\chi_{i,j}^{(i,1)}(r_+^*\otimes e_{(i,1)},n_{(i,1)}),  \label{eq:ad2}
\\
\chi_{k,l}^{(i,1)}(-, N_{(i,1),(j,1)}(\ell^{(1)}\otimes n_{(j,1)})) &=
\chi_{k,l}^{(j,1)}(-\cdot \ell^{(1)}, n_{(j,1)}).  \label{eq:ad3}
\end{align}
Equation~\eqref{eq:ad1} means that we have indeed a well-defined restriction
\[
r_{N,M}\df \Hom_{\wti{H}}(N, R_0^*(M))\ra \Hom_{H_{(0)}}(\Res_0(N), M).
\]
Combining~\eqref{eq:ad2} and~\eqref{eq:ad3} we see that the maps
$\chi_{k,l}^{(j,1)}$ for $(k,l)\in\Omega$ and $j\in Q_0$ are determined
by the maps $\chi^{(i,0)}$ with $i\in Q_0$, in other words $r_{N,M}$ is injective.

By the same token we see that for each $\chi^{(0)}\in\Hom_{H_{(0)}}(\Res_0(N), M)$ there
exists $\wti{\chi}\in\Hom_{H(\wti{C},\wti{D},\wti{\Omega})}(N,\wti{M})$
which restricts to $\chi^{(0)}$. 
We leave it as an exercise to show that if $N\in\rep(\wti{H})$ then
$\Ima(\wti{\chi})\subset R_0^*(M)$. 
Thus, $r_{N,M}$ is bijective.
\end{proof}

\begin{Prop} \label{prop:gabriel1}
For $M \in \rep_\vp(H)$ we have
\[
\tau(TM) \cong {\Res_1}\circ R_0^*(M),
\]
where in the right-hand side $H_{(0)}$ and $H_{(1)}$ are identified with $H$
by means of the isomorphisms $\eta_0$ and $\eta_1$.
\end{Prop}

\begin{proof}
Since $M$ is locally free, $TM$ is also locally free and 
Corollary~\ref{cor:resolHM} provides a projective resolution:
\[
0 \to
\bigoplus_{(j,i) \in \Omega} H e_j \otimes {_j}H_i \otimes (TM)_i
\xrightarrow{d \otimes TM}
\bigoplus_{k=1}^n He_k \otimes (TM)_k \xrightarrow{{\rm mult}} TM \to 0 
\]
Therefore, by definition of the Auslander-Reiten translation $\tau$, 
we know that $\tau(TM)$ is isomorphic to $\Ker(\Du\Hom_H(d \otimes_H TM,H))$.

On the other hand, the construction of $R_0^*(M)$ shows that
${\Res_1}\circ R_0^*(M)$ can be identified with the kernel of the map
\[
d_M^*\df \bigoplus_{(j,i) \in \Omega} \Hom_{H_i}({_j}H_i^* \otimes e_jH,M_i)
\quad\longrightarrow\quad \bigoplus_{k=1}^n \Hom_{H_k}(e_kH,M_k)
\]
whose $(j,i)$-component is defined by
\begin{equation}\label{eq1}
\varphi_{(j,i)} \mapsto 
\sum_{\ell \in {_j}L_i} \varphi_{(j,i)}(\ell^* \otimes_j \ell \cdot -) 
+ \sum_{r \in {_j}R_i} M_{ji}(r \otimes \varphi_{(j,i)}(r^* \otimes_j -)).
\end{equation}
Indeed, ${\Res_1}\circ R_0^*(M)$ is the subspace of ${\Res_1}(\wti{M})$
defined by equations (\ref{eq:R0*}).
Our goal is to identify under the trace pairing the map 
$\Du\Hom_H(d \otimes TM,H)$ with $d_M^*$.
For $(j,i)\in\Omega$, the restriction 
$d\otimes TM \df He_j\otimes {}_jH_i\otimes M_i \to He_i\otimes M_i \oplus He_j\otimes M_j$ is 
given by
\begin{equation}\label{eq2}
 (d\otimes TM)(p\otimes h\otimes m) = ph\otimes m + p\otimes M_{ji}(h\otimes m).
\end{equation}
(Note the plus sign, coming from the twist map $T$).
Using adjunction we have 
\[
\Hom_H(He_i\otimes M_i,H) \cong \Hom_{H_i}(M_i,\Hom_H(He_i,H)) \cong \Hom_{H_i}(M_i,e_iH),
\]
so under the trace pairing we get
\[
\Du\Hom_H(He_i\otimes M_i \oplus He_j\otimes M_j,H) \cong \Hom_{H_i}(e_iH,M_i) \oplus \Hom_{H_j}(e_jH,M_j).
\]
Similarly, 
\[
\Du\Hom_H(He_j\otimes {}_jH_i\otimes M_i) \cong \Hom_{H_j}(e_jH,{}_jH_i\otimes M_i)
 \cong \Hom_{H_i}({}_iH_j \otimes e_jH, M_i),
\]
where the second isomorphism is given by $\ad_{ij}^{-1}$. Hence $\Du\Hom_H(d \otimes TM,H)$ can be identified with 
a map from $\oplus_{(j,i)\in\Omega} \Hom_{H_i}({}_iH_j \otimes e_jH, M_i)$ to $\oplus_k \Hom_{H_k}(e_kH,M_k)$,
and comparison between (\ref{eq1}) and (\ref{eq2}) shows that this map is 
indeed $d_M^*$.
\end{proof}

Now we are ready to prove part~(b) of Theorem~\ref{thm:functoriso}.
By Lemma~\ref{lem:radj} and the uniqueness of adjoint functors we have
a functorial isomorphism $R_0^*(-) \cong \Res_0^* (-)$.
Hence, by Proposition~\ref{prop:gabriel1}, if $M\in\rep_\vp(H)$ we have 
\[
\Hom_{H}(X,TM) = \Res_1\circ\Res_0^*(TM) \cong {\Res_1}\circ R_0^*(TM)\cong\tau(T^2M)=\tau(M). 
\]
This proves Theorem~\ref{thm:functoriso}~(b).

\subsection{Proof of Theorem~\ref{thm:functoriso}(c),(d)} 
\label{pf:fiso2cd}
Clearly, (c) follows from (a) since $C^-$ is left adjoint
to $C^+$. 
In order to show (d), let $M,N \in \rep_\vp(H)$.
Recall, that this implies that both, $M$ and $N$, have projective and injective
dimension at most $1$. 
We obtain functorial isomorphisms
\begin{align*}
\Hom_{H}(\tau^-(M), N) &\cong \Hom_H(M,\tau(N))
\\
&\cong \Hom_H(M,C^+(TN))
\\
&\cong \Hom_H(C^-(TM),N).
\end{align*}
The first isomorphism is obtained from the Auslander-Reiten formulas,
the second follows from (b), and the third isomorphism is just
the adjunction map.

With the usual $H$-$H$-bimodule structure on $\Du(H)$ we obtain 
a functorial 
isomorphism of
right $H$-modules 
\[
\Du(X) \cong \Hom_H(X,\Du(H))
\] 
for all (left) $H$-modules $X$.
Now, in our situation $\Du(H)$ is locally free, thus taking $N=\Du(H)$ in
the above chain of functorial isomorphisms we get
\[
\tau^-(M) \cong C^{-}(TM) \cong X^T\otimes_H M
\]
where the last isomorphism comes from (c). 
This proves Theorem~\ref{thm:functoriso}~(d).

\subsection{Another description of $\Pi(C,D)$ as a tensor algebra}\label{subsec:9.6}
Let $\Pi = \Pi(C,D)$.
Recall from Section~\ref{sec7} that $\Pi_1$ is 
the subspace of $\Pi$ of elements of degree $1$.
Let $X^T$ be the twisted version of the $H$-$H$-bimodule $X$, where 
the bimodule structure is defined by
\[
hxh' := h xT(h'),\qquad (h, h' \in H,\ x\in X).
\]

\begin{Thm}\label{thm:10.6}
We have isomorphisms of $H$-$H$-bimodules
\[
\Pi_1 \cong X^T \cong \Ext^1_H(\Du(H),H).
\]
\end{Thm}

\begin{proof}
Note that the bimodule isomorphism $\Pi_1 \cong X^T$ follows directly 
from the definitions.
On the other hand, we have by Theorem~\ref{thm:functoriso}(d) for locally free
modules $M$ a functorial isomorphism 
\begin{equation}
X^T\otimes_H M \cong \tau^-(M) \cong \Ext^1_H(\Du(H),M). 
\label{eq:XT-Ext}
\end{equation}
Note that the functor 
\[
\Ext_H^1(\Du(H),-)\colon\rep(H)\ra\rep(H)
\] 
is right exact
since $\pdim(\Du(H)) \le 1$. 
For $M \in \rep(H)$ let
\begin{equation}\label{eq:9.5}
P_1 \to P_0 \to M \to 0
\end{equation}
be a projective presentation of $M$.
Applying the right exact functors $X^T \otimes_H -$ and
$\Ext_H^1(\Du(H),-)$ to (\ref{eq:9.5}) yields a functorial commutative
diagram 
\[
\xymatrix{
X^T \otimes_H P_1 \ar[d]^{\eta_{P_1}} \ar[r] &
X^T \otimes_H P_0 \ar[d]^{\eta_{P_0}} \ar[r] &
X^T \otimes_H M \ar[d]^{\eta_M} \ar[r] & 0
\\
\Ext_H^1(\Du(H),P_1) \ar[r] &
\Ext_H^1(\Du(H),P_0) \ar[r] &
\Ext_H^1(\Du(H),M) \ar[r] & 0
}
\]
with exact rows.
Since the restrictions of $X^T \otimes_H -$ and $\Ext_H^1(\Du(H),-)$
to $\rep_\vp(H)$ are isomorphic, we get that $\eta_{P_0}$ and
$\eta_{P_1}$ are isomorphisms. 
This implies that $\eta_M$ is an isomorphism as well.
It follows that the functors
$X^T \otimes_H -$ and $\Ext_H^1(\Du(H),-)$ are isomorphic.
From the canonical isomorphism of $H$-$H$-bimodules
$\Ext^1_H(\Du(H),H)\otimes_H H\cong \Ext_H^1(\Du(H),H)$ we conclude that the 
right exact functors $\Ext_H^1(\Du(H),-)$ and $\Ext^1_H(\Du(H),H)\otimes_H-$ 
are isomorphic. This implies that  $\Ext_H^1(\Du(H),H)$ and $X^T$ are 
isomorphic as $H$-$H$-bimodules.
\end{proof}

\begin{Cor}\label{cor:10.7}
We have $K$-algebra isomorphisms
\[
\Pi \cong T_H(X^T) \cong T_H(\Ext_H^1(\Du(H),H)).
\]
\end{Cor}

\begin{proof}
Combine Theorem~\ref{thm:10.6} and Proposition~\ref{prop:7.5}.
\end{proof}

\begin{Cor}\label{cor:10.8}
For $M \in \rep(H)$ there are functorial isomorphisms 
\[
\Hom_H(X^T,M) \cong \Du\Ext_H^1(M,H)
\text{\;\;\; and \;\;\;}
X^T \otimes_H M  \cong \Ext_H^1(\Du(H),M).
\]
\end{Cor}

\begin{proof}
We get
the second isomorphism from the proof of Theorem~\ref{thm:10.6}.
The first isomorphism follows then by adjunction.
\end{proof}

\subsection{The morphism categories $\cC(1,TC^+)$ and $\cC(TC^-,1)$}\label{subsec:9.7}
Let $\Pi = \Pi(C,D)$.
Following a definition due to Ringel \cite{Ri4} ,
we define a category $\cC(1,TC^+)$ as follows.
Its objects are the $H$-module homomorphims $M \to TC^+(M)$, where
$M \in \rep(H)$ and
the morphisms in $\cC(1,TC^+)$ are given by commutative diagrams
\[
\xymatrix{
M \ar[r]^>>>>f\ar[d]_h & TC^+(M) \ar[d]^{TC^+(h)}\\
N \ar[r]^>>>>g & TC^+(N).
}
\]
Similarly, let $\cC(TC^-,1)$ be the category with objects the
$H$-module homomorphisms $TC^-(M) \to M$.

\begin{Thm}\label{Ringelgen1}
The categories $\rep(\Pi)$, $\cC(1,TC^+)$ and $\cC(TC^-,1)$
are isomorphic.
\end{Thm}

\begin{proof}
It follows from Proposition~\ref{prop:refl1} that $(TC^-,TC^+)$ is a pair
of adjoint functors $\rep(H) \to \rep(H)$.
Now \cite[Lemma~1]{Ri4} implies that the categories $\cC(1,TC^+)$ and $\cC(TC^-,1)$
are isomorphic.
It follows from
Theorem~\ref{thm:functoriso}(c) that there is a functorial isomorphism
$X^T \otimes_H - \cong TC^-(-)$, and by Corollary~\ref{cor:10.7} we have
$\Pi \cong T_H(X^T)$.
Now  \cite[Lemma~12]{Ri4} gives an isomorphism of categories
$\cC(TC^-,1) \cong \rep(\Pi(C,D))$.
\end{proof}

One can also adapt Ringel's proof of \cite[Theorem~B]{Ri4} to obtain
a more direct proof of 
Theorem~\ref{Ringelgen1}.

\subsection{The kernel of the Coxeter functor}\label{subsec9.8}
As before, let $H = H(C,D,\Omega)$.
Recall that $H$ is a $1$-Iwanaga-Gorenstein algebra with
the subcategory
\[
\cG\cP(H) = \{ M \in \rep(H) \mid \Ext_H^1(M,H) = 0 \}
\]
of Gorenstein-projective modules.

As an immediate consequence of 
Theorem~\ref{thm:functoriso}(a), 
Corollary~\ref{cor:10.8},
and the definition of $C^+(-)$ 
we get the following result.
Here, the map $M_{i,\inn}$ is defined as in
Section~\ref{Minn}, since we can regard the $H$-module $M$ 
also as a module over $\Pi(C,D)$.

\begin{Thm}\label{thm:GP}
For an $H$-module $M$ the following are equivalent:
\begin{itemize}

\item[(i)]
$M \in \cG\cP(H)$;

\item[(ii)]
$C^+(M) = 0$;

\item[(iii)]
$M_{i,\inn}$ is injective for all $1 \le i \le n$.

\end{itemize}
\end{Thm}

If $C$ is symmetric, then the equivalence of (i) and (iii) in Theorem~\ref{thm:GP} is a special case
of \cite[Theorem~5.1]{LuZ}.
For $C$ symmetric and $D = \diag(2,\ldots,2)$
the category $\cG\cP(H)$
has been studied in detail in \cite{RiZ}.


\section{$\tau$-locally free $H$-modules}\label{sec12}


\subsection{}\label{sec11.1}
Let $M$ be an indecomposable $H$-module.
Recall that $M$ is
$\tau$-\emph{locally free} if $\tau^k(M)$ is locally free for all
$k \in \Z$.
Furthermore, $M$ is called \emph{preprojective} (resp. \emph{preinjective})
if there exists some
$k \ge 0$ such that $M \cong \tau^{-k}(P)$ (resp.
$M \cong \tau^k(I)$) for some indecomposable projective
$H$-module $P$ (resp. indecomposable injective
$H$-module $I$). 
A $\tau$-locally free $H$-module $M$ is 
\emph{$\tau$-locally free regular} if
$M$ is neither preprojective nor preinjective.
(An indecomposable module $M$ over a finite-dimensional algebra
is called \emph{regular} if $\tau^k(M) \not= 0$ for all $k \in \Z$.)

Let $\cC$ be a $K$-linear category.
The \emph{stable category} $\underline{\cC}$ (resp. $\ov{\cC}$)
is the quotient category of $\cC$ modulo the ideal of
all morphisms factoring through projective (resp. injective) objects.

\begin{Prop}\label{prop:12.1}
The restriction of $\tau(-)$ yields an equivalence
of stable categories
\[
\underline{\rep}_\vp(H) \to \ov{\cF}(H)
\]
where 
$\cF(H) := \{ M \in \rep(H) \mid \Hom_H(\Du(H),M) = 0 \}$, and
$\tau^-(-)$ yields an equivalence
of stable categories
\[
\ov{\rep}_\vp(H) \to \underline{\cG}(H)
\]
where 
$\cG(H) := \{ M \in \rep(H) \mid \Hom_H(M,H) = 0 \}$.
\end{Prop}

\begin{proof}
Combine
Proposition~\ref{vpdim} with 
\cite[Lemma~4.1]{AR} and its dual.
\end{proof}

\begin{Cor}\label{cor:12.2} 
For an indecomposable $M \in \rep(H)$ the following are
equivalent:
\begin{itemize}

\item[(i)]
$M \in \rep_\vp(H)$;

\item[(ii)]
$\Hom_H(\tau^-(M),H) = 0$; 

\item[(iii)]
$\Hom_H(\Du(H),\tau(M)) = 0$.

\end{itemize}
\end{Cor}

\begin{Cor}\label{cor:12.3}
For an indecomposable $M \in \rep(H)$ the following are
equivalent:
\begin{itemize}

\item[(i)]
$M$ is $\tau$-locally free;

\item[(ii)]
$\Hom_H(\tau^-(\tau^k(M)),H) = 0$ for all $k \in \Z$; 

\item[(iii)]
$\Hom_H(\Du(H),\tau(\tau^k(M))) = 0$ for all $k \in \Z$.

\end{itemize}
\end{Cor}

\begin{proof}
Recall that for an indecomposable module $M$ we have 
$\tau(\tau^-(M)) \cong M$ if and only if $M$ is not injective, and
$\tau^-(\tau(M)) \cong M$ if and only if $M$ is not projective.
Now the statement is a direct consequence of Corollary~\ref{cor:12.2}.
\end{proof}

\begin{Prop}\label{prop:12.5}
Let $M \in \rep_\vp(H)$ be indecomposable and rigid.
Then $M$ is $\tau$-locally free and $\tau^k(M)$ is rigid for all
$k \in \Z$. 
\end{Prop}

\begin{proof}
By Theorem~\ref{thm:functoriso} we know that 
$\tau^k(M) \cong T^kC^k(M)$.
Now the result follows from Proposition~\ref{prop:reflvp}.
\end{proof}

Recall from Section~\ref{subsec-Coxeter} the Coxeter matrix $\Phi_H$.

\begin{Prop}\label{prop:12.8}
For a $\tau$-locally free module $M \in \rep(H)$
the following hold:
\begin{itemize}

\item[(i)]
If $\tau^k(M) \not= 0$ for some $k \in \Z$, then
$\rkv(\tau^k(M)) = (D^{-1}\Phi_HD)^k(\rkv(M))$.

\item[(ii)]
If $\tau^k(M) \not= 0$ for some $k \in \Z$ and $\rkv(M)$ is contained in $\Delta_\real^+(C)$ or $\Delta_\imag^+(C)$, 
then $\rkv(\tau^k(M))$ is in $\Delta_\real^+(C)$
or $\Delta_\imag^+(C)$, respectively.

\end{itemize}
\end{Prop}

\begin{proof}
Part (i) follows from \cite[Section~2.4, p.75]{Ri3} combined with 
Corollary~\ref{cor:12.3} and the fact
that $\tau$-locally free $H$-modules and their $\tau^k$-translates have
projective and injective dimension at most $1$.
To prove (ii),
let $i$ be a sink (resp. source) in $Q^\circ(C,\Omega)$.
Then for a $\tau$-locally free module $M$ with $M \not\cong E_i$ the map
$M_{i,\inn}$ is surjective (resp. $M_{i,\out}$ is injective).
Now the result follows from
Proposition~\ref{prop:refl4}
and Theorem~\ref{thm:functoriso}(b),(d).
\end{proof}

\begin{Prop}\label{prop:12.9}
Let $M$ be a preprojective or preinjective $H$-module.
Then the following hold:
\begin{itemize}

\item[(i)]
$M$ is $\tau$-locally free and rigid;

\item[(ii)]
$\rkv(M) \in \Delta^+_\real(C)$;

\item[(iii)]
If $M$ and $N$ are preprojective or preinjective $H$-modules
with $\dimv(M) = \dimv(N)$, then $M \cong N$.

\end{itemize}
\end{Prop}

\begin{proof}
By definition we have $M \cong \tau^{-k}(P_i)$ or
$M \cong \tau^k(I_i)$ for some $k \ge 0$ and some
$1 \le i \le n$.
The modules $P_i$ and $I_i$ are indecomposable, locally free and rigid.
Thus by Proposition~\ref{prop:12.5} the module 
$M$ is $\tau$-locally free and rigid.
We know from Section~\ref{sec3.3} that $\rkv(P_i),\rkv(I_i) \in
\Delta_\real^+(C)$.
Now part (ii) follows from
Proposition~\ref{prop:12.8}(ii), and
(iii) is a consequence of Lemmas~\ref{lem:ct1} and \ref{lem:ct2}.
\end{proof}

\begin{Lem}
Assume $C$ is connected and not of Dynkin type.
Let $X$ be a preprojective, $Y$ a $\tau$-locally free regular and $Z$ a preinjective $H$-module.
Then we have $\Hom_H(Z,Y) = 0$, $\Hom_H(Y,X) = 0$ and
$\Hom_H(Z,X) = 0$.
\end{Lem}

\begin{proof}
We have $X \cong \tau^{-k}(P_i)$ for some $1 \le i \le n$ and 
some $k \ge 0$.
We get
$\Hom_H(Y,X) \cong \Hom_H(\tau^k(Y),P_i)$ and
$\Hom_H(Z,X) \cong  \Hom_H(\tau^k(Z),P_i)$.
Now Corollary~\ref{cor:12.3}
yields that these homomorphism spaces are zero.
Similarly, one shows that $\Hom_H(Z,Y) = 0$.
\end{proof}

A sequence $((i_1,p_1),\ldots,(i_t,p_t))$ with
$1 \le i_k \le n$ and $p_i \in \{+,-\}$ is \emph{admissible} for 
$(C,\Omega)$
if the following hold:
\begin{itemize}

\item[(i)]
Either
$i_1$ is a sink in $Q^\circ(C,\Omega)$ and $p_1 = +$, or $i_1$ is 
a source in $Q^\circ(C,\Omega)$ and $p_1 = -$;

\item[(ii)]
For each $2 \le k \le n$, either
$i_k$ is a sink in $Q^\circ(C,s_{i_{k-1}} \cdots s_{i_1}(\Omega))$ and $p_k = +$, or $i_k$ is 
a source in $Q^\circ(C,s_{i_{k-1}} \cdots s_{i_1}(\Omega))$ and $p_k = -$.

\end{itemize}

\begin{Prop}
For an indecomposable locally free $M \in \rep(H)$ the following are
equivalent:
\begin{itemize}

\item[(i)]
$M$ is $\tau$-locally free;

\item[(ii)]
For each admissible sequence 
$((i_1,p_1),\ldots,(i_t,p_t))$ for $(C,\Omega)$
the module
\[
F_{i_t}^{p_t} \cdots F_{i_1}^{p_1}(M)
\]
is locally free.

\end{itemize}
\end{Prop}

\begin{proof}
Assume $M$ is $\tau$-locally free.
Let $i$ be a sink in $Q^\circ(C,\Omega)$.
We want to show that
$F_i^+(M)$ is $\tau$-locally free.
If $M \cong P_i = E_i$, then $F_i^+(M) = 0$, which is trivially $\tau$-locally free.
Thus we can assume that $M \not\cong E_i$.
If $M \cong P_j$ for some $j \not= i$, then $\tp_i(M) = 0$.
In particular, $M_{i,\inn}$ is surjective.
Thus $F_i^+(M)$ is locally free.
Now Lemma~\ref{lem:8.7} yields
that $F_i^+(M)$ is indecomposable.
Next, assume that $M$ is not projective.
In other words, we have
$\tau(M) \not= 0$.
There clearly exists a $+$-admissible sequence $(i_1,\ldots,i_n)$
for $(C,\Omega)$ with
$i_1=i$.
Using that $M$ is $\tau$-locally free and applying 
Theorem~\ref{thm:functoriso} we get
\[
\tau(M) \cong TC^+(M) \cong TF_{i_n}^+ \cdots F_{i_1}^+(M).
\]
By Lemma~\ref{lem:8.7}, the module
$F_i^+(M)$ is indecomposable and locally free.
We can now assume that $\tau^k(F_i^+(M)) \not= 0$ for all
$k \in \Z$. (Otherwise, the indecomposable  module $F_i^+(M)$ is 
preprojective or preinjective and therefore $\tau$-locally free.)

Let $k > 0$.
There exists a $+$-admissible sequence $(j_1,\ldots,j_n)$ for $(C,s_i(\Omega))$
with $j_n=i$.
It follows that $(i,j_1,\ldots,j_{n-1})$ is a $+$-admissible sequence
for $(C,\Omega)$.
We have
\[
\tau(F_i^+(M)) \cong 
T(F_{j_n}^+ \cdots F_{j_1}^+(F_i^+(M)) \cong
F_i^+(\tau(M)),
\]
and this module is indecomposable and locally free since $\tau(M)$ is 
$\tau$-locally free.
Now it follows by induction that
$\tau^k(F_i^+(M)) \cong F_i^+(\tau^k(M))$ is
indecomposable and locally free for each
$k > 0$.

Next, let $k < 0$.
Then there exists a $-$-admissible sequence $(j_1,\ldots,j_n)$ for
$(C,s_i(\Omega))$ with $j_1=i$.
Then 
$(j_2,\ldots,j_n,i)$ is a $-$-admissible sequence for $(C,\Omega)$.
We get
\[
\tau^-(F_i^+(M))
\cong F_{j_n}^- \cdots F_{j_2}^-F_i^-F_i^+(M)
\cong F_{j_n}^- \cdots F_{j_2}^-(M)
\]
and
\[
F_i^+(\tau^-(M)) 
\cong 
F_i^+F_i^-F_{j_n}^- \cdots F_{j_2}^-(M) 
\cong F_{j_n}^- \cdots F_{j_2}^-(M).
\]
(For the last isomorphism we used the dual of Lemma~\ref{lem:8.7}.)
Again by induction we get that $\tau^k(F_i^+(M)) \cong F_i^+(\tau^k(M))$
is indecomposable and locally free for each $k<0$.

Altogether we showed that $F_i^+(M)$ is $\tau$-locally free.
Dually, one shows that $F_j^-(M)$ is $\tau$-locally free for each source $j$
in $Q^\circ(C,\Omega)$.
This implies (ii).

To show the other direction, assume that (ii) holds.
It follows that $T^kC^k(M)$ is locally free for all $k \in \Z$.
Now we can apply Theorem~\ref{thm:functoriso} and get
$\tau^k(M) \cong T^kC^k(M)$.
Thus $M$ is $\tau$-locally free.
\end{proof}

\begin{Prop}\label{prop:10.9}
For an $H$-module $M$ the following are equivalent:
\begin{itemize}

\item[(i)]
$M$ is locally free;

\item[(ii)]
$\tau(M) \cong TC^+(M)$;

\item[(iii)]
$\tau^-(M) \cong TC^-(M)$.

\end{itemize}
\end{Prop}

\begin{proof}
By Theorem~\ref{thm:functoriso}(b) we know that (i) implies (ii).
Now suppose that (ii) holds.
Let $f\df M \to N$ be a monomorphism from $M$ to a locally free
$H$-module $N$.
(For example we could just take the injective envelope of $M$.)
Since $TC^+$ is a left exact functor we get an exact sequence
\[
0 \to TC^+(M) \to TC^+(N).
\]
By Theorem~\ref{thm:functoriso}(b) and assumption (ii) we get an
exact sequence 
\[
0 \to \tau(M) \to \tau(N).
\]
By Proposition~\ref{prop:12.1} we have $\Hom_H(\Du(H),\tau(N)) = 0$,
since $N$ is locally free.
Applying $\Hom_H(\Du(H),-)$ to
the exact sequence above gives $\Hom_H(\Du(H),\tau(M)) = 0$.
Again by Proposition~\ref{prop:12.1} this implies that $M$ is locally free.
The equivalence of (i) and (iii) is proved dually.
\end{proof}

\subsection{Finite type classification for $\tau$-locally free modules}
\label{sec:Dynkinroots}

\begin{Thm}\label{Dynkinroots}
For a Cartan matrix $C$ of Dynkin type 
the following hold:
\begin{itemize}

\item[(i)]
The map 
$M \mapsto \rkv(M)$ 
yields a bijection between the set of isomorphism classes of 
$\tau$-locally free $H$-modules
and the set $\Delta^+(C)$ of positive roots of the Lie
algebra associated with $C$.

\item[(ii)]
For an indecomposable $H$-module $M$ the following are equivalent:
\begin{itemize}

\item[(a)]
$M$ is preprojective;

\item[(b)]
$M$ is preinjective;

\item[(c)]
$M$ is $\tau$-locally free;

\item[(d)]
$M$ is locally free and rigid.

\end{itemize}
\end{itemize}
\end{Thm}

\begin{proof}
Let $\bi = (i_1,\ldots,i_n)$ be a $+$-admissible sequence for $(C,\Omega)$.
We have
$\rkv(P_{i_k}) = \beta_{\bi,k}$ for $1 \le k \le n$, compare Section~\ref{sec3.3}.
Since $C$ is of Dynkin type, we get all elements of $\Delta^+(C)$ by
applying the Coxeter transformation $c^{-s}$ to the $\beta_{\bi,k}$ with $s \ge 0$ and $1 \le k \le n$, compare Lemma~\ref{lem:ct2}.
In particular,
the preprojective $H$-modules and the preinjective $H$-modules coincide.
Now
Proposition~\ref{prop:12.9} 
implies part (i).
We also get that (a) and (b) in part (ii) are equivalent, and that
(a) and (b) implies (c) and (d).
By Proposition~\ref{prop:12.5} we know that (d) implies (c).
Now let
$M \in \rep(H)$ be 
$\tau$-locally free.
Then there exists an injective $H$-module $I_i$ with
$\Hom_H(M,I_i) \not= 0$.
Since $C$ is of Dynkin type, we know that $P_j = \tau^k(I_i)$ for
some $k \ge 0$ and some $1 \le j \le n$.
If $\tau^s(M) = 0$ for some $s \ge 0$, then $M$ is preprojective and we are done.
Thus assume that
$\tau^s(M) \not= 0$ for all $s \ge 0$. 
Then we have 
$
\Hom_H(M,I_i) 
\cong \Hom_H(\tau^k(M),\tau^k(I_i)) = \Hom_H(\tau^k(M),P_j) \not= 0.
$
(Note that all modules appearing here have projective and injective
dimension at most one. 
Thus the stable homomorphism spaces are equal to the ordinary
homomorphism spaces.)
It follows that $\Hom_H(\tau^-(\tau^{k+1}(M)),P_j) \not= 0$,
a contradiction to Corollary~\ref{cor:12.3}.
This finishes the proof.
\end{proof}

Combining the results in Section~\ref{sec11.1}, 
Theorem~\ref{Dynkinroots} and Lemmas~\ref{lem:ct1} and \ref{lem:ct2}
we get the following result.

\begin{Thm}
There are finitely many isomorphism classes of $\tau$-locally free
$H$-modules if and only if $C$ is of Dynkin type.
\end{Thm}

\subsection{The algebra $\Pi$ as a module over $H$}\label{sec:10.3}
Let $\Pi = \Pi(C,D)$.

\begin{Thm}\label{thm:preproj}
${_H}\Pi \cong \bigoplus_{m \ge 0} \tau^{-m}({_H}H)$.
\end{Thm}

\begin{proof}
By Proposition~\ref{prop:12.9} we know that
$\tau^{-m}(H)$ is locally free
for all $m \ge 0$. 
Thus we have $\tau^-(H) \cong \Ext_H^1(\Du(H),H)$ and
\begin{align*}
\tau^-(\tau^{-(m-1)}(H)) &\cong 
\Ext_H^1(\Du(H),\tau^{-(m-1)}(H)) \\
&\cong
\Ext_H^1(\Du(H),H) \otimes_H \tau^{-(m-1)}(H)\\
&\cong \Ext_H^1(\Du(H),H) \otimes_H \Ext_H^1(\Du(H),H)^{\otimes(m-1)}
\end{align*}
where the last isomorphism follows by induction.
By Corollary~\ref{cor:10.7} we know that $\Pi \cong T_H(\Ext_H^1(\Du(H),H))$.
The result follows.
\end{proof}

\begin{Cor} \label{Cor:Pi-fd}  
$\Pi$ is finite-dimensional if and only if
$C$ is of Dynkin type.
\end{Cor}

\begin{proof}
This follows directly from Theorems~\ref{Dynkinroots} and 
\ref{thm:preproj} and
the fact that $\Delta^+(C)$ is finite if $C$ is of Dynkin type.
\end{proof}

\subsection{Regular components of the Auslander-Reiten quiver}
A connected component $\cC$ of the Auslander-Reiten quiver
of $H$ is \emph{$\tau$-locally free regular} if it consists only of $\tau$-locally free regular modules.
(A connected component of the Auslander-Reiten quiver
of a finite-dimensional algebra is called \emph{regular} if it 
consists only of regular modules.)

\begin{Prop}\label{prop:12.6}
For a connected component $\cC$ of the Auslander-Reiten quiver of $H$
the following are equivalent:
\begin{itemize}

\item[(i)]
$\cC$ contains a $\tau$-locally free regular module;

\item[(ii)]
$\cC$ is $\tau$-locally free regular.

\end{itemize}
\end{Prop}

\begin{proof}
Trivially, (ii) implies (i).
For the other direction assume that $M$ is a $\tau$-locally free regular
module in $\cC$.
Let $0 \to \tau(M) \to E \to M \to 0$ be the Auslander-Reiten sequence
ending in $M$.
Applying $\tau^k(-)$ yields again an Auslander-Reiten sequence
\[
0 \to \tau^{k+1}(M) \to \tau^k(E) \to \tau^k(M) \to 0
\]
for each $k \in \Z$.
Here we used that $\tau^{k+1}(M)$ and $\tau^k(M)$ and therefore also $\tau^k(E)$ 
have projective and
injective dimension equal to $1$.
It follows that $\tau^k(N)$ is locally free for each indecomposable
direct summand $N$ of $E$.
Now (ii) follows by induction.
\end{proof}

Let $\cC$ be a connected component of the Auslander-Reiten
quiver of $H$.
Suppose $\cC$ contains an indecomposable projective module $P_i$
with $c_i \ge 2$.
Then $\rad(P_i)$ is obviously not locally free.
Thus $\rad(P_i)$ contains an indecomposable direct summand
$R$, which is not locally free.
Since the inclusion $\rad(P_i) \to P_i$ is a sink map, there
is an arrow $[R] \to [P_i]$ in the Auslander-Reiten quiver of $H$.
Thus $\cC$ contains a module, which is not locally free.

Ringel \cite{Ri2} proved that the regular components of the Auslander-Reiten
quiver of a wild hereditary algebra are always of type $\Z A_\infty$.
An alternative proof is due to Crawley-Boevey \cite[Section~2]{CB1} and 
can easily be adapted to obtain the following theorem.

\begin{Thm}
Assume that $C$ is connected and neither of Dynkin nor of
Euclidean type. 
Let $\cC$ be a $\tau$-locally free regular component of the Auslander-Reiten quiver of $H$.
Then $\cC$ is a component of type $\Z A_\infty$.
\end{Thm}


\section{Projective resolutions and Ext-group symmetries of $\Pi$-modules}\label{sec9}


\subsection{Projective resolutions of $\Pi(C,D)$-modules}
Let $H = H(C,D,\Omega)$ and $\Pi = \Pi(C,D)$.  
Let $\Rep_\vp(\Pi)$ denote the category of 
\emph{all} locally free $\Pi$-modules, possibly of infinite rank.
For typographic reasons we use in this section the convention 
$I=\{1,2,\ldots, n\}$. 
Recall that we have defined the $S$-$S$-bimodule
\[
\ov{B} := \bigoplus_{(i,j)\in\ov{\Omega}} {_iH_j}, 
\]
and that we can identify
$\Pi$ and  $T_S(\ov{B})/J$, 
where 
$J$ is the ideal of $T_S(\ov{B})$ which is generated by the elements
$\rho_i\in \ov{B}\otimes_S\ov{B}$ for $i\in I$. 
For the next result we follow closely the ideas from~\cite[Lemma~3.1]{CBSh}.

\begin{Prop}\label{prop:resol1}
There is an exact sequence of $\Pi$-$\Pi$-bimodules
\begin{equation}\label{eq:resol1}
\underbrace{\bigoplus_{i\in I} \Pi e_i \otimes e_i\Pi \xrightarrow{f}
\bigoplus_{(j,i) \in \ov{\Omega}} \Pi e_j \otimes_j {_j}H_i \otimes_i e_i\Pi \xrightarrow{g}
\bigoplus_{i\in I} \Pi e_i \otimes_i e_i\Pi}_{P_\bullet} 
\xrightarrow{h} \Pi \to 0
\end{equation}
where 
\begin{align*}
f(e_i \otimes e_i) &:= \rho_i \otimes e_i + e_i \otimes \rho_i,\\  
g(e_j \otimes h \otimes e_i) &:= he_i \otimes e_i - e_j \otimes e_jh,\\
h(m \otimes m') &:= mm'. 
\end{align*}
\end{Prop}

\begin{proof}
Observe first, that the above complex can be written more compactly as
\[
\Pi\otimes_S \Pi\xrightarrow{f} \Pi\otimes_S\otimes\ov{B}\otimes_S\Pi
\xrightarrow{g}\Pi\otimes_S\Pi\xrightarrow{h}\Pi\ra 0
\]
Note that we have a surjective $\Pi$-$\Pi$-bimodule homomorphism  
\[
\bigoplus_{i=1}^n \Pi e_i \otimes_i e_i\Pi \xrightarrow{r} J/J^2
\]
defined by $e_i \otimes e_i \mapsto \rho_i$. Moreover, we have a canonical map
\[
J/J^2 \xrightarrow{{\rm can}} \Pi \otimes_S \overline{B} \otimes_S \Pi
\]
given by $\overline{x} \mapsto \tilde{x}_l \otimes 1 + 1 \otimes \tilde{x}_r$
coming from the compositions of
\[
J \xrightarrow{i_{J,l}} \bigoplus_{k \ge 1}(\ov{B}^{\otimes k} \otimes_S \ov{B})
\xrightarrow{{\rm proj} \otimes \id} \Pi \otimes_S \ov{B}
\]
and
\[
J \xrightarrow{i_{J,r}} \bigoplus_{k \ge 1} (\ov{B} \otimes_S \ov{B}^{\otimes k})
\xrightarrow{\id \otimes {\rm proj}} \ov{B} \otimes_S \Pi,
\]
respectively, where the maps $i_{J,l}$ and $i_{J,r}$ are the obvious inclusions. 
Note that both compositions vanish on $J^2$.
It is easy to see that $f = {\rm can} \circ r$.
Thus we only have to show that the sequence
\[
J/J^2 \xrightarrow{\rm can} \Pi \otimes_S \overline{B} \otimes_S \overline{B}
\xrightarrow{u} \Pi \otimes_S \Pi \xrightarrow{{\rm mult}} \Pi \to 0
\]
where $u(1 \otimes b \otimes 1) := b \otimes 1 - 1 \otimes b$ 
is exact.
This is a special case of a combination of results by Schofield
\cite[Theorems~10.1,10.3,10.5]{Sch}.
\end{proof}

\begin{Cor}
For each $M \in \Rep_\vp(\Pi)$ 
the complex $P_\bullet \otimes_\Pi M$ is the beginning of a projective resolution of $M$. 
\end{Cor}

\begin{proof}
The components of $P_\bullet$
are projective as left and as right modules. 
In fact, for example 
$H_i^m\otimes_i e_i\Pi\cong (e_i\Pi)^m$ is a projective right $\Pi$-module.
Now, ${_jH_i}$ is a free right $H_i$-module and  $\Pi e_j$ is a free right 
$H_j$-module and thus $\Pi e_j\otimes_j {_jH_i}$ is also a free right 
$H_i$-module. 
Altogether, $\Pi e_j\otimes_j {_jH_i}\otimes_i e_i \Pi$ 
is a projective right $\Pi$-module. A similar argument shows that
$\Pi e_i\otimes_{H_i} e_i\Pi$ is projective as a right $\Pi$-module.
Thus as a sequence of right modules the sequence $P_\bullet$
splits.
This implies that the sequence $P_\bullet \otimes_\Pi M$
is exact.
Now, if $M \in \Rep_\vp(\Pi)$, then the relevant components of $P_\bullet \otimes_\Pi M$ are
evidently projective. 
\end{proof}

Let us write the complex $P_\bullet \otimes_\Pi M$ explicitly: 
\begin{equation}\label{resolution2}
\bigoplus_{i\in I} \Pi e_i \otimes_i e_i M \xrightarrow{\; f_M \;}
\bigoplus_{(j,i)\in\ov{\Omega}} \Pi e_j \otimes_j {_j}H_i \otimes_i e_iM
\xrightarrow{\; g_M \;}
\bigoplus_{i\in I} \Pi e_i \otimes_i e_i M \xrightarrow{\; h_M \;} M \to 0
\end{equation}
and the maps $f_M, g_M, h_M$ act on generators as follows:
\begin{align*}
f_M(e_j\otimes m_j) &= \sum_{\substack{i\in\Omega(-,j)\\l\in{_iL_j}}} 
\left(
l^*\otimes l\otimes m_j+e_j\otimes l^*\otimes M_{ij}(l\otimes m_j)\right)
\\
&\;\;\;\;\;\;-\sum_{ \substack{k\in\Omega(j,-)\\r\in{_jR_k}}} 
\left(
r\otimes r^*\otimes m_j+e_j\otimes r\otimes M_{kj}(r^*\otimes m_j)\right),\\
g_M(e_i\otimes h\otimes m_j) &= h\otimes m_j - e_i\otimes M_{ij}(h\otimes m_j),\\
h_M(e_i\otimes m_i) &= m_i.
\end{align*}

\begin{Prop} \label{prop:non-Dynk}
Suppose that the Cartan matrix $C$ has no components of Dynkin type. Then in
the complex $P_\bullet$ for $\Pi$ we have $\Ker(f)=0$. 
In particular,
for all $M \in \Rep_\vp(\Pi)$ we have $\pdim(M) \le 2$.
\end{Prop}

\begin{proof}
In a first step we show that $\Ker(E'_i\otimes_\Pi f)=0$ for the generalized 
simple right $\Pi$-modules $E'_i$ concentrated at the vertex $i \in I$.
To this end we
adapt the relevant part of  the proof of~\cite[Proposition~4.2]{BBK} to our setting:
Choose an orientation $\Omega$ such that in $H=H(C,D,\Omega)$ the projective $H$-module
$P_i = He_i$ is the generalized simple $H$-module $E_i$. 
We have a short exact sequence of $H$-modules
\begin{equation}\label{eq:pseudo-AR}
0\ra P_i \ra\oplus_{j\in \Omega(i,-)}P_j\otimes_j{_jH_i} \ra \tau^-(P_i) \ra 0, 
\end{equation}
see the proof of Theorem~\ref{thm:APR}. Note, that this sequence is possibly
not an Auslander-Reiten sequence.
Applying $\Hom_H(-,\Pi)$ to~\eqref{eq:pseudo-AR} we obtain the sequence of
right $\Pi$-modules
\begin{equation} \label{eq:pseudo-AR-Pi}
0\ra \Hom_H(\tau^-(P_i),\Pi)\ra \oplus_{j\in\Omega(i,-)} {_iH_j}\otimes_j\Pi\ra
e_i\Pi \ra \Ext^1_H(\tau^-(P_i),\Pi)\ra 0
\end{equation}
Now, by Theorem~\ref{thm:preproj} 
\[
{_H\Pi}=\oplus_{j\in I}\oplus_{k\in\N} \tau^{-k}(P_j),
\]
with all summands indecomposable, locally free preprojective modules.
In particular, in our situation all summands have injective dimension $1$. Thus,
$\Ext^1_H(\tau^-(P_i),\tau^{-k}(P_j))=\Du\Hom_H(\tau^{-k}(P_j),P_i)=0$ unless
$(j,k)=(i,0)$. We conclude that  we have an isomorphism of
right $\Pi$-modules $\Ext^1_H(\tau^-(P_i),\Pi)\cong E_i'$. Next
\[
\Hom_H(\tau^-(P_i),\Pi)\cong\Hom_H(P_i,\tau({_H}\Pi))=e_i \Pi
\]
where the last equality holds since by our hypothesis on $C$ we have 
$\tau({_H}\Pi) = {_H}\Pi$. Alltogether, we can identify 
now the exact sequence~\eqref{eq:pseudo-AR-Pi} with
\[
0 \ra e_i\Pi\xrightarrow{E'_i\otimes f}\oplus_{j\in \Omega(i,-)} {_iH_j}\otimes_j \Pi
\xrightarrow{E'_i\otimes g} e_i \Pi \ra E'_i \ra 0
\]
Now, let $U:=\Ker(f)$. Then $U$ is projective as a left $\Pi$-module since 
$P_\bullet\xrightarrow{h}\Pi\ra 0$ is a (split) exact sequence of 
projective left modules. Next we observe that 
$E'_i\otimes_\Pi U=0$ since $\Coker(f)$ is projective as a left module.
Now, let $\m\subset\Pi$ be the ideal which is generated by
$\overline{B}$.  Thus $\Pi/\m^j\in\rep_{\vp}(\Pi)$ is filtered by the 
generalized simples $E'_i$ 
for all $j \ge 1$ and $\Pi\subset\varprojlim_j(\Pi/\m^j)$. 
Thus, since $U$ is projective as a left module we get
\[
U=\Pi\otimes_\Pi U\subset 
(\varprojlim_j(\Pi/\m^j))\otimes_\Pi U\subset
\varprojlim_j(\Pi/\m^j\otimes_\Pi U)=0.
\]
This finishes the proof.
\end{proof}

\subsection{Symmetry of extension groups}\label{sec:extsymm}
Let $\Pi = \Pi(C,D)$, and
let $M = (M_i,M_{ij},M_{ji})$ and $N = (N_i,N_{ij},N_{ji})$ be in 
$\Rep_\vp(\Pi)$. 
Let $Q_\bullet(M,N)$ be \emph{the complex}
\begin{equation} \label{eq:rcpxMN}
\bigoplus_{k\in I} \Hom_{H_k}(M_k,N_k) \xleftarrow{\tilde{f}_{M,N}} 
\bigoplus_{(i,j)\in\ov{\Omega}} 
\Hom_{H_i}({_iH_j} \otimes_j M_j,N_i) \xleftarrow{\tilde{g}_{M,N}} 
\bigoplus_{k\in I} \Hom_{H_k}(M_k,N_k)
\end{equation}
where 
$\tilde{f}_{M,N}$ is defined by
\[
\left(\tilde{f}_{M,N}((\psi_{ij})_{(i,j)\in\ov{\Omega}})\right)_k := 
\sum_{j\in\ov{\Omega}(-,k)}
\sgn(j,k) (N_{kj} \circ \ad_{j,k}(\psi_{j,k}) - \psi_{k,j} \circ \ad_{j,k}(M_{jk}))
\]
and
$\tilde{g}_{M,N}$ is defined by
\[
(\tilde{g}_{M,N}((\phi_k)_{k\in I}))_{(i,j)} :=  
N_{ij} \circ (\id_{{_iH_j}} \otimes \phi_j) - \phi_i \circ M_{ij}.
\]
If $N$ is of finite rank, via the trace pairing from Section~\ref{sec:trace1},
we can identify the 
$K$-dual of the shifted complex $Q_\bullet(N,M)[2]$ 
with the following complex $Q_\bullet(N,M)^*$:
\begin{equation} \label{eq:rcpxMN2}
\bigoplus_{k\in I} \Hom_{H_k}(M_k,N_k) \xleftarrow{\tilde{g}^*_{N,M}} 
\bigoplus_{(i,j)\in\ov{\Omega}} 
\Hom_{H_i}(M_i, {_iH_j} \otimes_j N_j) \xleftarrow{\tilde{f}^*_{N,M}} 
\bigoplus_{k\in I} \Hom_{H_k}(M_k,N_k)
\end{equation}
(Recall that we have for example  natural isomorphisms
\[
\Hom_{H_k}(M_k,N_k)\cong \Du\Hom_{H_k}(N_k,M_k)
\]
since $N_k$ is free of finite rank by hypothesis.)
We define moreover
\[
\ad_{M,N}\colon
\bigoplus_{(i,j)\in\ov{\Omega}}\Hom_{H_j}({_jH_i}\otimes_i M_i,N_j) 
\xrightarrow{\oplus_{(i,j)\in\ov{\Omega}}\;\sgn(i,j)\ad_{ji}}
\bigoplus_{(i,j)\in\ov{\Omega}}\Hom_{H_i}(M_i,{_iH_j}\otimes_j N_j).
\]
We know that $\ad_{M,N}$ is an isomorphism.

\begin{Lem} \label{lem:PoM-N}
For $M,N \in \Rep_\vp(\Pi)$ the following hold:
\begin{itemize}

\item[(a)]
The complex $\Hom_\Pi(P_\bullet \otimes_\Pi M,N)$ is isomorphic to 
$Q_\bullet(M,N)$.
\item[(b)] If $N$ is of finite rank
\[
(\id_{\oplus \Hom_{H_k}{(M_k,N_k)}},\ad_{M,N}, \id_{\oplus\Hom_{H_k}(M_k,N_k)})
\]
is an isomorphism between the complexes
$Q_\bullet(M,N)$ and $Q_\bullet(N,M)^*$.
\end{itemize}
\end{Lem}

\begin{proof}
Part (a) is straightforward. 
For (b) we show that $\tilde{f}^*_{N,M}=\ad_{M,N}\circ \tilde{g}_{M,N}$. To this end
we evaluate
for $(i,j)\in\overline{\Omega}$ the $(i,j)$-component of
$\tilde{f}^*_{N,M}$ on $(\la_k)_{k\in I}\in \oplus_{k\in I} \Hom_{H_k}(M_k,N_k)$,
according to our discussion of the trace pairing in Sections~\ref{sec:trace1} and
\ref{sec:trace2}: 
\begin{align*}
(\tilde{f}_{N,M}^*((\la_k)_k))_{ij} &=
\sgn(i,j)(\ad_{i,j}^*(\la_j\circ M_{ji}) -\ad_{j,i}(N_{ji})\circ\la_i)\\
&=\sgn(i,j)(\ad_{j,i}(\la_j\circ M_{ji})-\ad_{j,i}(N_{ji})\circ\la_i)\\
&=\sgn(i,j)(\ad_{j,i}(\la_j\circ M_{ji}-N_{ji}\circ(\id_{H_i}\otimes\la_i)))
\end{align*}
where the second equality follows from Proposition~\ref{prop:adj1}, and 
the third equality is just the definition of $\ad_{ji}$. 
The proof of
$\tilde{g}^*_{N,M} \circ\ad_{M,N}=\tilde{f}_{M,N}$ is similar.
\end{proof}

\begin{Prop}\label{prop:9.10}
For $M,N \in \Rep_\vp(\Pi)$ we have the following
functorial isomorphisms:
\begin{itemize}

\item[(a)]
$\Ker(\tilde{g}_{M,N}) = \Hom_\Pi(M,N)$ \text{ and }
$\Ker(\tilde{f}_{M,N})/\Ima(\tilde{g}_{M,N}) \cong \Ext_\Pi^1(M,N)$.

\item[(b)]
$\Coker{\tilde{f}_{M,N}}=\Ext^2_\Pi(M,N)$ if $C$ has no component of Dynkin type.

\item[(c)] 
$\Hom_\Pi(N,M)\cong \Du\Coker(\tilde{f}_{M,N})$ if $M$ is of finite rank.
\end{itemize}
\end{Prop}

\begin{proof}
Part (a) follows from the functorial isomorphism of~\eqref{eq:rcpxMN} with
$\Hom_\Pi(P_\bullet \otimes_\Pi M,N)$. 
Part~(b) follows by the same token since in this situation  
by Proposition~\ref{prop:non-Dynk} the map $f_M$ in~\eqref{resolution2}
is injective.

For~(c) we just note that by~(a) and Lemma~\ref{lem:PoM-N}~(b) we have  
$\Ker(\tilde{f}^*_{M,N})=\Ker(\tilde{g}_{N,M})=\Hom_\Pi(N,M)$.
\end{proof}

\begin{Thm}\label{thm:9.11}
Let $M\in\Rep_{\vp}(\Pi)$ and $N\in\rep_{\vp}(\Pi)$.
 
\begin{itemize}
\item[(a)]
There is a functorial isomorphism
\[
\Ext^1_\Pi(M,N)\cong \Du\Ext^1_\Pi(N,M).
\]

\item[(b)]
If the Cartan matrix $C$ has no component of Dynkin type, we have
more generally functorial isomorphisms
\[
\Ext^{2-i}_\Pi(M,N)\cong \Du\Ext^i_\Pi(N, M) \text{ for } i=0,1,2.
\]

\item[(c)]
If $M$ is also of finite rank we have
\[
\dim \Ext_\Pi^1(M,N) = \dim \Hom_\Pi(M,N) + \dim \Hom_\Pi(N,M) 
- (M,N)_H.
\]
\end{itemize}
\end{Thm}

\begin{proof}
Recall, that by Proposition~\ref{prop:9.10}~(a) and~(b) we have 
naturally $H^i(Q_\bullet(M,N)) \cong \Ext^i_\Pi(M,N)$ for $i=0,1$ and also for 
$i=2$ in case $C$ has no Dynkin component. In our situation ($N$ of finite
rank), the complex $Q_\bullet (N,M)^*$ is, via the trace pairing, 
identified to the $K$-dual shifted complex $\Du Q_\bullet(N,M)[2]$. Thus, by the
same token we have $H^{2-i}(Q_\bullet(N,M)^*)\cong \Du\Ext^i_\Pi(N,M)$ for $i=0,1$, and also
for $i=2$ in the Dynkin-free case. 
Now, by Lemma~\ref{lem:PoM-N}~(b)
the complexes $Q_\bullet(M,N)$ and $Q_\bullet(N,M)^*$ are naturally isomorphic, 
which implies~(a) and~(b).
For~(c), we observe that by Proposition~\ref{prop:9.10} we obtain
from the complex~(\ref{eq:rcpxMN}) the equality
\begin{multline*}
\dim\Hom_\Pi(N,M)-\dim\Ext^1_\Pi(M,N)+\dim \Hom_\Pi(M,N)\\
= 2\cdot\sum_{k\in I}\dim \Hom_{H_k}(M_k,N_k) - 
\sum_{(i,j)\in\ov{\Omega}}\dim \Hom_{H_j} ({_jH_i}\otimes_i M_i,N_j)=(M,N)_H,
\end{multline*}
which is equivalent to our claim.
\end{proof}

The last statement  in Theorem~\ref{thm:9.11} 
generalizes Crawley-Boevey's formula in 
\cite[Lemma~1]{CB3}.

\begin{Cor} Suppose that the Cartan matrix $C$ is connected of Dynkin type.
Then $\Pi$ is a selfinjective algebra. 
Moreover, in this situation we have
$\Ker(f)=\Hom_\Pi(\Du(\Pi),\Pi)$ as a $\Pi$-$\Pi$-bimodule, where $f$ is
the last morphism in the complex $P_\bullet$ of Proposition~\ref{prop:resol1}.
\end{Cor}

\begin{proof} By Corollary~\ref{Cor:Pi-fd}, $\Pi$ is in this situation a 
finite-dimensional basic $K$-algebra. So, for the first claim we have to
show only that $\Pi$, as a left module, is injective. In any case, we can
find a short exact sequence of left $\Pi$-modules
\[
0\ra \Pi \xrightarrow{\iota} Q \xrightarrow{\pi} R \ra 0
\]
with $Q$ injective. Since $\Pi$ and $Q$ are locally free, $R$ is also locally
free. Now, $\Hom_\Pi(\iota,\Pi)$ is surjective since 
$\Ext^1_\Pi(R,\Pi)=\Du\Ext^1_\Pi(\Pi,R)=0$ by~Theorem~\ref{thm:9.11}.
Thus, there exists 
$\rho\in\Hom_\Pi(Q,\Pi)$ with $\rho\iota=\id_\Pi$. In other words, $\Pi$
is a direct summand of the injective module $Q$.

For the second claim we note that we have here natural identifications
$f= \Du(\tilde{f}_{\Pi,\Du(\Pi)}) = \tilde{g}_{\Du(\Pi),\Pi}$. Now, 
$\Ker(\tilde{g}_{\Du(\Pi),\Pi})=\Hom_\Pi(\Du(\Pi),\Pi)$ by 
Proposition~\ref{prop:9.10}~(a).
\end{proof}

\begin{Rem} For (classical) preprojective algebras $\Pi$ associated to a 
Dynkin quiver it seems to be folklore that $\Ker(f)\cong \Du(\Pi)$. This is
not in contradiction with the above statement. In fact, the
Nakayama automorphism, viewed as an element of the group of outer 
automorphism has order $2$ (except for a few cases over fields of characteristic
$2$ when it is the identity), see~\cite[Theorem~4.8]{BBK}.
and thus $\Du(\Pi) \cong \Hom_\Pi(\Du(\Pi),\Pi)$ as a bimodule. We expect that
for our generalized preprojective algebras a similar statement holds.
\end{Rem}


\section{Examples}\label{sec13}


\subsection{}\label{example-def-H-Pi}
The matrix
\[
C = \left(\bbm 2&-4&0\\-6&2&-3\\0&-9&2\ebm\right)
\]
is a Cartan matrix, and
$D = \diag(9,6,2)$
is the minimal symmetrizer of $C$.
Let $\Omega = \{(1,2),(2,3)\}$.
This is an orientation of $C$.
We have
$f_{12} = 2$,
$f_{21} = 3$,  $f_{23} = 1$, $f_{32} = 3$, 
$g_{12} = 2$ and $g_{23} = 3$.
The algebra $H = H(C,D,\Omega)$ is given by the quiver 
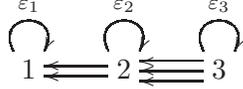
\[
\xymatrix{
1 \ar@(ul,ur)[]^{\vep_1}& 2 \ar@(ul,ur)[]^{\vep_2}\ar@<0.4ex>[l]\ar@<-0.4ex>[l]
& 3 \ar@(ul,ur)[]^{\vep_3}\ar[l]\ar@<0.8ex>[l]\ar@<-0.8ex>[l]
}
\]
with relations
\begin{align*}
\vep_1^9 &= 0,  & \vep_2^6 &= 0, & \vep_3^2 &= 0,
\end{align*}
\begin{align*}
\vep_1^3\alpha_{12}^{(g)} &= \alpha_{12}^{(g)}\vep_2^2, &
(g&=1,2),
\\
\vep_2^3\alpha_{23}^{(g)} &= \alpha_{23}^{(g)}\vep_3, &
(g &=1,2,3).
\end{align*}
(Recall that $\alpha_{ij}^{(g)}$ denotes an arrow $j \to i$.)

The preprojective algebra $\Pi = \Pi(C,D)$ is given by the double quiver
$\ov{Q}(C)$ with relations
\begin{align*}
\vep_1^9 &= 0, & \vep_2^6 &= 0, & \vep_3^2 &= 0,
\end{align*}
\begin{align*}
\vep_1^3\alpha_{12}^{(g)} &= \alpha_{12}^{(g)}\vep_2^2, &
\vep_2^2\alpha_{21}^{(g)} &= \alpha_{21}^{(g)}\vep_1^3, &
(g&=1,2),
\\
\vep_2^3\alpha_{23}^{(g)} &= \alpha_{23}^{(g)}\vep_3, &
\vep_3\alpha_{32}^{(g)} &= \alpha_{32}^{(g)}\vep_2^3, & 
(g &=1,2,3),
\end{align*}
\begin{align*}
&\sum_{g=1}^2 \left(\alpha_{12}^{(g)}\alpha_{21}^{(g)}\vep_1^2 +
\vep_1\alpha_{12}^{(g)}\alpha_{21}^{(g)}\vep_1 +
\vep_1^2\alpha_{12}^{(g)}\alpha_{21}^{(g)}\right) = 0,
\\
&\sum_{g=1}^2\left( -\alpha_{21}^{(g)}\alpha_{12}^{(g)}\vep_2 
-\vep_2\alpha_{21}^{(g)}\alpha_{12}^{(g)}   \right)
+
\sum_{g=1}^3\left( \alpha_{23}^{(g)}\alpha_{32}^{(g)}\vep_2^2+
\vep_2\alpha_{23}^{(g)}\alpha_{32}^{(g)}\vep_2+
\vep_2^2\alpha_{23}^{(g)}\alpha_{32}^{(g)}
\right) = 0,
\\
&\sum_{g=1}^3 -\alpha_{32}^{(g)}\alpha_{23}^{(g)} = 0.
\end{align*}

\subsection{Cartan matrices of Dynkin type}\label{secDynkin}
Figure~\ref{fig:Dynkin} shows a list of
valued graphs called \emph{Dynkin graphs}.
By definition
each of the graphs $A_n$, $B_n$, $C_n$ and $D_n$ has $n$ vertices.
The graphs $A_n$, $D_n$, $E_6$, $E_7$ and $E_8$ are the
\emph{simply laced Dynkin graphs}.
A Cartan matrix $C$ is of \emph{Dynkin type} 
if the valued graph $\GG(C)$ is isomorphic (as a valued graph) to a
disjoint union of  Dynkin graphs.
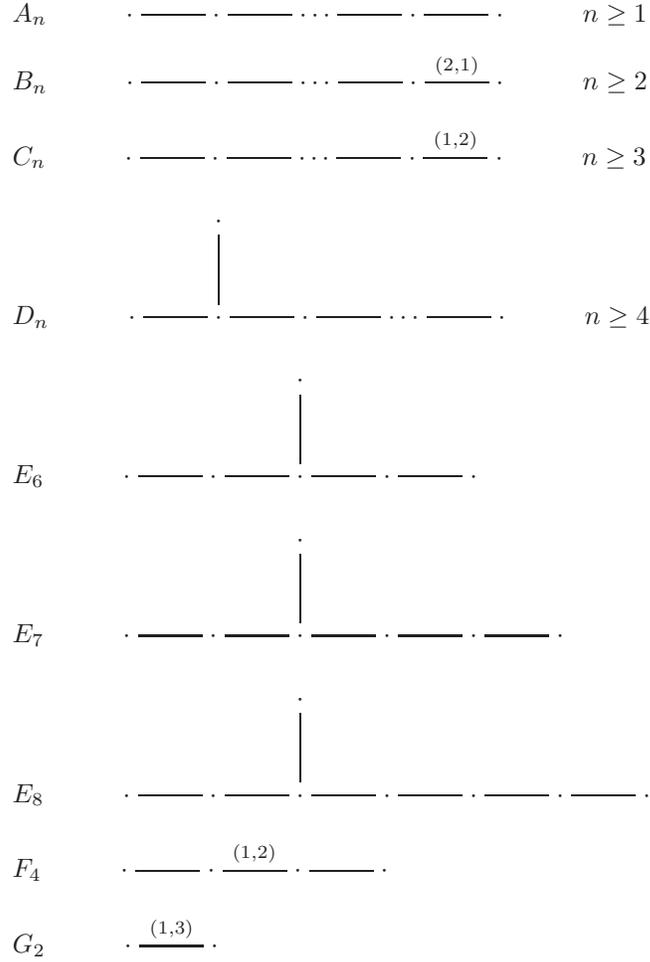
\begin{figure}[!htb]
\centering
{\small
\begin{tabular}{l}
\vspace{0.2cm}
\xymatrix{
A_n &
\cdot \ar@{-}[r] & \cdot \ar@{-}[r] & \cdots \ar@{-}[r] & \cdot 
\ar@{-}[r] & \cdot & n \ge 1 
}\\
\vspace{0.3cm}
\xymatrix{
B_n &
\cdot \ar@{-}[r] & \cdot \ar@{-}[r] & \cdots \ar@{-}[r] & \cdot 
\ar@{-}[r]^{(2,1)} & \cdot & n \ge 2
} \\
\vspace{0.3cm}
\xymatrix{
C_n&
\cdot \ar@{-}[r] & \cdot \ar@{-}[r] & \cdots \ar@{-}[r] & \cdot 
\ar@{-}[r]^{(1,2)} & \cdot & n \ge 3
}\\
\vspace{0.3cm}
\xymatrix{
&& \cdot \ar@{-}[d]\\
D_n& \cdot \ar@{-}[r] & 
\cdot \ar@{-}[r] & \cdot \ar@{-}[r] & \cdots \ar@{-}[r] 
& \cdot & n \ge 4
}\\
\vspace{0.3cm}
\xymatrix{
&&& \cdot \ar@{-}[d]\\
E_6&\cdot \ar@{-}[r] & \cdot \ar@{-}[r] & \cdot \ar@{-}[r] & \cdot 
\ar@{-}[r] & \cdot 
}\\
\vspace{0.3cm}
\xymatrix{
&&& \cdot \ar@{-}[d]\\
E_7 &\cdot \ar@{-}[r] & \cdot \ar@{-}[r] & \cdot \ar@{-}[r] & \cdot 
\ar@{-}[r] & \cdot \ar@{-}[r] & \cdot
}\\
\vspace{0.3cm}
\xymatrix{
&&& \cdot \ar@{-}[d]\\
E_8&\cdot \ar@{-}[r] & \cdot \ar@{-}[r] & \cdot \ar@{-}[r] & \cdot 
\ar@{-}[r] & \cdot \ar@{-}[r] & \cdot \ar@{-}[r] & \cdot
}\\
\vspace{0.3cm}
\xymatrix{
F_4 &
\cdot \ar@{-}[r] & \cdot \ar@{-}[r]^{(1,2)} & \cdot \ar@{-}[r] & \cdot 
}\\
\vspace{0.3cm}
\xymatrix{
G_2 & 
\cdot \ar@{-}[r]^{(1,3)} & \cdot
}
\end{tabular}
}
\caption{Dynkin graphs.}
\label{fig:Dynkin}
\end{figure}

\subsection{Finite representation type}
Let $H = H(C,D,\Omega)$ with
$D = \diag(c_1,\ldots,c_n)$.
Without loss of generality assume that $C$ is connected.
We only sketch the proof of the following proposition.

\begin{Prop}
The algebra $H$ is representation-finite if and only if
we are in one of the following cases:
\begin{itemize}

\item[(i)]
$C$ is of Dynkin type $A_n$, $C_n$, $D_n$, $E_6$, $E_7$, $E_8$, $B_2$, $B_3$ or $G_2$, and $D$ is minimal;

\item[(ii)]
$C$ is of Dynkin type $A_1$;

\item[(iii)]
$C$ is of Dynkin type $A_2$, and we have
$(c_1,c_2) = (2,2)$ or $(c_1,c_2) = (3,3)$;

\item[(iv)]
$C$ is of Dynkin type $A_3$, and we have
$(c_1,c_2,c_3) = (2,2,2)$.

\end{itemize}
\end{Prop}

\begin{proof}
Assume that $D$ is minimal.
For Dynkin types $A_n$, $D_n$, $E_6$, $E_7$, $E_8$, the algebra
$H$ is representation-finite by Gabriel's Theorem.
For type $C_n$, the algebra $H$ is a representation-finite string algebra.
The Auslander-Reiten quiver of $H$ for types
$B_2$, $B_3$ and $G_2$ can be computed by covering theory and the
knitting algorithm for preprojective components. 
They all turn out to be finite.

If $C$ is of type $A_1$, then the symmetrizers are
$D = (m)$ with $m \ge 1$.
Then
$H \cong K[\vep_1]/(\vep_1^m)$
is just a truncated polynomial ring, which is obviously
representation-finite.

If $C$ is of type $A_2$ and $(c_1,c_2) = (2,2)$ or $(c_1,c_2) = (3,3)$, then
$H$ is a representation-finite algebra, see Bongartz and Gabriel's
list \emph{Maximal algebras with $2$ simples modules} in
\cite[Section~7]{BG}.

If $C$ is of type $A_3$ with $(c_1,c_2,c_3) = (2,2,2)$, then one can again use
covering theory and the knitting algorithm to check that $H$ is
representation-finite.

It is straightforward to check that these are all representation-finite 
cases.
(One first compiles the list of all minimal algebras $H$, which are not 
mentioned in (i), (ii), (iii) and (iv).
These are the algebras $H = H(C,D,\Omega)$ of types
\begin{itemize}

\item 
$A_2$ with $D = \diag(4,4)$;

\item
$A_3$ with $D = \diag(3,3,3)$;

\item
$A_4$ with $D = \diag(2,2,2,2)$;

\item 
$B_2$ with $D = \diag(4,2)$;

\item 
$B_4$ with $D$ minimal;

\item 
$D_4$ with $D = \diag(2,2,2,2)$;

\item 
$F_4$ with $D$ minimal.

\end{itemize}
Then one uses covering theory and the Happel-Vossieck list (see 
\cite{HV})
to check that these minimal algebras are representation infinite.)
\end{proof}

\subsection{Notation}
In the following subsections we discuss several examples.
We also display the Auslander-Reiten quivers of
some representation-finite algebras $H$.
The $\tau$-locally free $H$-modules are marked with a double frame,
the locally free $H$-modules, which are not $\tau$-locally free, are
marked with a single solid frame, and the Gorenstein-projective $H$-modules,
which are not projective, have a dashed frame.

\subsection{Dynkin type $A_2$}\label{example:typeA2}
Let 
\[
C = \left(\bbm 2&-1\\-1&2 \ebm\right)
\]
with symmetrizer $D = \diag(2,2)$ and $\Omega = \{ (1,2)\}$.
Thus $C$ is a Cartan matrix of Dynkin type $A_2$ with
a non-minimal symmetrizer.
We have $f_{12} = f_{21} = 1$.
Thus $H = H(C,D,\Omega)$ is given by the quiver
\[
\xymatrix{
1 \ar@(ul,ur)^{\vep_1}& 2 \ar[l]^{\alpha_{12}}\ar@(ul,ur)^{\vep_2}
}
\]
with relations $\vep_1^2 = \vep_2^2 = 0$ and 
$\vep_1\alpha_{12} = \alpha_{12}\vep_2$.
The Auslander-Reiten quiver of 
$H$ is displayed in Figure~\ref{Fig:A2(2,2)}.
The numbers in the figure correspond to composition factors and
basis vectors.
(The three modules in the left most column have to be identified with the
three modules in the right most column.)
Note that $P_2 \cong I_1$ is projective-injective.
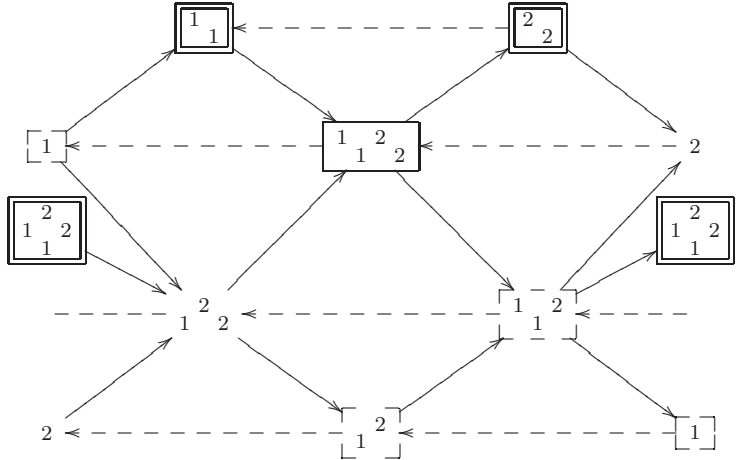
\begin{figure}[!htb]
\[
\xymatrix@-1.0pc@R=10pt{
&&*+[F=]{\bsm 1&\\&1 \esm}\ar[ddrr]&&&&
*+[F=]{\bsm 2&\\&2 \esm}\ar@{-->}[llll]\ar[ddrr]
\\
&&&&&&&&
\\
*+[F--]{\bsm 1 \esm}\ar[uurr]\ar[ddrr]&&&&
*+[F]{\bsm 1&&2\\&1&&2 \esm}\ar[uurr]\ar[ddrr]\ar@{-->}[llll]&&&&
{\bsm 2 \esm}\ar@{-->}[llll]
\\
*+[F=]{\bsm &2\\1&&2\\&1 \esm} \ar[drr]
&&&&&&&&
*+[F=]{\bsm &2\\1&&2\\&1 \esm}
\\
&&{\bsm &2\\1&&2 \esm}\ar[ddrr]\ar[uurr]\ar@{--}[ll]&&&&
*+[F--]{\bsm 1&&2\\&1 \esm}\ar@{-->}[llll]\ar[ddrr]\ar[uurr]\ar[rru]
&&\ar@{-->}[ll]
\\
&&&&&&&&
\\
{\bsm 2 \esm}\ar[uurr]&&&&
*+[F--]{\bsm &2\\1 \esm}\ar[uurr]\ar@{-->}[llll]&&&&
*+[F--]{\bsm 1 \esm}\ar@{-->}[llll]
}
\] 
\caption{
The Auslander-Reiten quiver of $H(C,D,\Omega)$ of type $A_2$ with
$D = \diag(2,2)$.}
\label{Fig:A2(2,2)}
\end{figure}

The preprojective algebra $\Pi = \Pi(C,D)$ is given by the quiver
\[
\xymatrix{
1 \ar@(ul,ur)^{\vep_1} \ar@<0.4ex>[r]^{\alpha_{21}}& 2 \ar@<0.4ex>[l]^{\alpha_{12}} \ar@(ul,ur)^{\vep_2}
}
\]
with relations $\vep_1^2 = \vep_2^2 = 0$, 
$\vep_1\alpha_{12} = \alpha_{12}\vep_2$,
$\vep_2\alpha_{21} = \alpha_{21}\vep_1$,
$\alpha_{12}\alpha_{21} = 0$ and
$-\alpha_{21}\alpha_{12} = 0$.
The indecomposable projective $\Pi$-modules are shown in
Figure~\ref{Fig:projPA2}. 
(The arrows indicate when an arrow of the
algebra $\Pi$ acts with a non-zero scalar on a basis vector.)
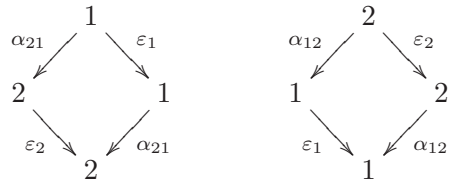
\begin{figure}[!htb]
\[
\xymatrix@-0.7pc{ 
& 1 \ar[dl]_{\alpha_{21}}\ar[dr]^{\vep_1}
&&&& 2 \ar[dl]_{\alpha_{12}}\ar[dr]^{\vep_2}
\\
2 \ar[dr]_{\vep_2} && 1 \ar[dl]^{\alpha_{21}}
&&1 \ar[dr]_{\vep_1} && 2 \ar[dl]^{\alpha_{12}}
\\
& 2 &&&& 1
}
\]
\caption{
The indecomposable projective $\Pi(C,D)$-modules for type $A_2$
with $D = (2,2)$.}
\label{Fig:projPA2}
\end{figure}

\subsection{Dynkin type $B_2$}
Let 
\[
C = \left(\bbm 2&-1\\-2&2 \ebm\right)
\]
with symmetrizer $D = \diag(2,1)$ and
$\Omega = \{ (1,2) \}$.
The graph $\GG(C)$ looks as follows:
\[
\xymatrix{
1 \ar@{-}[r]^{(2,1)} & 2
}
\]
Thus $C$ is a Cartan matrix of Dynkin type $B_2$.
We have $f_{12} = 1$ and $f_{21} = 2$.
Then $H = H(C,D,\Omega)$ is given by the quiver
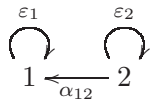
\[
\xymatrix{
1 \ar@(ul,ur)^{\vep_1} & 2 \ar[l]^{\alpha_{12}}\ar@(ul,ur)^{\vep_2}
}
\]
with relations $\vep_1^2 = 0$ and $\vep_2 = 0$. 
The Auslander-Reiten quiver of 
$H$ is shown in Figure~\ref{Fig:B2}.
The numbers in the figure correspond to composition factors and basis vectors.
(In the last two rows the two modules on the left have to be identified
with the corresponding two modules on the right.)
\begin{figure}[!htb]
\[
\xymatrix@R=15pt{
&&*+[F=]{\bsm 2\\1\\1 \esm}\ar[dr]&&
*+[F=]{\bsm 2  \esm}\ar@{-->}[ll]
\\
&*+[F=]{\bsm 1\\1\esm}\ar[dr]\ar[ur]&&
*+[F=]{\bsm 2\\1&&2\\&1 \esm}\ar@{-->}[ll]\ar[dr]\ar[ur]
\\
*+[F--]{\bsm 1 \esm}\ar[ur]\ar[dr]&&
*+[F]{\bsm 1&&2\\&1 \esm}\ar[dr]\ar[ur]\ar@{-->}[ll]&&
{\bsm 2\\1 \esm}\ar@{-->}[ll]
\\
&{\bsm 2\\1 \esm}\ar[ur]&&
*+[F--]{\bsm 1 \esm}\ar@{-->}[ll]\ar[ur]
}
\] 
\caption{
The Auslander-Reiten quiver of $H(C,D,\Omega)$ of type $B_2$ with
$D$ minimal.}
\label{Fig:B2}
\end{figure}
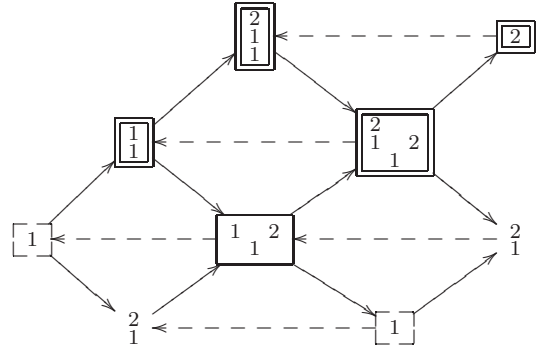

The preprojective algebra $\Pi = \Pi(C,D)$ is given by the quiver
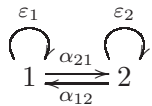
\[
\xymatrix{
1 \ar@(ul,ur)^{\vep_1} \ar@<0.4ex>[r]^{\alpha_{21}}& 2\ar@(ul,ur)^{\vep_2} \ar@<0.4ex>[l]^{\alpha_{12}}
}
\]
with relations $\vep_1^2 = 0$, $\vep_2 = 0$, 
$\alpha_{12}\alpha_{21}\vep_1 + \vep_1\alpha_{12}\alpha_{21} = 0$
and
$-\alpha_{21}\alpha_{12} = 0$.
The indecomposable projective $\Pi$-modules are shown in
Figure~\ref{Fig:projPB2}. 
(The arrows indicate when an arrow of the
algebra $\Pi$ acts with a non-zero scalar on a basis vector.)
\begin{figure}[!htb]
\[
\xymatrix@-0.7pc{ 
&1 \ar[dl]_{\alpha_{21}}\ar[dr]^{\vep_1}
&&&&
2 \ar[d]^{\alpha_{12}}
&&&&
\\
2 \ar[d]_{\alpha_{12}} && 1 \ar[d]^{\alpha_{21}} 
&&&
1 \ar[d]^{\vep_1} 
&&&&
\\
1 \ar[dr]_{\vep_1} && 2 \ar[dl]^{\alpha_{12}} 
&&&
1 \ar[d]^{\alpha_{21}}
&&&&
\\
&1 
&&&& 
2
&&&&
}
\]
\caption{
The indecomposable projective $\Pi(C,D)$-modules for type $B_2$
with $D$ minimal.}
\label{Fig:projPB2}
\end{figure}

\subsection{Dynkin type $B_3$}\label{secB3ex}
Let 
\[
C = \left(\bbm 2&-1&0\\-1&2&-1\\0&-2&2 \ebm\right)
\]
with symmetrizer $D = \diag(2,2,1)$ and $\Omega = \{(1,2),(2,3)\}$.
The graph $\GG(C)$ looks as follows:
\[
\xymatrix{
1 \ar@{-}[r] & 2 \ar@{-}[r]^{(2,1)} & 3
}
\]
Thus $C$ is a Cartan matrix of Dynkin type $B_3$.
We have $f_{12} = f_{21} = 1$, $f_{23} = 1$ and $f_{32} = 2$.
Thus $H = H(C,D,\Omega)$ is given by the quiver
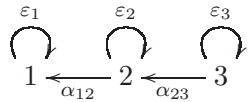
\[
\xymatrix{
1 \ar@(ul,ur)^{\vep_1} & 2 \ar[l]^{\alpha_{12}}\ar@(ul,ur)^{\vep_2} & 3
\ar@(ul,ur)^{\vep_3} \ar[l]^{\alpha_{23}}
}
\]
with relations $\vep_1^2 = \vep_2^2 = 0$, $\vep_3 = 0$ and
$\vep_1\alpha_{12} = \alpha_{12}\vep_2$.
The Auslander-Reiten quiver of 
$H$ is shown in Figure~\ref{Fig:B3}.
As vertices we have the graded dimension vectors (arising from the obvious $\Z$-covering of $H$) of the indecomposable 
$H$-modules. 
(In the last three rows the three modules on the left have to be identified
with the corresponding three modules on the right.)
The indecomposable $H$-module $M$ with graded dimension vector
\[
\bsm 1&1&0\\1&2&1\\0&1&0\esm
\]
is locally free. 
(It is a direct summand of an extension of locally free
modules.)
We have $\rkv(M) = (1,2,1)$.
In the root lattice of $C$ this corresponds to $\alpha_1+2\alpha_2+\alpha_3$.
Thus we have $\rkv(M) \notin \Delta^+(C)$.
\begin{landscape}
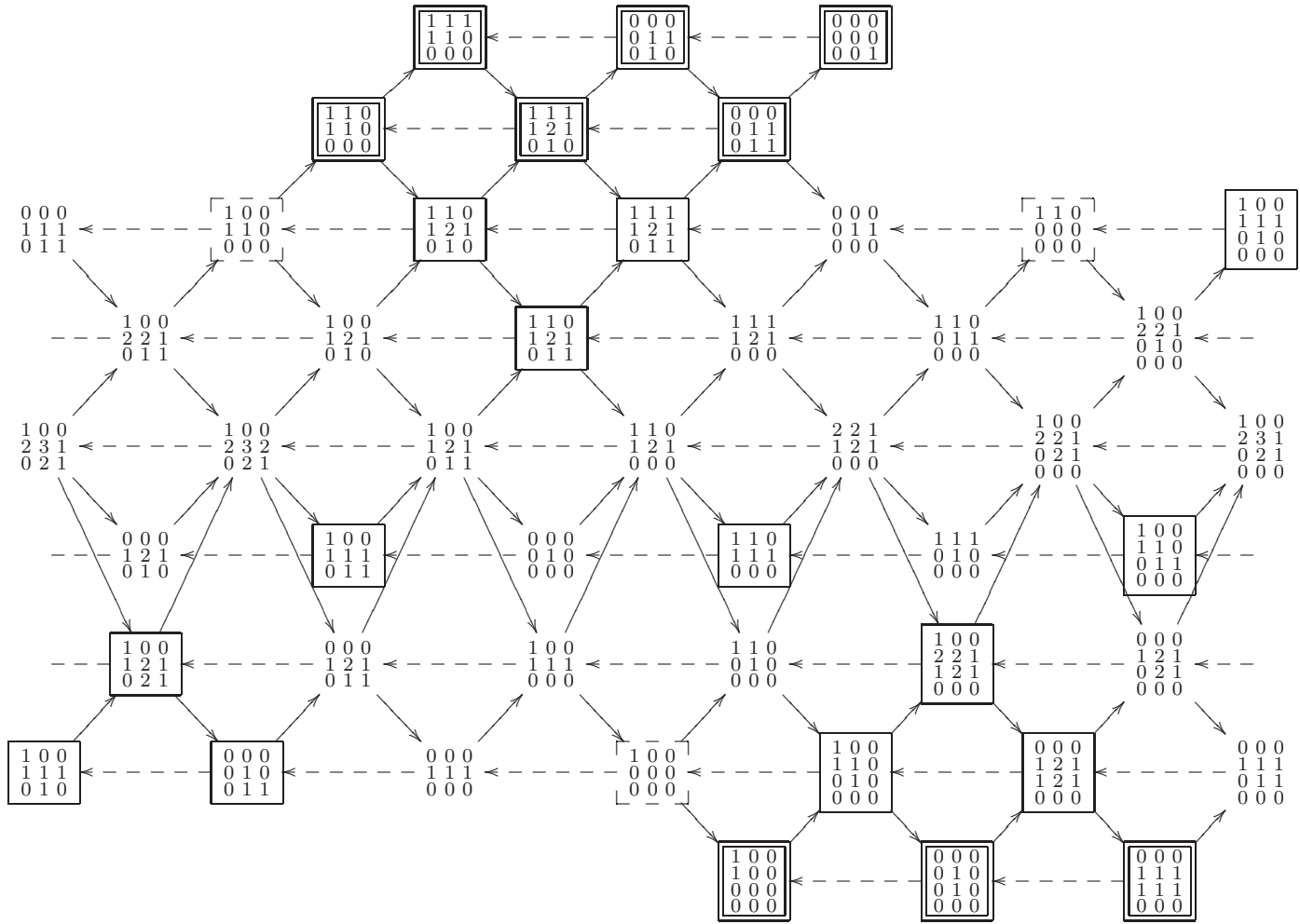
\begin{figure}[!htb]
\[
\xymatrix@-2.5ex{
&&&&*+[F=]{\bsm 1&1&1\\1&1&0\\0&0&0 \esm}\ar[dr]
&&*+[F=]{\bsm 0&0&0\\0&1&1\\0&1&0 \esm}\ar[dr]\ar@{-->}[ll]
&&*+[F=]{\bsm 0&0&0\\0&0&0\\0&0&1 \esm}\ar@{-->}[ll]
&&&&
\\
&&&*+[F=]{\bsm 1&1&0\\1&1&0\\0&0&0 \esm}\ar[ur]\ar[dr]
&&*+[F=]{\bsm 1&1&1\\1&2&1\\0&1&0 \esm}\ar[ur]\ar[dr]\ar@{-->}[ll]
&&*+[F=]{\bsm 0&0&0\\0&1&1\\0&1&1 \esm}\ar@{-->}[ll]\ar[ur]\ar[dr]
\\
{\bsm 0&0&0\\1&1&1\\0&1&1 \esm}\ar[dr]
&&*+[F--]{\bsm 1&0&0\\1&1&0\\0&0&0 \esm}\ar@{-->}[ll]\ar[dr]\ar[ur]
&&*+[F]{\bsm 1&1&0\\1&2&1\\0&1&0 \esm}\ar@{-->}[ll]\ar[dr]\ar[ur]
&&*+[F]{\bsm 1&1&1\\1&2&1\\0&1&1 \esm}\ar@{-->}[ll]\ar[dr]\ar[ur]
&&{\bsm 0&0&0\\0&1&1\\0&0&0 \esm}\ar@{-->}[ll]\ar[dr]
&&*+[F--]{\bsm 1&1&0\\0&0&0\\0&0&0 \esm}\ar@{-->}[ll]\ar[dr]
&&*+[F]{\bsm 1&0&0\\1&1&1\\0&1&0\\0&0&0 \esm}\ar@{-->}[ll]
\\
&{\bsm 1&0&0\\2&2&1\\0&1&1 \esm}\ar[ur]\ar[dr]\ar@{--}[l]
&&{\bsm 1&0&0\\1&2&1\\0&1&0 \esm}\ar[dr]\ar[ur]\ar@{-->}[ll]
&&*+[F]{\bsm 1&1&0\\1&2&1\\0&1&1 \esm}\ar[ur]\ar[dr]\ar@{-->}[ll]
&&{\bsm 1&1&1\\1&2&1\\0&0&0 \esm}\ar[ur]\ar[dr]\ar@{-->}[ll]
&&{\bsm 1&1&0\\0&1&1\\0&0&0 \esm}\ar[ur]\ar[dr]\ar@{-->}[ll]
&&{\bsm 1&0&0\\2&2&1\\0&1&0\\0&0&0 \esm}\ar[ur]\ar[dr]
\ar@{-->}[ll]&\ar@{-->}[l]
\\
{\bsm 1&0&0\\2&3&1\\0&2&1 \esm}\ar[dr]\ar[ur]\ar[ddr]
&&{\bsm 1&0&0\\2&3&2\\0&2&1 \esm}\ar[dr]\ar[ur]\ar[ddr]\ar@{-->}[ll]
&&{\bsm 1&0&0\\1&2&1\\0&1&1 \esm}\ar[dr]\ar[ur]\ar[ddr]\ar@{-->}[ll]
&&{\bsm 1&1&0\\1&2&1\\0&0&0 \esm}\ar[dr]\ar[ur]\ar[ddr]\ar@{-->}[ll]
&&{\bsm 2&2&1\\1&2&1\\0&0&0 \esm}\ar[dr]\ar[ur]\ar[ddr]\ar@{-->}[ll]
&&{\bsm 1&0&0\\2&2&1\\0&2&1\\0&0&0 \esm} \ar[dr]\ar[ur]\ar[ddr]\ar@{-->}[ll]
&&{\bsm 1&0&0\\2&3&1\\0&2&1\\0&0&0 \esm}\ar@{-->}[ll]
\\
&{\bsm 0&0&0\\1&2&1\\0&1&0 \esm}\ar[ur]\ar@{--}[l]
&&*+[F]{\bsm 1&0&0\\1&1&1\\0&1&1 \esm}\ar[ur]\ar@{-->}[ll]
&&{\bsm 0&0&0\\0&1&0\\0&0&0 \esm}\ar[ur]\ar@{-->}[ll]
&&*+[F]{\bsm 1&1&0\\1&1&1\\0&0&0 \esm}\ar[ur]\ar@{-->}[ll]
&&{\bsm 1&1&1\\0&1&0\\0&0&0 \esm}\ar[ur]\ar@{-->}[ll]
&&*+[F]{\bsm 1&0&0\\1&1&0\\0&1&1\\0&0&0 \esm}\ar[ur]\ar@{-->}[ll]
&\ar@{-->}[l]
\\
&*+[F]{\bsm 1&0&0\\1&2&1\\0&2&1 \esm}\ar[uur]\ar[dr]\ar@{--}[l]
&&{\bsm 0&0&0\\1&2&1\\0&1&1 \esm}\ar[uur]\ar[dr]\ar@{-->}[ll]
&&{\bsm 1&0&0\\1&1&1\\0&0&0 \esm}\ar[uur]\ar[dr]\ar@{-->}[ll]
&&{\bsm 1&1&0\\0&1&0\\0&0&0 \esm}\ar[uur]\ar[dr]\ar@{-->}[ll]
&&*+[F]{\bsm 1&0&0\\2&2&1\\1&2&1\\0&0&0 \esm}\ar[uur]\ar[dr]
\ar@{-->}[ll]
&&{\bsm 0&0&0\\1&2&1\\0&2&1\\0&0&0 \esm}\ar[uur]\ar[dr]
\ar@{-->}[ll]&\ar@{-->}[l]
\\
*+[F]{\bsm 1&0&0\\1&1&1\\0&1&0 \esm}\ar[ur]
&&*+[F]{\bsm 0&0&0\\0&1&0\\0&1&1 \esm}\ar[ur]\ar@{-->}[ll]
&&{\bsm 0&0&0\\1&1&1\\0&0&0 \esm}\ar[ur]\ar@{-->}[ll]
&&*+[F--]{\bsm 1&0&0\\0&0&0\\0&0&0 \esm}\ar[ur]\ar[dr]\ar@{-->}[ll]
&&*+[F]{\bsm 1&0&0\\1&1&0\\0&1&0\\0&0&0 \esm}\ar[ur]\ar[dr]
\ar@{-->}[ll]
&&*+[F]{\bsm 0&0&0\\1&2&1\\1&2&1\\0&0&0 \esm}\ar[ur]\ar[dr]
\ar@{-->}[ll]
&&{\bsm 0&0&0\\1&1&1\\0&1&1\\0&0&0 \esm}\ar@{-->}[ll]
\\
&&&&&&&
*+[F=]{\bsm 1&0&0\\1&0&0\\0&0&0\\0&0&0 \esm}\ar[ur]
&&*+[F=]{\bsm 0&0&0\\0&1&0\\0&1&0\\0&0&0 \esm}\ar[ur]\ar@{-->}[ll]
&&*+[F=]{\bsm 0&0&0\\1&1&1\\1&1&1\\0&0&0 \esm}\ar[ur]\ar@{-->}[ll]
&
}
\] 
\caption{
The Auslander-Reiten quiver of $H(C,D,\Omega)$ of type $B_3$ with
$D$ minimal.}
\label{Fig:B3}
\end{figure}
\end{landscape}

\subsection{Dynkin type $C_3$}
Let 
\[
C = \left(\bbm 2&-1&0\\-1&2&-2\\0&-1&2 \ebm\right)
\]
with symmetrizer $D = \diag(1,1,2)$ and
$\Omega = \{ (1,2),(2,3) \}$.
The graph $\GG(C)$ looks as follows:
\[
\xymatrix{
1 \ar@{-}[r] & 2 \ar@{-}[r]^{(1,2)} & 3
}
\]
Thus $C$ is a Cartan matrix of Dynkin type $C_3$.
We have $f_{12} = f_{21} = 1$, $f_{23} = 2$ and $f_{32} = 1$.
Then $H = H(C,D,\Omega)$ is given by the quiver
\[
\xymatrix{
1 \ar@(ul,ur)^{\vep_1}& 2 \ar@(ul,ur)^{\vep_2}\ar[l]^{\alpha_{12}} & 3 \ar[l]^{\alpha_{23}}
\ar@(ul,ur)^{\vep_3}
}
\]
with relations $\vep_1 = \vep_2 = 0$ and $\vep_3^2 = 0$.
The Auslander-Reiten quiver of 
$H$ is shown in Figure~\ref{Fig:C3}.
The numbers in the figure correspond to composition factors and
basis vectors.
(In the last three rows the three modules on the left have to be identified
with the corresponding three modules on the right.)
\begin{figure}[!htb]
\[
\xymatrix@R=15pt{
*+[F=]{\bsm 1 \esm}
\ar[dr]
&&*+[F=]{\bsm 2 \esm}\ar[dr]\ar@{-->}[ll]&&
*+[F=]{\bsm 3\\3\\2\\1 \esm}\ar@{-->}[ll]\ar[dr]
\\
&*+[F=]{\bsm 2\\1 \esm}\ar[ur]\ar[dr]&&
*+[F=]{\bsm &3\\2&&3\\&&2\\&&1 \esm}\ar[ur]\ar[dr]\ar@{-->}[ll]&&
*+[F=]{\bsm 3\\3\\2 \esm}\ar@{-->}[ll]\ar[dr]
\\
&&*+[F=]{\bsm &3\\2&&3\\1&&2\\&&1 \esm}\ar[dr]\ar[ur]&&
*+[F=]{\bsm &3\\2&&3\\&&2 \esm}\ar@{-->}[ll]\ar[dr]\ar[ur]&&
*+[F=]{\bsm 3\\3 \esm}\ar@{-->}[ll]\ar[dr]
\\
&*+[F--]{\bsm 3\\2\\1 \esm}\ar[ur]\ar[dr]&&
*+[F]{\bsm &3\\2&&3\\1&&2 \esm}\ar[dr]\ar[ur]\ar@{-->}[ll]&&
*+[F]{\bsm &3\\2&&3 \esm}\ar[ur]\ar[dr]\ar@{-->}[ll]&&
{\bsm 3 \esm}\ar@{-->}[ll]
\\
&&{\bsm 3\\2 \esm}\ar[dr]\ar[ur]&&
*+[F]{\bsm &3\\2&&3\\1 \esm}\ar@{-->}[ll]\ar[dr]\ar[ur]&&
{\bsm 3\\2 \esm}\ar[ur]\ar@{-->}[ll]
\\
&&&{\bsm 3 \esm}\ar[ur]&&
*+[F--]{\bsm 3\\2\\1 \esm}\ar[ur]\ar@{-->}[ll]
}
\] 
\caption{
The Auslander-Reiten quiver of $H(C,D,\Omega)$ of type $C_3$ with
$D$ minimal.}
\label{Fig:C3}
\end{figure}
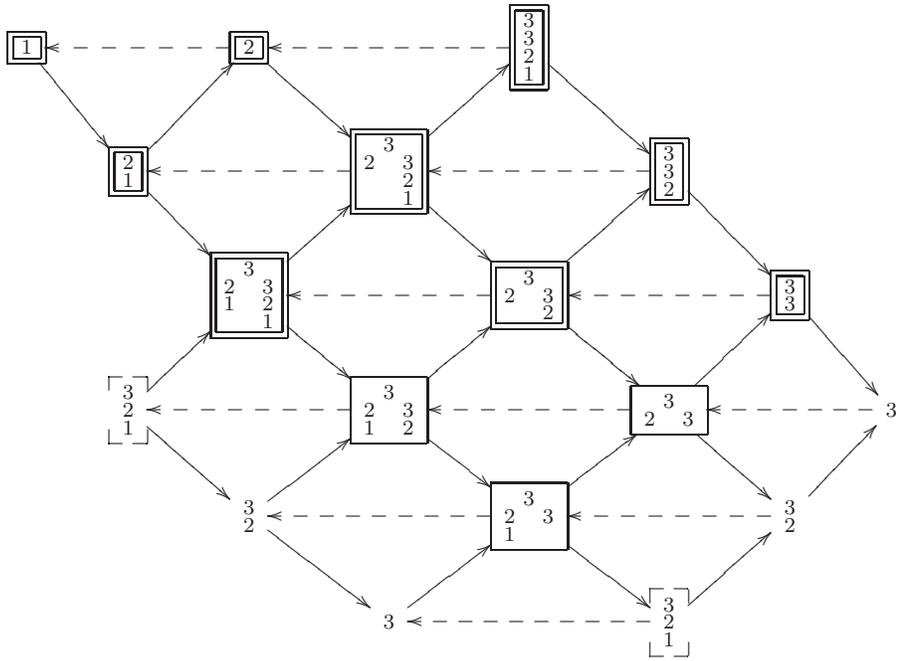

The preprojective algebra $\Pi = \Pi(C,D)$ is given by the quiver
\[
\xymatrix{
1 \ar@(ul,ur)^{\vep_1}\ar@<0.4ex>[r]^{\alpha_{21}}& 2\ar@(ul,ur)^{\vep_2} \ar@<0.4ex>[l]^{\alpha_{12}} \ar@<0.4ex>[r]^{\alpha_{32}}& 3 \ar@<0.4ex>[l]^{\alpha_{23}}
\ar@(ul,ur)^{\vep_3}
}
\]
with relations $\vep_1 = \vep_2 = 0$, $\vep_3^2 = 0$, 
$\alpha_{12}\alpha_{21} = 0$,
$-\alpha_{21}\alpha_{12} + \alpha_{23}\alpha_{32} = 0$ and
$-\alpha_{32}\alpha_{23}\vep_3 - \vep_3\alpha_{32}\alpha_{23} = 0$.
The indecomposable projective $\Pi$-modules are shown in
Figure~\ref{Fig:projPC3}.
\begin{figure}[!htb]
\[
\xymatrix@-0.7pc{
1 \ar[d]^{\alpha_{21}} &&& 
2 \ar[dl]_{\alpha_{12}}\ar[dr]^{\alpha_{32}} &&&&&&& 
3 \ar[dl]_{\alpha_{23}}\ar[dr]^{\vep_3}
\\
2 \ar[d]^{\alpha_{32}} && 
1 \ar[dr]_{\alpha_{21}}&& 3 \ar[dl]_{\alpha_{23}}\ar[dr]^{\vep_3} &&&&& 
2 \ar[dl]_{\alpha_{12}}\ar[dr]_{\alpha_{32}} && 3 \ar[dr]^{\alpha_{23}}
\\
3 \ar[d]^{\vep_3} &&& 
2 \ar[dr]_{\alpha_{32}} && 3 \ar[dr]^{\alpha_{23}} &&& 
1 \ar[dr]_{\alpha_{21}} && 3 \ar[dl]_{\alpha_{23}}\ar[dr]_{\vep_3} && 2 \ar[dl]_{\alpha_{32}}\ar[dr]^{\alpha_{12}}
\\
3 \ar[d]^{\alpha_{23}} &&&&
3 \ar[dr]_{\vep_3} && 2 \ar[dl]_{\alpha_{32}}\ar[dr]^{\alpha_{12}} &&&
2 \ar[dr]_{\alpha_{32}} && 3 \ar[dr]_{\alpha_{23}} && 1 \ar[dl]_{\alpha_{21}}
\\
2 \ar[d]^{\alpha_{12}} &&&&&
3 \ar[dr]_{\alpha_{23}} && 1 \ar[dl]_{\alpha_{21}} &&&
3 \ar[dr]_{\vep_3} && 2 \ar[dl]_{\alpha_{32}}
\\
1 &&&&&&
2 &&&&&
3
}
\]
\caption{
The indecomposable projective $\Pi(C,D)$-modules for type $C_3$ with
$D$ minimal.}
\label{Fig:projPC3}
\end{figure}
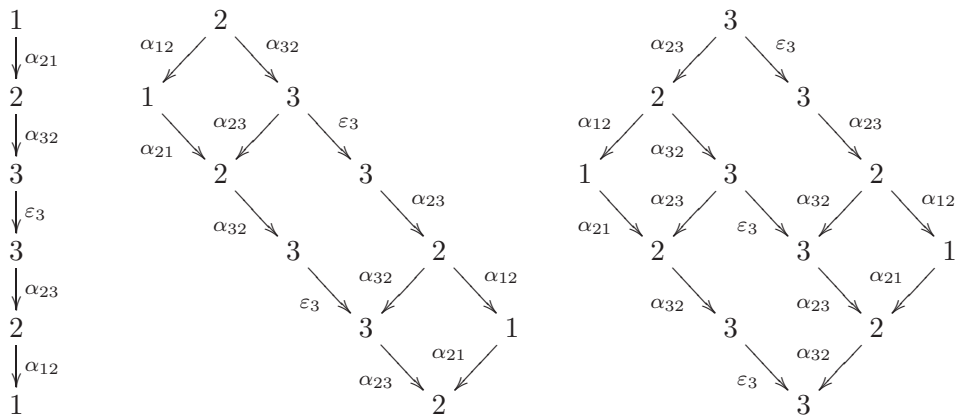

\subsection{Dynkin type $G_2$}
Let 
\[
C = \left(\bbm 2&-3\\-1&2 \ebm\right)
\]
with symmetrizer $D = \diag(1,3)$ and $\Omega = \{ (1,2) \}$.
The graph $\GG(C)$ looks as follows:
\[
\xymatrix{
1 \ar@{-}[r]^{(1,3)} & 2
}
\]
Thus $C$ is a Cartan matrix of Dynkin type $G_2$.
We have $f_{12} = 3$ and $f_{21} = 1$.
Thus $H = H(C,D,\Omega)$ is given by the quiver
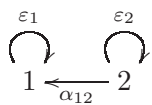
\[
\xymatrix{
1 \ar@(ul,ur)^{\vep_1}& 2 \ar[l]^{\alpha_{12}}\ar@(ul,ur)^{\vep_2}
}
\]
with relations $\vep_1 = 0$ and $\vep_2^3 = 0$.
The Auslander-Reiten quiver of 
$H$ is displayed in Figure~\ref{Fig:G2}.
As vertices we have the graded dimension vectors (arising from the obvious 
$\Z$-covering of $H$) of the indecomposable 
$H$-modules. 
(The three modules in the left most column have to be identified with the
three modules in the right most column.)
\begin{figure}[!htb]
\[
\xymatrix@-1.3pc@R=1pt{
&& *+[F=]{\bsm 1&0\\0&0\\0&0\\0&0 \esm} \ar[ddrr] &&&& 
*+[F=]{\bsm 0&1\\1&1\\1&1\\0&0 \esm}\ar[ddrr]\ar@{-->}[llll] &&&&
*+[F=]{\bsm 0&0\\0&1\\0&1\\1&1 \esm} \ar[ddrr]\ar@{-->}[llll]
\\
&&&&
\\
&&&& *+[F=]{\bsm 1&1\\1&1\\1&1\\0&0 \esm} \ar[uurr]\ar[ddrr]  &&&& 
*+[F=]{\bsm 0&1\\1&2\\1&2\\1&1 \esm}\ar[uurr]\ar[ddrr]\ar@{-->}[llll] &&&&
*+[F=]{\bsm 0&0\\0&1\\0&1\\0&1 \esm} \ar[ddrr]\ar@{-->}[llll]
\\
&&&&
\\
&& *+[F--]{\bsm 0&0\\1&1\\1&1\\0&0 \esm} \ar[uurr]\ar[ddrr]\ar@{--}[ll] &&&& 
*+[F]{\bsm 1&1\\1&2\\1&2\\1&1 \esm}\ar[uurr]\ar[ddrr]\ar@{-->}[llll] &&&&
*+[F]{\bsm 0&1\\1&2\\1&2\\0&1 \esm} \ar[uurr]\ar[ddrr]\ar@{-->}[llll] &&&&
{\bsm 0&0\\0&1\\0&1\\0&0 \esm} \ar[ddrr]\ar@{-->}[llll]
&&\ar@{-->}[ll]
\\
&&&&
\\
{\bsm 0&0\\1&1\\1&2\\0&1 \esm}\ar[uurr]\ar[ddrr] &&&& 
{\bsm 0&0\\1&2\\1&2\\1&1 \esm}\ar[uurr]\ar[ddrr]\ar@{-->}[llll] &&&&
*+[F]{\bsm 1&1\\1&2\\1&2\\0&1 \esm}\ar[uurr]\ar[ddrr]\ar@{-->}[llll]  &&&&
{\bsm 0&1\\1&2\\1&2\\0&0 \esm}\ar[uurr]\ar[ddrr] \ar@{-->}[llll] &&&&
{\bsm 1&1\\1&2\\0&1\\0&0 \esm}\ar@{-->}[llll]
\\
{\bsm 0&0\\0&1\\1&2\\1&1 \esm} \ar[drr] &&&&
*+[F]{\bsm 0&0\\1&1\\0&1\\0&1 \esm} \ar[drr]\ar@{-->}[llll] &&&&
{\bsm 0&0\\0&1\\1&1\\0&0 \esm} \ar[drr]\ar@{-->}[llll] &&&&
*+[F]{\bsm 1&1\\1&1\\0&1\\0&0 \esm} \ar[drr]\ar@{-->}[llll] &&&&
{\bsm 0&1\\1&2\\1&1\\0&0 \esm} \ar@{-->}[llll]
\\
&&
{\bsm 0&0\\1&2\\1&3\\1&2 \esm} \ar[ddrr]\ar[uurr]\ar[urr]\ar@{--}[ll] &&&&
{\bsm 0&0\\1&2\\1&2\\0&1 \esm} \ar[ddrr]\ar[uurr]\ar[urr]\ar@{-->}[llll] &&&&
{\bsm 1&1\\1&2\\1&2\\0&0 \esm} \ar[ddrr]\ar[uurr]\ar[urr]\ar@{-->}[llll] &&&&
{\bsm 1&2\\2&3\\1&2\\0&0 \esm} \ar[ddrr]\ar[uurr]\ar[urr]\ar@{-->}[llll]
&&\ar@{-->}[ll]
\\
&&&&
\\
*+[F]{\bsm 0&0\\1&2\\1&2\\1&2 \esm} \ar[ddrr]\ar[uurr] &&&&
{\bsm 0&0\\0&1\\1&2\\0&1 \esm} \ar[ddrr]\ar[uurr]\ar@{-->}[llll]  &&&&
{\bsm 0&0\\1&1\\0&1\\0&0 \esm} \ar[ddrr]\ar[uurr]\ar@{-->}[llll]  &&&&
{\bsm 1&1\\1&2\\1&1\\0&0 \esm} \ar[ddrr]\ar[uurr]\ar@{-->}[llll]  &&&&
*+[F]{\bsm 1&2\\1&2\\1&2\\0&0 \esm}\ar@{-->}[llll]
\\
&&&&
\\
&& *+[F]{\bsm 0&0\\0&1\\1&1\\0&1 \esm} \ar[uurr]\ar@{--}[ll] &&&& 
{\bsm 0&0\\0&0\\0&1\\0&0 \esm}\ar[uurr]\ar@{-->}[llll] &&&&
*+[F--]{\bsm 0&0\\1&1\\0&0\\0&0 \esm} \ar[uurr]\ar@{-->}[llll]&&&&
*+[F]{\bsm 1&1\\0&1\\1&1\\0&0 \esm} \ar[uurr]\ar@{-->}[llll]&&\ar@{-->}[ll]
}
\]
\caption{
The Auslander-Reiten quiver of $H(C,D,\Omega)$ of type $G_2$ with
$D$ minimal.}
\label{Fig:G2}
\end{figure}
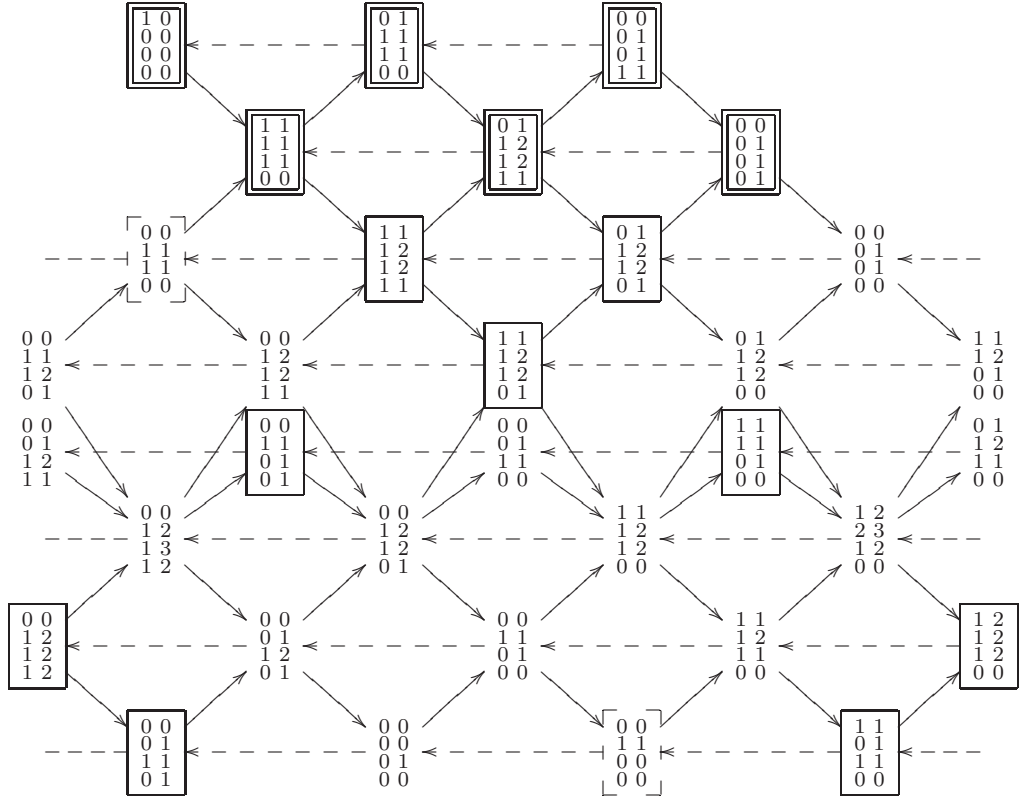

\bigskip
{\parindent0cm \bf Acknowledgements.}\,
We thank the CIRM (Luminy) for two weeks of hospitality in 
July 2013, where this work was initiated.
The first author acknowledges financial support from UNAM-PAPIIT grant 
IN108114.
The third author thanks the SFB/Transregio TR 45 for 
financial support, and the UNAM for one month of hospitality in
March 2014.
We thank W. Crawley-Boevey, H. Lenzing and C.M. Ringel for helpful comments.


\end{document}